\documentclass{amsart}
\usepackage{amsfonts,amsthm,amsmath}
\usepackage{amssymb}
\usepackage[utf8]{inputenc}
\usepackage[a4paper]{geometry}
\usepackage{enumerate}
\usepackage[usenames,dvipsnames]{color}
\usepackage{tikz}
\DeclareMathOperator{\Coeff}{Coeff} 
\DeclareMathOperator*{\Res}{Res}

\theoremstyle{plain}
	\newtheorem{thm}{Theorem}
	
	\newtheorem{lemma}[thm]{Lemma}
	\newtheorem{prop}[thm]{Proposition}
	\newtheorem{cor}[thm]{Corollary}
\theoremstyle{remark}
	\newtheorem{remark}[thm]{Remark}
	\newtheorem*{remark*}{Remark}

\theoremstyle{definition}
		
\title{
On Poincar\'e series associated with links of normal surface singularities}

\author{Tam\'as L\'aszl\'o}
\address{Alfr\'ed R\'enyi Institute of Mathematics, Hungarian Academy of Sciences \\ 
1053 Budapest, Re\'altanoda u. 13-15,  Hungary.}
\email{laszlo.tamas@renyi.mta.hu}
\thanks{T.L. is partially supported by OTKA Grants 100796 and K112735.}

\author{Zsolt Szil\'agyi}
\address{Alfr\'ed R\'enyi Institute of Mathematics, Hungarian Academy of Sciences \\ 
1053 Budapest, Re\'altanoda u. 13-15,  Hungary.}
\email{szilagyi.zsolt@renyi.mta.hu}
\thanks{Zs.Sz. is supported by the `Lendület' program.}

\begin{document}

\pagestyle{myheadings} \markboth{{\normalsize T. L\'aszl\'o and Zs. Szil\'agyi}} {{\normalsize 
On Poincar\'e series associated with links of normal surface singularities
}}

\begin{abstract}
Assume that $M$  is a rational homology sphere
plumbed 3--manifold associated with a connected negative definite graph $\mathcal{T}$.
We consider the topological Poincar\'e series (\cite{NPS}) associated with $\mathcal{T}$ and its counting function, which expresses topological information, e.g. the Seiberg--Witten invariant of $M$ (\cite{NJEMS}).

In this article, we study the counting function interpreting as an alternating sum of coefficient functions associated with some Taylor 
expansions. This is motivated by a theorem of Szenes and Vergne \cite{SzV}, 
which expresses these functions in terms of Jeffrey--Kirwan residues.

Using this method, we prove several results for the counting function: uniqueness of the quasipolynomiality inside a special cone associated 
with the topology of $M$, a structure theorem which gives a formula in terms of only one- and two-variable counting 
functions indexed by the edges and the vertices of the graph. Moreover, this helps us to construct the multivariable polynomial part 
of the Poincar\'e series, which leads to a polynomial generalization of the Seiberg--Witten invariant of $M$.
Finally, we reprove and discuss surgery formulas for the counting function (\cite{NJEMS}) and for 
the Seiberg--Witten invariant of $M$ (\cite{BN}) using this framework.

\end{abstract}

\maketitle

\setcounter{tocdepth}{1}
\tableofcontents

\section{Introduction}

\subsection{} Let $M$ be a closed oriented plumbed $3$--manifold associated with a connected negative definite plumbing graph 
$\mathcal{T}$ with vertices $\mathcal{V}$. $M$ can be realized as the link of a complex normal surface singularity $(X,0)$, where 
$\mathcal{T}$ is a good dual resolution graph and it 
determines completely 
the topology of the surface $X$ at the singular point. Naturally, it has become an active research area to understand what can we 
say about (invariants of) the analytic type of $(X,0)$ from the topological invariants of $M$ (cf. Artin--Laufer program, see \cite{Nfive}, \cite{LPhd}).

A strategy to attack this question would be to study the topological analoques of the analytic invariants and develop their 
connections. 
One can think about the \emph{Seiberg--Witten invariant conjecture} of N\'emethi and Nicolaescu \cite{NN1} as an example, which targets 
the relations of the geometric genus of the possible analytic types with the Seiberg--Witten invariant of the given 
manifold $M$. 

More generally, we consider the theory of Hilbert--Poincar\'e series 
associated with the analytic and topological types of the singularity $(X,0)$. Series of articles of 
Campillo, Delgado and Gusein-Zade (cf. \cite{CDGPs,CDGEq}), and N\'emethi (cf. \cite{NPS,NCL}), studied the Poincar\'e series $P(\mathbf{t})$ 
associated with a divisorial multi-index filtration of functions on $(X,0)$. For `nice' analytic structures (e.g. rational singularities \cite{CDGPs}, splice-quotient singularities \cite{NCL}) $P(\mathbf{t})$ is equal with a topological series $Z(\mathbf{t})$, which can be 
expressed from the combinatorics of the resolution graph. Although, in general this is not the case (we refer to \cite{NPS} for examples and obstruction), it makes a 
bridge between the geometry and topology of the singularities, and one can think of $Z(\mathbf{t})$ as the 
topological counterpart of the `analytic' Poincar\'e series $P(\mathbf{t})$ associated with $(X,0)$. 
In the sequel, we provide some details on the definition and the topological information given by $Z(\mathbf{t})$.

\subsection{} We assume that $M$ is a rational homology sphere, or, equivalently $\mathcal{T}$ is a tree and all genus 
decorations in the plumbing are zero.
Let $\widetilde{X}$ be the plumbed 4--manifold associated with
$\mathcal{T}$, hence $\partial \widetilde{X} = M$. Its second
homology $L:=H_2(\widetilde{X},\mathbb{Z})$ is a lattice, freely generated by the classes of 2--spheres $\{E_v\}_{v\in\mathcal{V}}$, with a negative definite intersection form  $I=(\,,\,)$. The dual lattice is $L':=H^2(\widetilde{X},\mathbb{Z})$ generated 
by the (anti)dual classes $\{E^*_v\}_{v\in\mathcal{V}}$. The intersection form embeds $L$ into $L'$ and it extends to $L\otimes\mathbb{Q}\supset L'$. We use the notation $x^2:=(x,x)$ for any $x\in L'$. Then one has $H:=H_1(M,\mathbb{Z})\simeq L'/L$ with $|H|=\det(\mathcal{T}):=\det(-I)$ and denote the class of $x$ in $H$ by $[x]$. 

Let $K\in L'$ be the canonical class,
$\widetilde{\sigma}_{can}$ the canonical
$spin^c$--structure on $\widetilde{X}$ with $c_1(\widetilde{\sigma}_{can})=-K$,
and $\sigma_{can}\in \mathrm{Spin}^c(M)$ its restriction on $M$. Then  $\mathrm{Spin}^c(M)$ is an $H$--torsor, with action denoted by $*$. We denote by $\mathfrak{sw}_{\sigma}(M)$ the 
Seiberg--Witten invariant of $M$ indexed by the $spin^c$--structure $\sigma$ of $M$.

We define the  \emph{multivariable topological Poincar\'e series} as the 
Taylor expansion $Z(\mathbf{t})=\sum_{l'} p_{l'}\mathbf{t}^{l'}$ at the  origin of the rational function
\begin{equation*}
f(\mathbf{t})=\prod_{v\in \mathcal{V}} (1-\mathbf{t}^{E^*_v})^{\delta_v-2},
\end{equation*}
where we write
$\mathbf{t}^{l'}=\prod_{v\in \mathcal{V}}t_v^{l_v}$  for any $l'=\sum _{v\in \mathcal{V}}l_vE_v\in L'$ and $\delta_v$ is the valency of the vertex $v$.
 Notice that $Z(\mathbf{t})\in \mathbb{Z}[[L']]$, the submodule of $\mathbb{Z}[[\mathbf{t}^{\pm 1/|H|}]]$ in variables $\{t_v^{\pm 1/|H|}\}_v$.
It has a natural decomposition $Z(\mathbf{t})=\sum_{h\in H}Z_h(\mathbf{t})$, where $Z_h(\mathbf{t})=\sum_{[l']=h}p_{l'}\mathbf{t}^{l'}$. 
We can also introduce the \emph{equivariant} function 
$f_{H}(\mathbf{t})=\prod_{v\in \mathcal{V}} (1-[E^*_v]\mathbf{t}^{E^*_v})^{\delta_v-2}$ with coefficents in the group ring $\mathbb{Z}[H]$, which decomposes into $f_{H}(\mathbf{t})=\sum_{h\in H}f_h(\mathbf{t})\cdot h$ and the Taylor expansion of $f_h(\mathbf{t})$ is 
$Z_h(\mathbf{t})$ (see \ref{sec:coeff}).

For any $h\in H$ we associate a counting function 
$$Q_{h}(x)=\sum_{l'\ngeq x,\, [l']=h}p_{l'}$$
with the coefficients of $Z_{h}(\mathbf{t})$, where `$\geq$' is the partial ordering  (\ref{ss:order}) on $L'$. 
It is a finite sum since $Z(\mathbf{t})$ is supported on the \emph{Lipman cone} $\mathcal{S}':=\mathbb{Z}_{\geq0}\langle E^{*}_{v}\rangle$ (cf. (\ref{eq:finite})).

\subsection{}The first main result on the topological model is a theorem of N\'emethi \cite{NJEMS}, which shows that 
the Seiberg--Witten invariant can be expressed from 
the counting function. More precisely, if $x\in -K + \textnormal{int} (\mathcal{S}')$ then
\begin{equation}\label{eq:INTR1}
Q_{[x]}(x)=
-\frac{(K+2x)^2+|\mathcal{V}|}{8}-\mathfrak{sw}_{[-x]*\sigma_{can}}(M).\end{equation}

Notice that to guarantee the polynomiality (right hand side) of the counting function,
$x$ should sit in the affine subcone given above.
This behaviour is understood precisely in the 
work of L\'aszl\'o and N\'emethi \cite{LN}, where they 
develop an equivariant mutivariable Ehrhart theory which explains the (quasi)polynomiality of the 
counting function and it leads to the definition of the \emph{periodic constant} of multivariable series (one-variable case defined by N\'emethi and Okuma \cite{NOk,Ok}). 

By the general theory, the Lipman cone in principle is divided into different (conical) chambers, where the counting function can be realized by 
different quasipolynomials, in particular, they give different periodic constants associated with the chambers. 
However, (\ref{eq:INTR1}) `suggests' that all the quasipolynomials appearing in the Lipman cone  coincide. In particular, \cite[Corollary 5.2.1]{LN} concludes that there is a unique periodic constant associated with the Lipman cone which is the constant term of the 
right hand side of (\ref{eq:INTR1}), i.e. it equals with the Seiberg--Witten invariant. The coincidence phenomenon of the quasipolynomials is 
somewhat hidden in the structure of the function $f(\mathbf{t})$. 

\subsection{}The main motivation of the present article is to propose a `different' interpretation of the counting function, which shows the uniqueness of the 
quasipolynomiality inside the Lipman cone. 
In Section \ref{sec:coeff}, we represent the counting function in the following form  
\begin{equation}\label{eq:q1} 
Q_{h}(x) = 
\sum_{\emptyset \neq \mathcal{I} \subseteq \mathcal{V}} (-1)^{|\mathcal{I}|-1} 
\Coeff \Big( \textnormal{T}\bigg[ \mathbf{t}_{\mathcal{I}}^{-r_{h-[x]}} f_h(\mathbf{t}_{\mathcal{I}})\cdot 
\prod_{i\in \mathcal{I}} \frac{\mathbf{t}_{\mathcal{I}}^{E_{i}}}{1 - \mathbf{t}_{\mathcal{I}}^{E_{i}}}\bigg],\,
\mathbf{t}_{\mathcal{I}}^{x} \Big),
\end{equation}
where $r_{h-[x]} \in L'\cap \sum_{v\in \mathcal{V}}[0,1)E_{v}$ is the unique lift of $h-[x]$, and  
$\Coeff( \textnormal{T}[R(\mathbf{t}_{\mathcal{I}})],\mathbf{t}_{\mathcal{I}}^{x})$ is the coefficient of $\mathbf{t}_{\mathcal{I}}^{x}$ 
in the Taylor expansion of the rational function $R$ with variable $\mathbf{t}_{\mathcal{I}}^{l'}:=\prod_{v\in\mathcal{I}}t_v^{l_v'}$ 
for any $l'=\sum_{v\in \mathcal{V}}l_v' E_v \in L'$. 

Using this interpretation, we connect the counting function with Jeffrey--Kirwan residues (see Section \ref{sec:SzV}) 
by a theorem of Szenes and Vergne \cite{SzV}. This result concludes that the coefficient functions from (\ref{eq:q1}) are equal with 
sums of Jeffrey--Kirwan residues and they are quasipolynomials inside the chambers of the Lipman cone. Again, in principle they can be different, but the special structures of the `projections' $f_h(\mathbf{t}_{\mathcal{I}})$ for 
$\mathcal{I}\subseteq \mathcal{V}$ (Lemma \ref{Lm-5}) forces these quasipolynomials to coincide. In Section \ref{ss:unique} we present 
the proof of this uniqueness of the quasipolynomials.
In particular, this is a quadratic polynomial on a 
sufficiently sparse sublattice of $L'$, which confirms the form of Equation (\ref{eq:INTR1}).

\subsection{}Several expressions for the interpretation and computation of the Seiberg--Witten invariant were established in the last years: surgery formula \cite{BN}, Turaev torsion normalized by Casson--Walker invariants \cite{Nic04}, etc. (see Section \ref{ss:sw} for more details). 
We emphasize that the connection with Jeffrey--Kirwan residues gives an explicit formula for the computation of the quasipolynomials, in particular, for the Seiberg--Witten invariant. Nevertheless, the explicit computation 
can be complicated in general (see Section \ref{sec:1dim} for the one dimensional case). 

Another interesting combinatorial interpretation is proposed in \cite{BN} and \cite{LN}. Notice that the number of variables 
of $f_h(\mathbf{t})$ is the number of vertices of the graph, which may increase considerably for `simple' (e.g. lens spaces, Seifert 3-manifolds) cases as well. However, \cite[Theorem 5.4.2]{LN} proved that from Seiberg--Witten invariant point of view, the number of variables can be reduced to the number of nodes of the graph. Hence, we look at the \emph{reduced} function $f_h(\mathbf{t}_{\mathcal{N}}):=f_h(\mathbf{t})\mid_{t_v=1,v\notin \mathcal{N}}$ for any $h\in H$. For graphs with at most two nodes, \cite{BN} and \cite{LN} proved a decomposition 
$$f_h(\mathbf{t}_{\mathcal{N}})=P_h(\mathbf{t}_{\mathcal{N}})+ f^{-}_h(\mathbf{t}_{\mathcal{N}}),$$
where $P_h(\mathbf{t}_{\mathcal{N}})$ is a Laurent polynomial and $P_h(\mathbf{1})$ gives the Seiberg--Witten invariant associated with $h\in H$. This polynomial is constructed by a division algorithm controlled by the vanishing of the periodic constant. 

In the present article we use the aforementioned interpretation to prove the existence and uniqueness of this polynomial for graphs with any number of nodes. First of all, the result of Section \ref{s:strcf} gives a formula for the counting function in terms of one- and two-variable counting functions associated with vertices and edges of $\mathcal{T}$. This can be done for the reduced counting function as well (cf. Section \ref{s:strcforb}). This is the main idea of the algorithm giving the polynomial in general, since it reduces to one- and two-variable cases.  In Section \ref{s:polpartsec} we reprove the one- and two-variable cases in a slightly more general setting than in \cite{BN} and \cite{LN}, and we deduce the multivariable case in general. Finally, Section \ref{sec:example} illustrates the algorithm with an example of a graph having three nodes. 

We would like to emphasize that this result has twofold importance: on the one hand, it defines a `\emph{polynomial invariant}' which generalizes the Seiberg--Witten invariant for negative definite plumbed 3-manifolds. On the other hand, it gives an explicit algorithm to compute it in terms of the graph $\mathcal{T}$.

\subsection{}In the last part, we are interested to discuss surgery formulas appearing in this context, using the new approach given by 
the present article. 
In \cite[Theorem 3.2.13]{NJEMS} one can find a recursion formula for the quasipolynomial associated with the counting function. 
The essence of Section \ref{sec:surgform} is to show how to use the methods of Section \ref{sec:resunique} in order to deduce the 
recursion, first for the counting function, then for the quasipolynomial associated with it. In particular, 
we discuss how the recursion behaves on the level of periodic constant and we compare with the Braun--N\'emethi \cite{BN} surgery formula for the 
Seiberg--Witten invariant. The present proof of the recurrence is based on a decomposition of the sum from $(\ref{eq:q1})$ in terms of the graph $\mathcal{T}$, 
and on a projection property (Lemma \ref{Lm-Pr}) of the coefficient functions.

\subsection{Acknowledgement}
The authors are grateful to Andr\'as N\'emethi for the motivating conversations. The second author is thankful for the hospitality of the Alfr\'ed R\'enyi Institute of Mathematics, Hungarian Academy of Sciences.

\section{Preliminaries}

\subsection{Links of normal surface singularities}\label{ss:link}
\subsubsection{}
We consider a connected negative definite plumbing graph $\mathcal{T}$ with vertices $\mathcal{V}$. By plumbing disk bundles along $\mathcal{T}$, we obtain a smooth 4--manifold $\widetilde{X}$ whose boundary is an oriented 
plumbed 3--manifold $M$. The graph $\mathcal{T}$ can be realized as the dual graph of a good resolution $\pi:\widetilde{X}\to X$ of some complex normal surface singularity $(X,0)$ 
and $M$ is called the link of the singularity. In this article we assume that $M$ is a {\it rational homology sphere} which is equivalent to the fact 
that $\mathcal{T}$ is a tree and all the genus decorations are zero.

Then $L:=H_2 (\widetilde{X},\mathbb{Z} )$ is a lattice, freely
generated by the classes of the irreducible exceptional curves
$\{E_v\}_{v\in\mathcal{V}}$ (or the cores of the plumbing construction), with 
a nondegenerate negative definite intersection form $I:=[(E_v,E_w)]_{v,w}$. 

If  $L'$ denotes
$H^2( \widetilde{X},\mathbb{Z})\simeq Hom(L,\mathbb{Z})$, then the intersection form provides an embedding $L \hookrightarrow L'$ with finite factor $H:=L'/L \simeq H^2(\partial \widetilde{X},\mathbb{Z})\simeq H_{1}(M, \mathbb{Z})$ and it extends to
$L'$ (since $L'\subset L\otimes \mathbb{Q}$).
Hence $L'$ is the dual lattice, freely generated by the (anti)duals $\{E_v^*\}_{v\in \mathcal{V}}$, where we
prefer the sign convention $( E_v^*, E_w) =-\delta_{vw}$ (the negative of the Kronecker symbol).  The class of $l'\in L'$ in $H$ is denoted by $[l']$. One can identify $H$ with its Pontrjagin dual $\widehat H$ by the isomorphism $[l']\mapsto e^{2\pi i (l',\cdot)}$.

We use notation $A$ for the positive definite matrix $-I$. In particular, $A_{vv} = -(E_{v},E_{v})>0$ for all $v\in \mathcal{V}$, and for $v\neq w$ we have $A_{vw}=-1$ if $\overline{vw}$ is an edge of $\mathcal{T}$ (since $\pi:\widetilde{X} \to X$ is a good resolution), otherwise $A_{vw}=0$. Having this special shape, we refer to $A$ as positive definite matrix associated with $\mathcal{T}$.
Moreover, $(A^{-1})_{vw} = -(E^{*}_{v}, E^{*}_{w})$ are the entries of the vectors $E^{*}_{v}$ in the basis $\{E_{v}\}_{v\in \mathcal{V}}$, and all of them are positive (cf. \cite[page 83 and \textsection 20]{EN}). We emphasize the relation 
\begin{equation}\label{eq:rel}
[E^*_v]_{v\in\mathcal{V}}\cdot A=[E_v]_{v\in\mathcal{V}},
\end{equation}
which will be used frequently in the sequel.


\subsubsection{}\label{ss:order}
We have the following partial ordering on $L\otimes \mathbb{Q}$: for any $l_1,l_2$ one writes $l_1\geq l_2$ if
$l_1-l_2=\sum_{v\in \mathcal{V}} \ell_v E_v$ with all $\ell_v\geq 0$. Denote by $\mathcal{S}'$  
the Lipman cone $\{l'\in L'\,:\, (l',E_v)\leq 0 \ \mbox{for all
$v$}\}$ which is generated over $\mathbb{Z}_{\geq 0}$ by the
elements $E_v^*$. We write $\mathcal{S}'_{\mathbb{R}} := \mathcal{S}'\otimes \mathbb{R}$ for the real Lipman cone.
Since  all the entries of $E_v^*$ are strictly positive, for any fixed $x\in L'$ the set 
\begin{equation}\label{eq:finite}
\{l'\in \mathcal{S}'\,:\, l'\ngeq x\} \ \ \mbox{is finite}.
\end{equation}

\subsubsection{}\label{ss:distreps}
For any class $h\in H=L'/L$ we consider two \emph{distinguished representatives}:
the unique element $r_h\in L'$ characterized by $r_h\in \sum_{v}[0,1)E_v$ and $[r_h]=h$, and the unique minimal element $s_h$ of 
$\{l'\in L'\,:\,[l']=h\}\cap \mathcal{S}'$ (cf. \cite[5.4]{NOSZ}). One has  $s_h\geq r_h$, and 
in general $s_h\neq r_h$ (for an example see \cite[4.5]{trieste}).
Both representatives appear from different perspectives in the `normalization' term of the Seiberg--Witten invariants (e.g. \cite{BN} and \cite{NOSZ}).
In Section \ref{ss:pcrec}, we will discuss their behaviour under certain projections. 

\subsubsection{}
For any $\mathcal{J} \subseteq \mathcal{V}$ let $\delta_{v,\mathcal{J}}$ be the valency of the vertex $v$ in the complete subgraph $\mathcal{T}_{\mathcal{J}}$  of $\mathcal{T}$ with vertices $\mathcal{J}$. We use shorter notation for $\delta_{v} := \delta_{v,\mathcal{V}}$. We distinguish the following subsets of vertices: the set of \emph{nodes} $\mathcal{N}:=\{v\in \mathcal{V}:\delta_v \geq 3\}$
and the set of \emph{ends} $\mathcal{E}=\{v\in \mathcal{V}:\delta_v= 1\}$.

We consider the \emph{orbifold graph} $\mathcal{T}^{orb}$ constructed from $\mathcal{T}$ as follows: the vertices are the nodes of $\mathcal{T}$ and two vertices are connected by an edge if the corresponding nodes in $\mathcal{T}$ are connected by a path which, apart from the two nodes, consists only vertices with valency 2. One can decorate the vertices and edges of $\mathcal{T}^{orb}$ and define a nonintegral positive definite matrix $A^{orb}$ by \cite[Lemma 4.14]{BNnewt}. Since we are interested rather in the shape of the graph, we omit the precise definitions here. For more information one can consult with e.g. \cite{BNnewt,LN}.

\vspace{0.2cm}
For more details on the links and resolution graphs of normal surface singularities see e.g. \cite{Nfive,NOSZ,trieste}.

\subsubsection{Spin$^c$--structures}
Let $\widetilde{\sigma}_{can}$ be the {\it canonical $spin^c$--structure} on $\widetilde{X}$.  Its
first Chern class $c_1( \widetilde{\sigma}_{can})=-K\in L'$, where $K$ is the canonical element in $L'$ defined by the
{\it adjunction formulas}
\begin{equation}\label{adjunction-formulae}
(K+E_v,E_v)+2=0 \ \ \ \textnormal{ for all $v\in\mathcal{V}$}
\end{equation}
(cf. \cite[p.\,415]{GS}).
The set of $spin^c$--structures $\mathrm{Spin}^c(\widetilde{X})$ of $\widetilde{X}$ is an $L'$--torsor: if we denote
the $L'$--action by $l'*\widetilde{\sigma}$, then $c_1(l'*\widetilde{\sigma})=c_1(\widetilde{\sigma})+2l'$.
Furthermore,  all the $spin^c$--structures of $M$ are obtained by restrictions from $\widetilde{X}$.
$\mathrm{Spin}^c(M)$ is an $H$--torsor, compatible with the restriction and the projection $L'\to H$.
The {\it canonical $spin^c$--structure} $\sigma_{can}$ of $M$ is the restriction 
of the canonical $spin^c$--structure $\widetilde{\sigma}_{can}$ of $\widetilde{X}$.

\subsection{Topological Poincar\'e series, Seiberg--Witten invariants and generalizations}

\subsubsection{Poincar\'e series: definitions and motivations}\label{ss:tps}
For any $l' = \sum_{v\in \mathcal{V}} l_{v}E_{v}\in L' $ we set $\mathbf{t}^{l'} = \prod_{v\in \mathcal{V}} t_{v}^{l_{v}}$ and let $\mathbb{Z}[[L']]$ be the $\mathbb{Z}[L']$-submodule of formal power series $\mathbb{Z}[[t_{v}^{\pm 1/|H|}:v\in \mathcal{V}]]$ consisting of series $\sum_{l'\in L'}a_{l'}\mathbf{t}^{l'}$, with $a_{l'}\in \mathbb{Z}$ for all $l'\in L'$. The \emph{support} of such a series is $\{l'\in L' : a_{l'}\neq 0\}\subset L'$.
We consider the rational function 
\begin{equation}\label{eq:Zdef}
f(\mathbf{t}) = \prod_{v\in \mathcal{V}}(1 - \mathbf{t}^{E_v^*})^{\delta_v-2}.
\end{equation}
Then the {\em topological Poincar\'e series} 
$Z(\mathbf{t})=\sum_{l'}p_{l'} \mathbf{t}^{l'} \in \mathbb{Z}[[L']]$ is the multivariable Taylor expansion at the origin of $f(\mathbf{t})$, and it decomposes uniquely into $Z(\mathbf{t})=\sum_{h\in H} Z_{h}(\mathbf{t})$, where $Z_{h}({\mathbf{t}})=\sum_{[l']=h} p_{l'} \mathbf{t}^{l'}$.
By (\ref{ss:order}), $Z(\mathbf{t})$ is supported in $\mathcal{S}'$, hence $Z_{h}(\mathbf{t})$ is supported in $(l'+L)\cap \mathcal{S}'$, where $l'\in L'$ with $[l']=h$.


This series was introduced by the work of N\'emethi \cite{NPS}, motivated by the following fact: we may consider the equivariant divisorial Hilbert series $\mathcal{H}(\mathbf{t})$ of
a normal surface singularity $(X,0)$ with fixed resolution graph $\mathcal{T}$.
The key point connecting $\mathcal{H}(\mathbf{t})$ with the topology of the link $M$ (or the graph $\mathcal{T}$) is introducing the series $\mathcal{P}(\mathbf{t})=
-\mathcal{H}( \mathbf{t}) \cdot \prod_{v\in \mathcal{V}}(1 - t_v^{-1})\in \mathbb{Z}[[L']]$. 
Then $Z(\mathbf{t})$ is the 
`topological candidate' for $\mathcal{P}(\mathbf{t})$. They agree for several singularities, e.g. for splice quotients (see \cite{NCL}), which
contain all the rational, minimally elliptic or weighted homogeneous singularities.
More details regarding this analytic motivation can be found in \cite{CDGPs,CDGEq,NPS,NCL}.

\subsubsection{Seiberg-Witten invariants and counting functions}\label{ss:sw}
For a closed oriented $3$-manifold $M$ one can associate the Seiberg--Witten invariant $\mathfrak{sw}(M):\mathrm{Spin}^c(M)\to \mathbb{Q}$, 
$\sigma\mapsto \mathfrak{sw}_\sigma$ (cf. \cite{Lim, Nic04}).
In the last years several combinatorial expressions were established regarding the Seiberg--Witten invariants, since it is difficult to compute using its very definition.
For rational homology spheres, one direction was opened by the result of Nicolaescu \cite{Nic04} showing that $\mathfrak{sw}(M)$ is 
equal with the Reidemeister--Turaev torsion normalized by the Casson--Walker invariant. In the case when $M$ is a negative 
definite plumbed rational homology sphere, combinatorial formula for Casson--Walker invariant in terms of the plumbing graph can be found in Lescop 
\cite{Lescop}. The Reidemeister--Turaev torsion is determined by N\'emethi and Nicolaescu \cite{NN1} using Dedekind--Fourier sums, which 
is still hard to determine in most of the cases.

In this article, we study a different direction which is a combinatorial interpretation of $\mathfrak{sw}(M)$ from the Poincar\'e series  $Z(\mathbf{t})$, using qualitative properties 
of the coefficients. For any $h\in H$ we define the {\it counting function} of the coefficients of $Z_{h}(\mathbf{t})=\sum_{[l']=h}p_{l'} 
\mathbf{t}^{l'}$ by 
$$x\mapsto Q_{h}(x):=\sum_{l'\not\geq x,\, [l']=h} \, p_{l'}.$$ 
The above sum is finite 
by (\ref{eq:finite}), moreover we have the following 
theorem.

\begin{thm}[N\'emethi \cite{NJEMS}]\label{th:JEMS} 
Fix some $x\in -K+ \textnormal{int}(\mathcal{S}')$. Then
\begin{equation}\label{eq:SUM} Q_{[x]}(x)=
-\frac{(K+2x)^2+|\mathcal{V}|}{8}-\mathfrak{sw}_{[-x]*\sigma_{can}}(M).
\end{equation}
\end{thm}

If we write $x=l+r_{h}$ with $l\in L$ and $h=[x]$ then the right hand side of (\ref{eq:SUM}) can be seen as a multivariable quadratic polynomial on $L$ with constant term
$$
-\frac{(K+2r_h)^2+|\mathcal{V}|}{8}-\mathfrak{sw}_{-h*\sigma_{can}}(M),
$$
which we call the {\em $r_{h}$-normalized Seiberg--Witten invariant} of $M$ associated with $h\in H$. 
Similarly, one can also define the {\em $s_{h}$-normalized Seiberg--Witten invariant} by writing $x=l+s_{h}$ for some $l\in L$ and $h=[x]$, which appears in different contexts. 

\begin{remark}
It is important to emphasize that the identity (\ref{eq:SUM}) was motivated by a similar analytic identity which 
expresses the geometric genus of the complex normal surface singularity $(X,0)$ from the series $P(\mathbf{t})$ (cf. \cite{NCL,Ok}).
\end{remark}

\subsubsection{Multivariable periodic constants}

The article \cite{LN} developes a theory which gives better understanding for the Equation (\ref{eq:SUM}) and shows how to recover the needed information from the series (or rational function) in order to obtain closed formula or algorithm for the Seiberg--Witten invariant, at least for special cases.

One can construct a (conical) chamber decomposition of the space $L'\otimes\mathbb{R}$. Once we fix a chamber, the counting function is represented 
by a multivariable Ehrhart type quasipolynomial (coming from counting lattice points in some special polytopes attached to $f(\mathbf{t})$) inside 
the chamber. The constant term of the quasipolynomial is understood by the notion of \emph{multivariable periodic constant} associated with $Z(\mathbf{t})$, which interprets the Seiberg--Witten invariant according to (\ref{eq:SUM}). More details about the constructions can be found in \cite[Sections 3 and 4]{LN}.

In the sequel we briefly recall the main ideas from the case of one-variable series introduced by N\'emethi and Okuma in 
\cite{NOk,Ok}, then we discuss its generalization for the multivariable case specialized to $Z(\mathbf{t})$. 

Let $S(t)=\sum_{l\geq0}c_l t^l \in 
\mathbb{Z}[[t]]$ be a formal power series with one variable and assume that for some $p\in \mathbb{Z}_{>0}$ the counting function 
$Q_p(n):=\sum_{l=0}^{pn-1}c_l$ is a polynomial in $n$. Then the constant term $Q_p(0)$ is independent of $p$ and it is called the \emph{periodic 
constant} $\mathrm{pc}(S)$ of the series $S$. The intuitive meaning of the periodic constant is shown by the following example: 
assume $S(t)$ is the Hilbert series associated with a graded
algebra/vector space $A=\oplus_{l\geq 0}A_l$ (i.e. $c_l=\dim
A_l$) and the series $S$ admits a Hilbert quasipolynomial $Q(l)$ (that is, $c_l=Q(l)$ for $l\gg 0$). 
Then the periodic constant of the `regularized series' $S_{reg}(t):=\sum_lQ(l)t^l$ is zero. Hence the periodic constant of $S(t)$ 
measures exactly the difference between $S(t)$ and $S_{reg}(t)$, that is 
$\mathrm{pc}(S)=(S-S_{reg})(1)$ collecting all the anomalies of the starting elements of $S$.

For the multivariable case we consider the settings of Sections \ref{ss:link} and \ref{ss:tps}. Assume there is a real cone 
$\mathcal{K}\subset L'\otimes\mathbb{R}$ whose affine closure has positive dimension, a 
sublattice $\widetilde L \subset L'$ of finite index and $l'_* \in \mathcal{K}$ such that $Q_h(l')$ is a quasipolynomial for any 
$\widetilde L\cap(l'_* +\mathcal{K})$. Then we say that $Z(\mathbf{t})$ admits a {\em multivariable periodic constant} 
associated with $\mathcal{K}$ and it is defined by
\begin{equation}\label{eq:PCDEF}
\mathrm{pc}^{\mathcal{K}}_h(Z)
:= 
Q_h(0) 
\qquad \textnormal{ for any $h\in H$}.
\end{equation}

\begin{remark}
 The definition does not depend on the choice of the sublattice $\widetilde L$, which corresponds to the choice of 
 $p$ in the one-variable case.
 This is responsible for the name `periodic' in the definition. We will often use notation $\mathrm{pc}^{\mathcal{K}}(Z_h)$ as well.
\end{remark}

Notice that the multivariable periodic constant exists for any chamber $\mathcal{C}$ given by the decomposition of $L'\otimes\mathbb{R}$. However, the chamber decomposition is too fine in the sense that the Lipman cone can be divided into several chambers in general. Hence, different quasipolynomials are attached, giving different periodic constants. Although, Theorem  \ref{th:JEMS} from \cite{NJEMS} suggests that they are the same. A proof of this uniqueness question, using the methods of the present article, is given in Section \ref{sec:unique}.

\subsubsection{Polynomial part of rational functions}
In the one-variable case, we assume that $f(t)$ is a rational function of the form $B(t)/A(t)$ with $A(t)=\prod_i(1-t^{a_i})$ and $a_i>0$. 
Then by \cite[7.0.2]{BN} one has a unique decomposition $f(t)=P(t)+\frac{D(t)}{A(t)}$, where $P(t)$ and $D(t)$ are polynomials and $D(t)/A(t)$ has negative degree with vanishing periodic constant.  Hence, $\mathrm{pc}(S)$ equals $P(1)$. 
This decomposition was generalized in 
\cite[Section 4.5]{LN} for two-variable functions of the form $f(\mathbf{t})=B(\mathbf{t})/\prod_{i=1}^{2}(1-\mathbf{t}^{a_i})$ with all the coordinates of $a_i$ are strict positive. E.g. if we choose the chamber $\mathcal{C}:=\mathbb{R}_{\geq 0}\langle a_1,a_2\rangle$ one has 
$$f(\mathbf{t})=P^{\mathcal{C}}(\mathbf{t})+f^{-,\mathcal{C}}(\mathbf{t})$$ 
with $\mathrm{pc}^{\mathcal{C}}(f^{-,\mathcal{C}})=0$, hence $\mathrm{pc}^{\mathcal{C}}(f)=P^{\mathcal{C}}(\mathbf{1})$. 

For more variables, the 
authors know no similar decompositions in general. Their experience is that the geometry of the (Ehrhart) quasipolynomials associated with the counting function by the theory of \cite{LN} cannot be controlled since the chamber decomposition can be cumbersome. However, the counting function of rational functions defined in (\ref{eq:Zdef}) have very special structure in terms of the graph which makes the decomposition possible. See Section \ref{s:strcf} and Section  \ref{s:polpartsec} for further motivation and details. 

\subsubsection{Reduced Poincar\'e series}\label{ss:rps}
The number of variables of $Z(\mathbf{t})$ is the number of vertices $|\mathcal{V}|$ of the graph $\mathcal{T}$ which implies that the quasipolynomials associated with the counting function have $|\mathcal{V}|$ many variables too. Therefore, the `complexity' can be different even in the case when $\mathcal{T}$ has no nodes at all, that is when $M$ is a lens space. Nevertheless, \cite[Theorem 5.4.2]{LN} says that, from Seiberg--Witten invariant 
(or periodic constant) point of view, the complexity is measured by the number of nodes, that is the number of variables can be reduced to $|\mathcal{N}|$.

For any $h\in H$, we define the \emph{reduced rational function} by 
$$f_h(\mathbf{t}_{\mathcal{N}}):=f_h(\mathbf{t})\mid_{t_v=1,v\notin \mathcal{N}}$$ and its Taylor expansion $Z_h(\mathbf{t}_{\mathcal{N}})$ is called the \emph{reduced Poincar\'e series}. Then there exists the periodic constant of $Z_h(\mathbf{t}_{\mathcal{N}})$ associated with the projected real Lipman cone  
$\pi_{\mathcal{N}}(S'_{\mathbb{R}})$ where $\pi_{\mathcal{N}}:\mathbb{R}\langle E_v\rangle_{v\in\mathcal{V}}\to \mathbb{R}\langle E_v\rangle_{v\in\mathcal{N}}$ is the projection, and 
$$\mathrm{pc}^{\pi_{\mathcal{N}}(S'_{\mathbb{R}})}(Z_h(\mathbf{t}_{\mathcal{N}}))=\mathrm{pc}^{S'_{\mathbb{R}}}(Z_h(\mathbf{t}))=
-\mathfrak{sw}_{-h*\sigma_{can}}(M)-\frac{(K+2r_h)^2+|\mathcal{V}|}{8}.$$
\begin{remark}
 Notice that before the elimination of variables, we have to decompose the series $Z(\mathbf{t})$ into $\sum_{h\in H}Z_h(\mathbf{t})$ in order to preserve the information about the $H$ invariants.
\end{remark}
 
We would like to highlight that not only it is `easier' to work with reduced series, but for special classes of singularities it can be compared with other series (or invariants) describing the analytic or topological types, see \cite{NPS} for examples. A more stronger reduction is given by lattice cohomological methods in \cite{LNRed}.

Therefore, in this article, the polynomial generalization of the Seiberg--Witten invariant will be only constructed from the reduced rational function, giving a polynomial with  $|\mathcal{N}|$ many variables.

\section{Residues and the counting function}\label{sec:resunique}

\subsection{Counting function as coefficients of Taylor expansions}\label{sec:coeff}%
We express the counting function $Q_{h}$ as alternating sum of coefficient functions of Taylor expansions. Using this presentation one can associate quasipolynomials with the counting function in a standard manner.

We also introduce the equivariant rational function 
$$
f_{H}(\mathbf{t}) = \prod_{v\in \mathcal{V}} (1 - h_{v} \mathbf{t} ^{E^{*}_{v}})^{\delta_{v}-2},
$$ 
where $h_{v}=[E^{*}_{v}]$ denotes the class of $E^{*}_{v}$ in $H=L'/L$.
It is the equivariant version of $f(\mathbf{t})$ with coefficients in the group ring $\mathbb{Z}[H]$. It can be decomposed into equivariant parts $f_{H}(\mathbf{t}) = \sum_{h\in H} f_{h}(\mathbf{t}) \cdot h$ with $f_{h}(\mathbf{t})$ rational of form $\frac{\sum_{l'\in h+L} b_{l'}\mathbf{t}^{l'}}{\prod_{u\in \mathcal{E}}1-\mathbf{t}^{|H|E^{*}_{u}}}$ (with $l'\in \mathcal{S}'$ when $b_{l'}\neq 0$) using expansion of type 
\begin{equation}\label{Eq-49}
\frac{1}{1-h \mathbf{t}^{x}} = \frac{\sum_{k=0}^{|H|-1} h^{k}\mathbf{t}^{kx}}{1- \mathbf{t}^{|H|x}},  \qquad\qquad  h\in H,\ x\in L'. 
\end{equation}
The $h$-equivariant part can also be extracted via Fourier transform
$$
f_{h}(\mathbf{t}) = \frac{1}{|H|} \sum_{\rho\in \widehat{H}} \rho\left( h^{-1} f_{H}(\mathbf{t})\right),
$$
where $\widehat{H}$ is the Pontryagin dual of $H$.
 Using notations of Section \ref{ss:tps}, the Taylor expansion of $f_{H}(\mathbf{t})$ at the origin is
$
\textnormal{T}[f_{H}(\mathbf{t})] = \sum_{l'\in L'} p_{l'}[l']t^{l'} = \sum_{h\in H} Z_{h}(\mathbf{t}) \cdot h.
$
In particular, the Taylor expansion of $f_{h}(\mathbf{t})$ at the origin is  $\textnormal{T}[f_{h}(\mathbf{t})] = Z_{h}(\mathbf{t}) = \sum_{[l']=h}p_{l'}\mathbf{t}^{l'}$, hence $Q_{h}$ is the counting function associated with the Taylor series $\textnormal{T}[f_{h}(\mathbf{t})]$.

The counting function $Q_{h}$ can be presented in terms of coefficient functions of Taylor expansions as follows. It can be written as
$
Q_{h}(x) = \sum_{\substack{l'\not \geq x,\ [l']=h}} p_{l'} = \sum_{l'\in \mathcal{P}(x)\cap (h+L)} p_{l'}, 
$
as counting function supported on the semi-open bounded concave polytope $\mathcal{P}(x) =
\mathcal{S}'_{\mathbb{R}} \setminus \left( x + \mathbb{R}_{\geq0}\langle E_{v}\rangle_{v\in \mathcal{V}}\right)$, where $\mathcal{S}'_{\mathbb{R}} = \mathbb{R}_{\geq0}\langle E^{*}_{v}\rangle_{v\in \mathcal{V}}$.
The decomposition of $\mathcal{P}(x)$ into semi-open convex polytopes $\mathcal{P}_{\mathcal{I}}(x) = \{\sum_{v\in \mathcal{V}}y_{v}E_{v}\in \mathcal{S}'_{\mathbb{R}} : y_{v} < x_{v},\ \forall v\in \mathcal{I}\}$, where $x=\sum_{v\in \mathcal{V}} x_{v} E_{v}$ and $\emptyset \neq \mathcal{I} \subseteq \mathcal{V}$, expressed in terms of characteristic functions as $\mathcal{P}(x) = \sum_{\emptyset \neq \mathcal{I} \subseteq \mathcal{V}} (-1)^{|\mathcal{I}|-1} \mathcal{P}_{\mathcal{I}}(x)$, yields a decomposition of the counting function
\begin{equation*}
Q_{h}(x) = \sum_{\emptyset \neq \mathcal{I} \subseteq \mathcal{V}} (-1)^{|\mathcal{I}|-1} \sum_{l' \in \mathcal{P}_{\mathcal{I}}(x) \cap (h+L)} p_{l'}.
\end{equation*}
It is not hard to see that for any $x\in L'$ 
\begin{equation}\label{Eq-0}
Q_{h}(x) = Q_{h}(x+r_{h-[x]}), 
\end{equation}
where $r_{h-[x]} \in L' \cap \sum_{v\in \mathcal{V}}[0,1) E_{v}$ is the unique lift of $h-[x]\in L'/L$ (cf. \cite[4.3.15]{LN}). 
Since $x+r_{h-[x]}\in h+L$ it is enough to consider the restriction of $Q_{h}$ to the coset $h+L\subset L'$. Hence, for $x,l'\in h+L$ we have $l'\in \mathcal{P}_{\mathcal{I}}(x)$ exactly when $l'\in \mathcal{S}'_{\mathbb{R}}$ and
\begin{equation}\label{Eq-1}
l'_{v}\in x_{v} + \mathbb{Z}_{<0},  \qquad\qquad \textnormal{for all }v\in \mathcal{I}. 
\end{equation}
For any non-empty subset $\mathcal{I}\subseteq \mathcal{V} $ we introduce notation $V_{\mathcal{I}} = \mathbb{R}\langle E_{i}\rangle_{i\in \mathcal{I}}$, $V=V_{\mathcal{V}}$ and we denote by $\pi_{\mathcal{I}}:V\to V_{\mathcal{I}}$ the projection along subspace $V_{\mathcal{V}\setminus \mathcal{I}}$. Moreover, for $y\in \sum_{v\in \mathcal{V}} y_{v}E_{v}\in L'$ we will use notations $\mathbf{t}_{\mathcal{I}}^{y} = \pi_{\mathcal{I}}(\mathbf{t}^{y}) = \mathbf{t}^{\pi_{\mathcal{I}}(y)}:= \prod_{v\in \mathcal{I}}t_{v}^{y_{v}}$. In this terms, (\ref{Eq-1}) can be reformulated as $\pi_{\mathcal{I}}(l' + \sum_{v\in \mathcal{V}}E_{v}) + z_{\mathcal{I}} = \pi_{\mathcal{I}}(x )$ for some $z_{\mathcal{I}}\in \mathbb{Z}_{\geq0}\langle E_{i}\rangle_{i\in \mathcal{I}}$, hence
$\sum_{ l' \in \mathcal{P}_{\mathcal{I}}(x) \cap (h+L)} p_{l'}$ 
can be considered as the coefficient of $\mathbf{t}_{\mathcal{I}}^{x }$ in the Taylor expansion of $f_{h}(\mathbf{t}_{\mathcal{I}}) \prod_{i\in \mathcal{I}} \frac{\mathbf{t}_{\mathcal{I}}^{E_{i}}}{ 1 - \mathbf{t}_{\mathcal{I}}^{E_{i}}  }$ at $\mathbf{t}_{\mathcal{I}}=0$. That is, for $x\in h+L$
\begin{align*}
\sum_{l'\in \mathcal{P}_{\mathcal{I}}(x) \cap (h+L)} p_{l'} 
={}&
\Coeff \bigg( \textnormal{T} \bigg[ f_{h}(\mathbf{t}_{\mathcal{I}}) \prod_{i\in \mathcal{I}} \frac{\mathbf{t}_{\mathcal{I}}^{E_{i}}}{1 - \mathbf{t}_{\mathcal{I}}^{E_{i}}} \bigg] ,\, \mathbf{t}_{\mathcal{I}}^{x} \bigg).
\end{align*} 
Summing up with signs with respect to subsets $\mathcal{I}$ and using relation (\ref{Eq-0}) we get 
$$
Q_{h}(x) 
= 
\sum_{\emptyset \neq \mathcal{I} \subseteq \mathcal{V}} (-1)^{|\mathcal{I}|-1} 
\Coeff \bigg(  \textnormal{T} \bigg[ \mathbf{t}_{\mathcal{I}}^{-r_{h-[x]}} f_{h}(\mathbf{t}_{\mathcal{I}}) \prod_{i\in \mathcal{I}} \frac{\mathbf{t}_{\mathcal{I}}^{E_{i}}}{1 - \mathbf{t}_{\mathcal{I}}^{E_{i}}} \bigg] ,\, \mathbf{t}_{\mathcal{I}}^{x} \bigg),
$$
for any $x \in L'$. In order to unify the calculations, we remove the term $r_{h-[x]}$ and replace $f_{h}$ by a more manageable function $f_{H}$. Therefore, it is convenient to introduce the function $C_{H}:L' \to \mathbb{Z}[H]$,
\begin{equation}\label{Eq-40}
C_{H}(x)= \sum_{\emptyset \neq \mathcal{I} \subseteq \mathcal{V}} (-1)^{|\mathcal{I}|-1} 
\Coeff \bigg(  \textnormal{T} \bigg[f_{H}(\mathbf{t_{\mathcal{I}}})\prod_{i\in \mathcal{I}}\frac{\mathbf{t}_{\mathcal{I}}^{E_{i}}}{1-\mathbf{t}_{\mathcal{I}}^{E_{i}}}\bigg],\, \mathbf{t}_{\mathcal{I}}^{x} \bigg).
\end{equation}
The decomposition $f_{H} = \sum_{h\in H} f_{h} \cdot h$ yields decomposition $C_{H} = \sum_{h\in H} C_{h}\cdot h$. In general, $C_{h}\neq Q_{h}$, nevertheless $C_{[x]}(x) = Q_{[x]}(x)$ for all $x\in L'$. We also define the function $C:L' \to \mathbb{Z}$,
\begin{equation}\label{Eq-50}
C(x) = C_{[x]}(x) = \frac{1}{|H|} \sum_{\rho \in \hat{H}} \rho \left( [x]^{-1} C_{H}(x) \right).
\end{equation}
Note that $Q_{h}$ can be recovered from $C$ via relations
\begin{equation}\label{Eq-39}
Q_{[x]}(x) {}= C(x) \quad \textnormal{and} \quad
Q_{[x]}(x- q) {}=  C(x) \qquad \textnormal{for all $q \in L'\cap \sum_{v\in \mathcal{V}} [0,1) E_{v}$}.
\end{equation}

In the sequel we will study the quasipolynomial behaviour of the functions $C_{H}$ and $C$ on the Lipman cone $\mathcal{S}'$.

\subsection{The Szenes--Vergne formula}\label{sec:SzV}
In this section we recall the notion of Jeffrey--Kirwan residue and the theorem of Szenes and Vergne about quasipolinomiality of coefficient functions, following \cite{SzV} and \cite{BV}.

Let $V$ be an $r$-dimensional real vector space with a rank $r$ lattice $\Gamma\subset V$. We consider a finite collection of non-zero vectors $\Psi\subset \Gamma$ (elements of $\Psi$ are not necessarily disctinct). Assume that all elements of $\Psi$ lie in an open half-space of $V$.
Denote $\mathfrak{B}(\Psi)$ the set of all bases $\sigma
\subset \Psi$ of $V$.
Let $\Gamma^{*} =\{p \in V^{*}_{\mathbb{C}} : e^{\langle \alpha, p\rangle} =1,\, \forall\, \alpha\in \Gamma\}$ be the $2\pi\sqrt{-1}$-multiple of the dual lattice of $\Gamma$.
A \emph{big chamber} is a connected component of $V\setminus \cup_{\sigma \in \mathfrak{B}(\Psi)} \partial \mathbb{R}_{\geq0}\langle\sigma\rangle$, where $\partial \mathbb{R}_{\geq0}\langle\sigma\rangle=\{v=\sum_{\beta \in \sigma} v_{\beta}\beta  : v_{\beta}\geq0\  \forall\, \beta\, \textnormal{ and }   \exists\, \beta \textnormal{ s.t. } v_{\beta}=0 \}$ is the boundary of the closed cone $\mathbb{R}_{\geq0}\langle \sigma\rangle$. 
For a big chamber $\mathfrak{c}$ we define the set of bases $\mathfrak{B}(\Psi, \mathfrak{c})=\{\sigma \in \mathfrak{B}(\Psi) : \mathfrak{c}\subset \mathbb{R}_{\geq0}\langle\sigma\rangle\}$.  

\subsubsection{Total residue}

Consider elements of $V$ as linear functions on $V^{*}$. Denote by $\mathbb{C}[V^{*}]$ the polynomial ring on $V^{*}$ (which can be identified by the symmetric algebra of $V$). Let $\mathbf{R}_{\Psi}$ be the complex vector space spanned by fractions of form $\psi=\frac{\theta}{\prod_{k=1}^{R}\beta_{k}}$, where $\theta\in \mathbb{C}[V^{*}]$, $\beta_{k}\in \Psi$ and $R\in \mathbb{Z}_{>0}$. For any $\sigma \in \mathfrak{B}(\Psi)$ the fraction $\frac{1}{\prod_{\beta\in \sigma}\beta}$ is called \emph{simple}. Let $\mathbf{S}_{\Psi}$ be the subspace of $\mathbf{R}_{\Psi}$ spanned by simple fractions. 
Recall \cite[Proposition 7]{BV} as formulated in \cite[Theorem 1.1]{SzV}.

\begin{thm}
For $v\in V^{*}$ and $\psi\in \mathbf{R}_{\Psi}$ let $\partial_{v}\psi(z):=\frac{d}{d\epsilon} \psi(z+\epsilon v)|_{\epsilon=0}$ be the differentiation in direction $v$. Then there is a direct sum decomposition
$$
\mathbf{R}_{\Psi} = \Big( \sum_{v\in V^{*}} \partial_{v}\mathbf{R}_{\Psi} \Big) \oplus \mathbf{S}_{\Psi}.
$$
\end{thm}
Thus, the projection to the second component $\textnormal{Tres}_{\Psi}:\mathbf{R}_{\Psi} \to \mathbf{S}_{\Psi}$ is a well defined map, which is called the \emph{total residue}.

Grading on $\mathbf{R}_{\Psi}$ is defined as follows. If $\theta \in \mathbb{C}[V^{*}]$ is a homogeneous polynomial of degree $N$ then $\frac{\theta}{\prod_{k=1}^{R}\beta_{k}}\in \mathbf{R}_{\Psi}$ has homogeneous degree $N-R$. The homogeneous degree $d$ part of $\mathbf{R}_{\Psi}$ is denoted by $\mathbf{R}_{\Psi}[d]$ and one has the decomposition $
\mathbf{R}_{\Psi}=\oplus_{d\in \mathbb{Z} }\mathbf{R}_{\Psi}[d]$.

\begin{remark}\label{Rk-1}
Since $\mathbf{S}_{\Psi}\subset \mathbf{R}_{\Psi}[-r]$ the total residue $\textnormal{Tres}_{\Psi}$ vanishes on $\mathbf{R}_{\Psi}[d]$ unless $d=-r$.
\end{remark}

Denote $\mathbb{C}[[V^{*}]]$ the formal power series on $V^{*}$ and let $\widehat{\mathbf{R}}_{\Psi}$ be the complex vector space spanned by fractions of form $\frac{\Theta}{\prod_{k=1}^{R}\beta_{k}}$, where $\Theta\in \mathbb{C}[[V^{*}]]$ and $\beta_{k}\in \Psi$. There is a filtration by degree on $\mathbb{C}[[V^{*}]]$: for $\Theta=\sum_{l\geq0} \Theta_{l}$ with $\Theta_{l}\in \mathbb{C}[V^{*}]_{l}$ homogeneous of degree $l$ we say that $\Theta \in \mathbb{C}[[V^{*}]]_{\geq d}$ if $\Theta_{l}=0$ for all $l<d$. It induces a filtration  $\widehat{\mathbf{R}}_{\Psi} = \sum_{d\in \mathbb{Z}} \widehat{\mathbf{R}}_{\Psi}[\geq d]$ compatible with the grading on $\mathbf{R}_{\Psi}$: we have $\frac{ \Theta }{ \prod_{k=1}^{R}\beta_{k}}\in  \widehat{\mathbf{R}}_{\Psi}[\geq d]$ if $\Theta \in \mathbb{C}[[V^{*}]]_{\geq d+R}$.
The total residue can be extended to $\widehat{\mathbf{R}}_{\Psi}$ as follows. Fix a $d>-r$ and write $\varphi \in \widehat{\mathbf{R}}_{\Psi}$ uniquely as $ \varphi = \varphi_{poly} + \varphi_{ps}$ with $\varphi_{poly} \in \oplus_{k<d }\mathbf{R}_{\Psi}[k]$ and $\varphi_{ps}\in \widehat{\mathbf{R}}_{\Psi}[\geq d]$, then define $
\textnormal{Tres}_{\Psi}( \varphi ) := \textnormal{Tres}_{\Psi}(\varphi_{poly})$. Note that it does not depend on the choice of $d$ by Remark \ref{Rk-1}.

\subsubsection{The Jeffrey--Kirwan residue}

Fix a volume form (or a scalar product) on $V$. For each $\sigma\in \mathfrak{B}(\Psi)$ it gives the volume $vol(\sigma)$ of any paralellepiped $\sum_{\beta \in \sigma}[0,1]\beta$.  
A big chamber $\mathfrak{c}$ is also fixed. 
Using projection $\textnormal{Tres}_{\Psi}$, the \emph{Jeffrey--Kirwan residue} on $\mathbf{S}_{\Psi}$ is defined by the formula
$$
\textnormal{JK}^{\mathfrak{c}} \left(\frac{1}{\prod_{\beta \in \sigma} \beta}\right) 
=
\begin{cases}
\dfrac{1}{vol(\sigma)} & \sigma \in \mathfrak{B}(\Psi,\mathfrak{c}) \\
0 & \sigma \in \mathfrak{B}(\Psi)\setminus \mathfrak{B}(\Psi, \mathfrak{c})
\end{cases} 
$$
and it vanishes on the direct summand $\sum_{v\in V^*}(\partial_{v} \textbf{R}_{\Psi})$.

\subsubsection{The Szenes--Vergne formula}

Consider a rational function $f(\mathbf{t}) = \frac{ \sum_{\xi \in I} c_{\xi} \mathbf{t}^{\xi} }{ \prod_{k=1}^{R} (1 - \mathbf{t}^{\beta_{k}}) }$ with $I$ finite subset of $\Gamma$ and $\beta_{k}\in \Psi$. 
It can be associated with a (meromophic) function on $V^{*}$ of form
$\varphi(z) = f(e^{z}) = \frac{ \sum_{\xi \in I} c_{\xi} e^{\langle \xi, z\rangle} }{ \prod_{k=1}^{R} (1 -  e^{\langle\beta_{k},z\rangle}) }.$ 
Expand $e^{\langle \xi, z\rangle}$ and $(1 - e^{\langle \beta_{k}, z\rangle})^{-1} = -\langle \beta_{k},z\rangle ^{-1} \sum_{l\geq0}  \gamma_{k}(z)^{l}$, where $\gamma_{k}(z) = - \frac{\langle \beta_{k},z\rangle}{2!} - \frac{\langle \beta_{k},z\rangle^{2}}{3!} - \ldots$ into power series, thus  $f(e^{z})$ can be considered as an element of $\widehat{\mathbf{R}}_{\Psi}$.

Let the \emph{set of poles} $SP(\varphi,\Gamma^{*})$ of $\varphi$ be the set of those $p\in V^{*}_{\mathbb{C}}$ such that $\{\beta_{k} : e^{\langle \beta_{k},p\rangle}=1\}$ spans $V$. Note that $SP(\varphi,\Gamma^{*})$ is invariant under translation by elements of $\Gamma^{*}$. Let the \emph{reduced set of poles} be the quotient set $RSP(\varphi,\Gamma^{*}) = SP(\varphi,\Gamma^{*})/\Gamma^{*} = \cup_{\sigma \in \mathfrak{B}(\Psi)} (\mathbb{Z}\sigma)^{*}/\Gamma^{*} $, where $(\mathbb{Z}\sigma)^{*}=\{ p\in V^{*}_{\mathbb{C}} : e^{\langle \alpha, p\rangle}=1,\, \forall\, \alpha\in \sigma\}$ is the $(2\pi\sqrt{-1})$-multiple of the dual lattice of $\mathbb{Z}\sigma$. Note that the reduced set of poles is a finite set and moreover, $e^{\langle \alpha, p\rangle}$ is well defined for any $p \in RSP(\varphi, \Gamma^{*})$ and $\alpha \in \Gamma$, since it does not depend on the representative of $p$.

Finally, we state a slightly weaker version of the \cite[Theorem 2.3 and Lemma 2.2]{SzV}. 
\begin{thm}\label{Thm-SzV}
Denote  $vol\,\Gamma$ the volume form on $V$ such that parallelepiped of basis vectors of $\Gamma$ has volume $1$.
Then for a big chamber $\mathfrak{c}$ the function 
$$
\mathcal{Q}^{\mathfrak{c}}(\lambda) =
 \sum_{p\in RSP(\varphi, \Gamma^{*})} \textnormal{JK}^{\mathfrak{c}}\left( e^{\langle\lambda, z-p\rangle} \varphi(p-z)\right)_{vol\, \Gamma}
$$
is an exponential-polynomial (quasipolynomial) on $\Gamma$, moreover
$
\mathcal{Q}^{\mathfrak{c}}(\lambda)= \Coeff(\textnormal{T}[f(\mathbf{t})],\, \mathbf{t}^{\lambda})$  for all 
\begin{equation}\label{Eq-BigLambdaf}
\lambda \in  \Lambda^{\mathfrak{c}}(f): = \bigcap_{\xi \in I}(\xi +\mathfrak{c} -\square(\beta)) \cap \Gamma, 
\end{equation}
where $\square(\beta) = \sum_{i=1}^{R}[0,1]\beta_{i}$. 
\end{thm}

\begin{remark}\label{Rk-3}
By the properties of the Jeffrey--Kirwan residue, the degree of the quasipolynomial is at most the number of linear terms in the denominator of $f$ minus the dimension of $V$.
\end{remark}

\begin{remark}\label{Rk-4}
Substituting $\frac{1}{1-\mathbf{t}^{\beta_{k}}}$ by equivalent fraction $\frac{1+\mathbf{t}^{\beta_{k}}+\ldots + \mathbf{t}^{(n-1) \beta_{k}}}{1-\mathbf{t}^{n\beta_{k}}}$ in $f(\mathbf{t})$ does not change the domain $\Gamma \cap_{\xi \in I}(\xi +\mathfrak{c} -\square(\beta))$. Thus, 
using expansion of type (\ref{Eq-49}), Theorem \ref{Thm-SzV} remains valid for 
$
\varphi_{H}(z) =  \frac{ \sum_{\xi \in I} c_{\xi} e^{\langle \xi, z\rangle} \prod_{k=1}^{R}\sum_{j=0}^{|H|-1} h_{k}^{j} e^{\langle j\beta_{k},z \rangle}  }{ \prod_{k=1}^{R} (1 -  e^{\langle |H|\beta_{k},z\rangle}) }
$ associated with rational functions with $\mathbb{Z}[H]$-coefficients of form
$$
f_{H}(\mathbf{t}) 
= \frac{ \sum_{\xi \in I} c_{\xi} \mathbf{t}^{\xi} }{ \prod_{k=1}^{R} (1 - h_{k}\mathbf{t}^{\beta_{k}}) }
$$
with $h_{k}$ in a finite abelian group $H$.
\end{remark}


\begin{remark}\label{Rk-2}
If $\xi \in \Gamma \cap \sum_{i=1}^{R} (0,1)\beta_{i}$ for all $\xi$ in the numerator of 
$f$ then $0\in \Lambda^{\mathfrak{c}}(f) $, thus $\mathcal{Q}^{\mathfrak{c}}(0) = \Coeff(\textnormal{T}[f(\mathbf{t})],\, \mathbf{t}^{0}) = 0$ by Theorem \ref{Thm-SzV}. This property also holds for $f_{H}$ with $\mathbb{Z}[H]$-coefficients as in Remark \ref{Rk-4}.
\end{remark}

\subsubsection{}\label{sec:1dim}
In dimension one the quasipolynomial in the above theorem can be computed as follows. 
Assume that $\beta_{1},\ldots,\beta_{k}$ are positive numbers and let $q>0$ be the generator of the lattice $\Gamma$, 
i.e. $\Gamma = q\mathbb{Z}$. 
In this case $\Gamma^{*}= \frac{2\pi \sqrt{-1}}{q}\mathbb{Z}$ and $RSP(\varphi, \Gamma^{*}) = 2\pi \sqrt{-1} \bigcup_{k=1}^{R}\left\{0, 
\frac{1}{\beta_{k}},\ldots, \frac{d_{k}-1}{\beta_{k}} \right\}$ with $d_{k} = \frac{\beta_{k}}{q} \in \mathbb{Z}_{>0}$. Then the 
quasipolynomial is given in terms of usual residue
$$
\lambda  \mapsto q \cdot \sum_{p\in RSP(\varphi, \Gamma^{*})} \Res_{z=0} \left( e^{\lambda(z-p)} \frac{\sum_{\xi \in I} c_{k} e^{\xi(p-z)} }{ \prod_{k=1}^{R} (1 - e^{p\beta_{k}} e^{-\beta_{k}z})} \right)dz
$$
(we chose the volume form on $V$ such that $vol(q)=1$).

We emphasize that the above theorem gives an explicit formula how to compute this quasipolynomial, 
nevertheless the explicit computation may be lengthy.

\section{Uniqueness of the quasipolynomial}\label{sec:unique}

Recall that the counting function $C_{H}(x)$ can be written as an alternating sum of coefficient functions of form $\Coeff \Big(\textnormal{T} \Big[ 
f_{H}(\mathbf{t}_{\mathcal{I}})
\prod_{i\in \mathcal{I}} \frac{\mathbf{t}_{\mathcal{I}}^{E_{i}}}{1 - \mathbf{t}_{\mathcal{I}}^{E_{i}}} \Big] ,\, \mathbf{t}_{\mathcal{I}}^{x} \Big)$ for $x\in  L'$. 
In general, the vector configuration $\Psi_{\mathcal{I}} = \{ \pi_{\mathcal{I}}(E_{l}^{*}), E_{i} : l\in\mathcal{E}, 
i\in \mathcal{I}\}$ of vectors appearing in the denominator of $f_{H}(\mathbf{t}_{\mathcal{I}})
\prod_{i\in \mathcal{I}} \frac{\mathbf{t}_{\mathcal{I}}^{E_{i}}}{1 - \mathbf{t}_{\mathcal{I}}^{E_{i}}}$
divides the open cone $\textnormal{int} (\pi_{\mathcal{I}}(\mathcal{S}'_{\mathbb{R}}))$ into several big chambers.  Hence, in principle one may happen that 
$C_{H}(x)$ can be represented with different quasipolynomials on different chambers.

However, we show that still there is an affine `subcone' $\sum_{v\in \mathcal{V}}(\delta_{v}-2)E^{*}_{v} + \textnormal{int}(\mathcal{S}')$
 of the Lipman cone, where $C_{H}(x)$ is  a unique quasipolynomial. 
In particular, the restriction of $C_{H}$ in the open cone $\textnormal{int}( \mathcal{S}'_{\mathbb{R}})$ to a sufficiently sparse finite rank sublattice of $L'$ is a polynomial.

\subsection{Factorization of the projections of $f_H$}
First, we prove a lemma showing how $f_{H}(\mathbf{t})$ reshapes for different projections. This is the core of the proof of the uniqueness result in Section \ref{ss:unique}. Moreover, it will be important for the structure theorems in Section \ref{ss:structure} too. 
\begin{lemma}\label{Lm-5}
\begin{enumerate}[(i)]
\item\label{Lm-5i}
Let $\overline{v'v''}$ be an edge of a tree $\mathcal{T}$ with vertex set $\mathcal{V}$. Decompose $\mathcal{T}\setminus \overline{v'v''}$ 
into disjoint union of trees $\mathcal{T}_{\mathcal{V}'}$ and $\mathcal{T}_{\mathcal{V}''}$ with vertex sets $\mathcal{V}' \ni v'$ and $\mathcal{V}'' \ni v''$, respectively.
 Then for all non-empty vertex set $\mathcal{J} \subseteq \mathcal{V}''$ the fraction
\begin{equation}\label{Eq-28}
\Big( 1 - h_{v''}\mathbf{t}_{\mathcal{J}}^{ E^{*}_{v''}} \Big) \prod_{v\in \mathcal{V}'} \Big( 1 - h_{v}\mathbf{t}_{\mathcal{J}}^{ E_{v}^{*}  } \Big)^{\delta_{v,\mathcal{V}}-2}
\in \mathbb{Z}[H] [ \pi_{\mathcal{J}}(L')],
\end{equation}
i.e. it is a Laurent polynomial  with coefficients in $\mathbb{Z}[H]$, supported on $\pi_{\mathcal{J}}(\mathcal{S}')$.

\item\label{Lm-5ii}
Let $\mathcal{J}\subseteq \mathcal{V}$ be a non-empty vertex set such that its associated subgraph $\mathcal{T}_{\mathcal{J}}$  is a tree.   Then
$$
\prod_{v\in \mathcal{V}} \Big( 1 - h_{v} \mathbf{t}_{\mathcal{J}}^{E^{*}_{v}} \Big)^{\delta_{v,\mathcal{V}}-2}
=
P_{\mathcal{J}} ( \mathbf{t}_{\mathcal{J}} ) \cdot \prod_{j\in \mathcal{J}} \Big( 1 - h_{j}\mathbf{t}_{\mathcal{J}}^{E^{*}_{j}} \Big)^{\delta_{j,\mathcal{J}}-2},
$$
where 
\begin{equation}\label{Eq-58}
P_{\mathcal{J}}(\mathbf{t}_{\mathcal{J}}) 
= 
\prod_{u\in \mathcal{V}\setminus \mathcal{J}} \Big(1 - h_{u} \mathbf{t}^{E^{*}_{u}}_{\mathcal{J}} \Big)^{\delta_{u,\mathcal{V}}-2} 
\prod_{j\in \mathcal{J}} \Big( 1 - h_{j} \mathbf{t}^{E^{*}_{j}}_{\mathcal{J}} \Big)^{\delta_{j,\mathcal{V}} - \delta_{j,\mathcal{J}}}
\in \mathbb{Z}[H][\pi_{\mathcal{J}}(L')]
\end{equation}
is a Laurent polynomial with coefficients in $\mathbb{Z}[H]$, supported on $\pi_{\mathcal{J}}(\mathcal{S}')$.  

\item\label{Lm-5iii}
There exists a positive definite rational matrix $\widetilde{A}_{\mathcal{J}}$ associated with the subtree $\mathcal{T}_{\mathcal{J}}$ such that
$$
\left[ \pi_{\mathcal{J}}(E_{j}^{*}) \right]_{j\in \mathcal{J}} \cdot \widetilde{A}_{\mathcal{J}} = \left[ E_{j} \right]_{j\in \mathcal{J}}.
$$
\end{enumerate}
\end{lemma}

\begin{proof}
\noindent
(\ref{Lm-5i}) Let $d_{\mathcal{V}',v'} = \max \{ d(v,v') : v\in \mathcal{V}' \}$, where $d(u,v)$ is the length of the minimal path between vertices $u$ and  $v$ in  $\mathcal{T}_{\mathcal{V}'}$. We proceed by induction on $d_{\mathcal{V}',v'}$. 

If $d_{\mathcal{V}',v'}=0$ then $\mathcal{V}'=\{v'\}$ and $v'$ is an end vertex of $\mathcal{T}$. 
Since $\mathcal{V}'\subseteq \mathcal{V} \setminus \mathcal{J}$, by the relation $A_{v'v'}E^{*}_{v'} - E_{v''}^{*} = E_{v'}$ 
comming from (\ref{eq:rel}) the expression (\ref{Eq-28}) becomes 
$$
\frac{ 1 - h_{v''}\mathbf{t}_{\mathcal{J}}^{E^{*}_{v''}} }{ 1- h_{v'}\mathbf{t}_{\mathcal{J}}^{E^{*}_{v'}} } 
=
\frac{ 1 -  h_{v'}^{A_{v'v'}}\mathbf{t}_{\mathcal{J}}^{A_{v'v'}E^{*}_{v'}} }{ 1 - h_{v'} \mathbf{t}_{\mathcal{J}}^{E^{*}_{v'}} } 
=
\sum_{k=0}^{A_{v'v'}-1}
h_{v'}^{k} \mathbf{t}_{\mathcal{J}}^{ k E^{*}_{v'}},
$$
which is a polynomial in $\mathbb{Z}[H][\pi_{\mathcal{J}}(L')]$. 

If $d_{\mathcal{V}',v'}=1$  then $\mathcal{T}_{\mathcal{V}'}$ is star-shaped with node $v'$ and we have $\delta_{v',\mathcal{V}'}=|\mathcal{V}'|-1$. That is, if $\mathcal{V}' = \{v'=v_{0}, v_{1}, \ldots, v_{s}\}$ then $v_{1},\ldots,v_{s}\in \mathcal{E}$ and 
$\delta_{v_{0},\mathcal{V}'} = \delta_{v_{0},\mathcal{V}}-1 = s$.  Moreover, we have relations
\begin{gather}
A_{v_{i}v_{i}} \pi_{\mathcal{J}}( E^{*}_{v_{i}} ) = \pi_{\mathcal{J}}( E^{*}_{v_{0}} ), \qquad i=1,\ldots,s,
\label{Eq-29a}
\notag
\\
A_{v_{0}v_{0}} \pi_{\mathcal{J}}( E^{*}_{v_{0}} ) - \pi_{\mathcal{J}}( E^{*}_{v_{1}} ) - \ldots - \pi_{\mathcal{J}}( E^{*}_{v_{s}} ) = \pi_{\mathcal{J}}( E^{*}_{v''}),
\label{Eq-29b}
\end{gather}
and similarly in $H$ we have
\begin{equation}\label{Eq-45b}
 h_{v_{i}}^{A_{v_{i}v_{i}}} =   h_{v_{0}}, \qquad i=1,\ldots,s,
\qquad\quad \textnormal{and} \qquad\quad
 h_{v_{0}}^{A_{v_{0}v_{0}}}\cdot h_{v_{1}}^{-1} \cdots h_{v_{s}}^{-1}
 = 
 h_{v''}.
\end{equation}
Note that in general one has the following decomposition
\begin{equation}\label{Eq-30}
1 - h_{0}\cdots h_{s}\mathbf{t}^{x_{0} + \ldots + x_{s}} 
=
\sum_{\emptyset \neq S \subset \{0, \ldots, s\}} (-1)^{|S|-1} \prod_{i\in \mathcal{S}} (1 - h_{i}\mathbf{t}^{x_{i}}) = \sum_{i=0}^{s} q_{i} (1- h_{i}\mathbf{t}^{x_{i}}),
\end{equation}
where $q_{i} \in  \mathbb{Z}[H][\mathbf{t}^{x_{i}}:i=0,\ldots,s]$. Hence, together with (\ref{Eq-29b}) and (\ref{Eq-45b}) imply
\begin{gather*}
1 - h_{v''}\mathbf{t}_{\mathcal{J}}^{E^{*}_{v''}} 
=   
q_{0} \Big( 1 - h_{v_{0}}^{A_{v_{0}v_{0}}}\mathbf{t}_{\mathcal{J}}^{A_{v_{0}v_{0}}E^{*}_{v_{0}}} \Big) 
+ 
\sum_{i=1}^{s} q_{i} \Big( 1 - h_{v_{i}}^{-1}\mathbf{t}_{\mathcal{J}}^{-E^{*}_{v_{i}}} \Big)
=
\sum_{i=0}^{s} p_{i} \Big( 1 - h_{v_{i}} \mathbf{t}_{\mathcal{J}}^{E^{*}_{v_{i}}} \Big),
\end{gather*}
where $p_{0} = q_{0}  \sum_{k=0}^{A_{v_{0}v_{0}}-1} h_{v_{0}}^{k}\mathbf{t}_{\mathcal{J}}^{kE^{*}_{v_{0}}} $ and $p_{i} =  - q_{i} h_{v_{i}}^{-1}\mathbf{t}_{\mathcal{J}}^{-E^{*}_{v_{i}}} $, $i=1,\ldots,s$ are Laurent polynomials.
Then (\ref{Eq-28}) becomes
\begin{gather*}
\frac{ \Big( 1 - h_{v''}\mathbf{t}_{\mathcal{J}}^{E^{*}_{v''}} \Big) \Big( 1 - h_{v_{0}}\mathbf{t}_{\mathcal{J}}^{E^{*}_{v_{0}}}   \Big)^{s-1} }{ \prod\limits_{i=1}^{s}  \Big( 1 - h_{v_{i}}\mathbf{t}_{\mathcal{J}}^{E^{*}_{v_{i}}} \Big)  }
=
p_{0} \frac{  \Big( 1 - h_{v_{0}}\mathbf{t}_{\mathcal{J}}^{E^{*}_{v_{0}}} \Big)^{s}  }{ \prod\limits_{l=1}^{s} \Big( 1 - h_{v_{l}}\mathbf{t}_{\mathcal{J}}^{E^{*}_{v_{l}}} \Big) }
+
\sum_{i=1}^{s} p_{i} \frac{  \Big( 1 - h_{v_{0}}\mathbf{t}_{\mathcal{J}}^{E^{*}_{v_{0}}} \Big)^{s-1}  }{ \prod\limits_{\substack{l=1\\ l\neq i}}^{s} \Big(  1 - h_{v_{l}} \mathbf{t}_{\mathcal{J}}^{E^{*}_{v_{l}}}  \Big) } 
\\
=
p_{0} \prod_{l=1}^{s} \sum_{k_{l}=0}^{A_{v_{l}v_{l}}-1} h_{v_{l}}^{k_{l}}\mathbf{t}_{\mathcal{J}}^{k_{l} E^{*}_{v_{l}}} 
+
\sum_{i=1}^{s} p_{i} \prod_{\substack{l=1\\l\neq i}}^{s} 
\sum_{k_{l}=0}^{A_{v_{l}v_{l}}-1} h_{v_{l}}^{k_{l}}\mathbf{t}_{\mathcal{J}}^{k_{l} E^{*}_{v_{l}}}, 
\end{gather*}
which is a Laurent polynomial with coefficents in $\mathbb{Z}[H]$.

In general, we have decomposition $\mathcal{V}' = \{v_{0}\} \uplus \mathcal{V}'_{1} \uplus \ldots \uplus \mathcal{V}'_{s}$ such that $\mathcal{T}_{\mathcal{V}'\setminus \{v_{0}\}} = \mathcal{T}_{\mathcal{V}'_{1}} \uplus \ldots \uplus \mathcal{T}_{\mathcal{V}'_{s}}$ is disjoint union of trees, and let $v_{i}\in \mathcal{V}'_{i}$, $i=1,\ldots,s$ be the neighboring vertices of $v_{0}$ in $\mathcal{T}_{\mathcal{V}'}$. We introduce notation $\phi_{\mathcal{V}'_{i}} = \prod_{u\in \mathcal{V}'_{i}} \big( 1 - h_{u}\mathbf{t}_{\mathcal{J}}^{E^{*}_{u}} \big)^{\delta_{u,\mathcal{V}}-2}$. Then by (\ref{Eq-30}) there are Laurent polynomials $p_{i}$ such that 
\begin{gather*}
\Big( 1 - h_{v''}\mathbf{t}_{\mathcal{J}}^{E^{*}_{v''}} \Big) \prod_{v\in \mathcal{V}'} \Big( 1 - h_{v} \mathbf{t}_{\mathcal{J}}^{E^{*}_{v}} \Big)^{\delta_{v,\mathcal{V}}-2} 
=
\Big( 1 - h_{v''}\mathbf{t}_{\mathcal{J}}^{E^{*}_{v''}} \Big) \cdot \Big( 1 - h_{v_{0}}\mathbf{t}_{\mathcal{J}}^{E^{*}_{v_{0}}} \Big)^{s-1} \prod_{i=1}^{s}  \phi_{\mathcal{V}'_{i}}
\\
=
\left( \sum_{i=0}^{s} p_{i} \Big( 1 - h_{v_{i}}\mathbf{t}_{\mathcal{J}}^{E^{*}_{v_{i }}} \Big)
\right) \cdot \Big( 1 - h_{v_{0}}\mathbf{t}_{\mathcal{J}}^{E^{*}_{v_{0}}}  \Big)^{s-1} \prod_{i=1}^{s} \phi_{\mathcal{V}'_{i}}
\\
=
p_{0} \prod_{i=1}^{s} \Big( 1 - h_{v_{0}}\mathbf{t}_{\mathcal{J}}^{E^{*}_{v_{0}}} \Big) \phi_{\mathcal{V}'_{i}}
+ 
\sum_{i=1}^{s} p_{i} \cdot \Big( 1 - h_{v_{i}}\mathbf{t}_{\mathcal{J}}^{E^{*}_{v_{i}}} \Big) 
\phi_{\mathcal{V}'_{i}}  \cdot 
\prod_{\substack{l=1\\ l\neq i}}^{s}  \Big( 1 - h_{v_{0}}\mathbf{t}_{\mathcal{J}}^{E^{*}_{v_{0}}} \Big) \phi_{\mathcal{V}'_{l}} .
\end{gather*}
Since $d_{\mathcal{V}'_{l},v_{l}} < d_{\mathcal{V}',v_{0}}$, by induction hypothesis $\big( 1 - h_{v_{0}}\mathbf{t}_{\mathcal{J}}^{E^{*}_{v_{0}}} \big) \phi_{\mathcal{V}'_{l}}$ is a Laurent polynomial.  
Finally, we have to show that the term $\big( 1 - h_{v_{i}}\mathbf{t}_{\mathcal{J}}^{E^{*}_{v_{i}}} \big) \phi_{\mathcal{V}'_{i}} $ is also a Laurent polynomial. One can write $\mathcal{T}_{\mathcal{V}'_{i} \setminus \{v_{i}\}} = \mathcal{T}_{\mathcal{W}_{1}} \uplus \ldots \uplus \mathcal{T}_{\mathcal{W}_{r}}$ as disjoint union of trees and let $w_{k}$ be the (unique) neighboring vertex of $v_{i}$ such that $w_{k} \in \mathcal{W}_{k} $, $k=1,\ldots,r$.   Then
\begin{equation}\label{Eq-31}
\Big( 1 - h_{v_{i}}\mathbf{t}_{\mathcal{J}}^{E^{*}_{v_{i}}} \Big) \phi_{\mathcal{V}'_{i}} 
=
\prod_{k=1}^{r} \Big( 1 - h_{v_{i}}\mathbf{t}_{\mathcal{J}}^{E^{*}_{v_{i}}} \Big) \phi_{\mathcal{W}_{k}},
\end{equation}
and $d_{\mathcal{W}_{k},w_{k}} < d_{\mathcal{V}',v_{0}}$, hence (\ref{Eq-31}) is also a Laurent polynomial by induction. This completes the proof of (\ref{Lm-5i}).
\\[1ex]
\noindent (\ref{Lm-5ii})
The subgraph $\mathcal{T}_{\mathcal{V}\setminus \mathcal{J}}$ decomposes into several trees $\mathcal{T}_{\mathcal{J}_{1}},\ldots, \mathcal{T}_{\mathcal{J}_{r}}$ such that $\mathcal{J}_{1} \uplus \ldots \uplus\mathcal{J}_{r} = \mathcal{V} \setminus \mathcal{J}$. For all $i=1,\ldots,r$ denote $j_{i} \in \mathcal{J}$ the unique vertex such that there is an edge in $\mathcal{T}$ from the vertex $j_{i}$ to a vertex in $\mathcal{J}_{i}$ (it is possible that $j_{i}=j_{k}$ for some $i\neq k$). Then
$$
\prod_{v\in \mathcal{V}} \Big( 1- h_{v}\mathbf{t}_{\mathcal{J}}^{E^{*}_{v}} \Big)^{\delta_{v,\mathcal{V}} - 2}
=
\prod_{j\in \mathcal{J}} \Big( 1 - h_{v_{j}}\mathbf{t}_{\mathcal{J}}^{E^{*}_{v_{j}}} \Big)^{\delta_{j,\mathcal{J}}-2} 
\cdot 
\prod_{k=1}^{r} \Big( 1 - h_{j_{k}}\mathbf{t}_{\mathcal{J}}^{E^{*}_{j_{k}}} \Big) 
\prod_{u\in \mathcal{J}_{k}} \Big( 1 - h_{u}\mathbf{t}_{\mathcal{J}}^{E^{*}_{u}} \Big)^{\delta_{u,\mathcal{V}}-2}, 
$$
and by (\ref{Lm-5i}) each term $\big( 1 - h_{j_{k}} \mathbf{t}_{\mathcal{J}}^{E^{*}_{j_{k}}} \big) \prod_{u \in \mathcal{J}_{k}} \big( 1 - h_{u}\mathbf{t}_{\mathcal{J}}^{E^{*}_{u}}  \big)^{\delta_{u,\mathcal{V}}-2}$ is a Laurent polynomial. Note that 
$$
\prod_{k=1}^{r} \Big( 1 - h_{j_{k}}\mathbf{t}_{\mathcal{J}}^{E^{*}_{j_{k}}} \Big) 
\prod_{u\in \mathcal{J}_{k}} \Big( 1 - h_{u}\mathbf{t}_{\mathcal{J}}^{E^{*}_{u}} \Big)^{\delta_{u,\mathcal{V}}-2}
=
\prod_{u\in \mathcal{V}\setminus \mathcal{J}} \Big(1 - h_{u} \mathbf{t}^{E^{*}_{u}}_{\mathcal{J}} \Big)^{\delta_{u,\mathcal{V}}-2} 
\prod_{j\in \mathcal{J}} \Big( 1 - h_{j} \mathbf{t}^{E^{*}_{j}}_{\mathcal{J}} \Big)^{\delta_{j,\mathcal{V}} - \delta_{j,\mathcal{J}}}.
$$
\\[1ex]
\noindent (\ref{Lm-5iii})
For $\mathcal{J}=\mathcal{V}$ we choose $\widetilde{A}_{\mathcal{V}} = A$.
Let $i\in \mathcal{J}$ be an end vertex with neighboring vertex  $j$ and denote $\mathcal{J}' = \mathcal{J}\setminus\{i\}$. Then $\pi_{\mathcal{J}'}$ is the composition of $\pi_{\mathcal{J}}$ and $\pi_{\overline{ij}}$,  the projection along $\mathbb{R}E_{i}$. The latter projection corresponds to removal of edge $\overline{ij}$ from $\mathcal{T}_{\mathcal{J}}$, hence by \cite[Section 21]{EN} we get that $\left[ \pi_{\mathcal{J'}}(E^{*}_{v}) \right]_{v\in \mathcal{J}'} \cdot \widetilde{A}_{\mathcal{J'}} = [E_{v}]_{v\in \mathcal{J}'}$, where $(\widetilde{A}_{\mathcal{J}'})_{kl} = (\widetilde{A}_{\mathcal{J}})_{kl}$ for all $k,l\in \mathcal{J}'$, except $(\widetilde{A}_{\mathcal{J}'})_{jj} = (\widetilde{A}_{\mathcal{J}})_{jj} - 1/(\widetilde{A}_{\mathcal{J}})_{ii}$. Moreover, $\widetilde{A}_{\mathcal{J}'}$ is positive definite and it is associated with $\mathcal{T}_{\mathcal{J}'}$. 
Finally, if $\mathcal{J}\subseteq \mathcal{V}$ then we can realize $\pi_{\mathcal{J}}$ as composition of projections $\pi_{\overline{ij}}$ corresponding to successive removal of edges not in $\mathcal{T}_{\mathcal{J}}$, and which are end edges of the respective intermediate trees.  
\end{proof}

\subsection{Uniqueness theorem}\label{ss:unique}
\begin{thm}\label{Thm-A}
There exists a unique $\mathbb{Z}[H]$-valued quasipolynomial $\mathcal{C}_{H}=\sum_{h\in H} \mathcal{C}_{h} \cdot h$ on the lattice $L'$ (i.e. all equivariant parts $\mathcal{C}_{h}$ are quasipolynomials) with property $\mathcal{C}_{H}(x) = C_{H}(x)$ for all $x\in \sum_{v\in \mathcal{V}} (\delta_{v}-2) E^{*}_{v}+\textnormal{int}\, (\mathcal{S}')$. In particular, there are unique quasipolynomials $\mathcal{C}$ and $\mathcal{L}_{h}$ on $L'$ such that
$$
\mathcal{C}(x) = C(x)\quad and \quad \mathcal{L}_{h}(x)=Q_{h}(x) \quad\qquad for\ all\ x\in \sum_{v\in \mathcal{V}} (\delta_{v}-2) E^{*}_{v} + \textnormal{int}(\mathcal{S}').
$$ 
\end{thm}

Since (\ref{Eq-50}) and (\ref{Eq-39}) imply that 
$$
\mathcal{C}(x) = \frac{1}{|H|}\sum_{\rho \in \hat{H}} \rho([x]^{-1}\mathcal{C}_{H}(x)),
$$ 
$\mathcal{L}_{h}(x) = \mathcal{C}(x)$ and $\mathcal{L}_{h}(x-q) = \mathcal{C}(x)$ for $x\in h+L$ and all $q\in L'\cap \sum_{v\in \mathcal{V}}[0,1)E_{v}$, it is enough to prove the theorem for $\mathcal{C}_{H}$. More precisely, the theorem follows from

\begin{prop}\label{prop:uni}
For all non-empty vertex set $\mathcal{I} \subseteq \mathcal{V}$ the function 
\begin{equation}\label{Eq-57}
x  \mapsto \Coeff \bigg(   \textnormal{T} \bigg[ \mathbf{t}_{\mathcal{I}}^{-q} f_{H}(\mathbf{t}_{\mathcal{I}}) \prod_{i\in \mathcal{I}} \frac{\mathbf{t}_{\mathcal{I}}^{E_{i}}}{1- \mathbf{t}_{\mathcal{I}}^{E_{i}}} \bigg],  \mathbf{t}_{\mathcal{I}}^{x} \bigg)
\end{equation} 
is quasipolynomial on $\sum_{v\in \mathcal{V}} (\delta_{v}-2) E^{*}_{v} + \textnormal{int}(\mathcal{S}')$ for all $q\in L' \cap \sum_{v\in \mathcal{V}}[0,1)E_{v}$.
\end{prop}


In the proof of the proposition we will use the following lemma.

\begin{lemma}\label{Lm-6}
Let $\mathcal{J}\subseteq \mathcal{V}$ be such that $\mathcal{T}_{\mathcal{J}}$ is a subtree of $\mathcal{T}$ and let $\mathcal{I} \subseteq \mathcal{J}$ be such that $\mathcal{I}$ contains the set of end vertices $\mathcal{E}_{\mathcal{J}}$ of $\mathcal{T}_{\mathcal{J}}$. If $\Psi_{\mathcal{I}}= \{  \pi_{\mathcal{I}}(E^{*}_{j}),E_{i} :  j\in \mathcal{E}_{\mathcal{J}}, i\in \mathcal{I}\}$ then for any $\sigma \in \mathfrak{B}(\Psi_{\mathcal{I}})$ we have either $\mathbb{R}_{>0}\langle \pi_{\mathcal{I}}(E^{*}_{j})\rangle_{j\in \mathcal{J}} \cap \mathbb{R}_{>0}\langle \sigma \rangle = \emptyset$ or $\mathbb{R}_{>0}\langle \pi_{\mathcal{I}}(E^{*}_{j})\rangle_{j\in \mathcal{J}} \subset  \mathbb{R}_{>0}\langle \sigma \rangle$. That is, the projected Lipman cone $\textnormal{int}\,  (\pi_{\mathcal{I}}(\mathcal{S}'_{\mathbb{R}})) = \mathbb{R}_{>0} \langle \pi_{\mathcal{I}}(E^{*}_{j}) \rangle_{j\in \mathcal{J}}$ is contained entirely in a big chamber of the configuration $\Psi_{\mathcal{I}}$.
\end{lemma}

\begin{proof}
The statement is equivalent with the following: no facet  
of $\mathbb{R}_{\geq0}\langle \sigma\rangle$ cuts into $\mathbb{R}_{>0} \langle \pi_{\mathcal{I}} (E^{*}_{j})\rangle_{ j\in \mathcal{J}}$, that is
\begin{equation}\label{Eq-24}
\mathbb{R}_{\geq0} \langle \pi_{\mathcal{I}} (E^{*}_{j_{1}}),\ldots,\pi_{\mathcal{I}} (E^{*}_{j_{l}}), E_{i_{l+1}},\ldots, E_{i_{|\mathcal{I}|-1}}  \rangle \cap  \mathbb{R}_{> 0}\langle \pi_{\mathcal{I}} (E^{*}_{j}) \rangle_{j\in \mathcal{J}} = \emptyset,
\end{equation}
for all $j_{1},\ldots,j_{l}\in \mathcal{E}_{\mathcal{J}}$ and $i_{l+1},\ldots,i_{|\mathcal{I}|-1}\in \mathcal{I}$. Suppose that (\ref{Eq-24}) does not hold. Then there are $j_{1},\ldots,j_{l}\in \mathcal{E}_{\mathcal{J}}$ and $i_{l+1},\ldots, i_{|\mathcal{I}|-1}\in \mathcal{I}$ such that
\begin{equation}\label{Eq-25}
\mathbb{R}_{\geq0}\langle \pi_{\mathcal{J}}(E^{*}_{j_{1}}),\ldots, \pi_{\mathcal{J}}(E^{*}_{j_{l}}), E_{i_{l+1}},\ldots, E_{i_{|\mathcal{I}|-1}}  \rangle + V_{\mathcal{J}\setminus \mathcal{I}} \cap \mathbb{R}_{>0} \langle \pi_{\mathcal{J}}(E^{*}_{j}) \rangle_{j\in \mathcal{J}} \neq \emptyset.
\end{equation}
Denote $\mathcal{J}' = \mathcal{J} \setminus \{j_{1},\ldots,j_{l}\}$ and let $\pi'$ be the composition of projections $V_{\mathcal{J}} \to V_{\mathcal{J}'}$ along $\mathbb{R}\langle \pi_{\mathcal{J}} (E^{*}_{j_{1}}),\ldots, \pi_{\mathcal{J}}(E^{*}_{j_{l}})\rangle$ and $\pi_{\mathcal{J}}$. 
Note that $\pi'$ restricted to $V_{\mathcal{J}\setminus \mathcal{I}}$ is the identity map, since $V_{\mathcal{J} \setminus \mathcal{I}} \subset V_{\mathcal{J}'}$. 
Hence, from (\ref{Eq-25}) we get
\begin{equation}\label{Eq-26}
\mathbb{R}_{\geq0}\langle \pi' (E_{i_{l+1}}), \ldots, \pi'(E_{i_{|\mathcal{I}|-1}}) \rangle + V_{\mathcal{J}\setminus \mathcal{I}} \cap \mathbb{R}_{>0} \langle \pi' (E^{*}_{j}) \rangle_{j\in \mathcal{J}'} \neq \emptyset.
\end{equation}

Denote $i' \in \mathcal{J}$ the unique neighboring vertex in $\mathcal{T}_{\mathcal{J}}$ of the end vertex $i\in \mathcal{E}_{\mathcal{J}}$. 
By Lemma \ref{Lm-5}(\ref{Lm-5iii}) we have $[\pi_{\mathcal{J}}(E^{*}_{j})]_{j\in \mathcal{J}} \cdot \widetilde{A} = [E_{j}]_{j\in \mathcal{J}}$, where $\widetilde{A} = \widetilde{A}_{\mathcal{J}}$ is a rational positive definite matrix associated with $\mathcal{T}_{\mathcal{J}}$. 
Then for all $i\in \{j_{1},\ldots, j_{l}\}$ the relation $\widetilde{A}_{ii}E^{*}_{i} - E^{*}_{i'} = E_{i}$ implies that 
$\pi'(E_{i}) = -\pi'(E^{*}_{i'})$ for $i\in \{j_{1},\ldots,j_{l}\}$ and 
$\pi'(E_{i}) = E_{i}$ for $i \in \mathcal{J}'$.
After relabeling such that $\{j_{1},\ldots,j_{l}\} \cap \{i_{l+1},\ldots,i_{|\mathcal{I}|-1}\} = \{i_{l+1},\ldots,i_{s}\}$ the relation (\ref{Eq-26}) becomes
\begin{equation}\label{Eq-27}
\mathbb{R}_{\geq0} \langle -\pi' (E^{*}_{i'_{l+1}}), \ldots, -\pi' (E^{*}_{i'_{s}}), E_{i_{s+1}}, \ldots, E_{i_{|\mathcal{I}|-1}} \rangle 
+ 
V_{\mathcal{J}\setminus \mathcal{I}} 
\cap 
\mathbb{R}_{>0}\langle \pi' (E^{*}_{j}) \rangle_{j\in \mathcal{J}'} \neq \emptyset.
\end{equation}
Note that we have $\big[ \pi'(E^{*}_{j}) \big]_{j \in \mathcal{J}'} \cdot A' = \left[ E_{j}\right]_{j \in \mathcal{J}'}$ with positive definite matrix $A'$ associated with tree $\mathcal{T}_{\mathcal{J}'}$, got from $\widetilde{A}$ by removing rows and columns $j_{1},\ldots,j_{l}$. In particular, for any $j\in \mathcal{J}'$ every coefficient of $\pi' (E^{*}_{j})$ in  the basis $\{ E_{i} \}_{ i\in \mathcal{J}'}$ is positive. Therefore, for $k\in \mathcal{I}\setminus \{j_{1},\ldots,j_{l},i_{s+1},\ldots,i_{|\mathcal{I}|-1}\} \neq \emptyset$ the set
$
\mathbb{R}_{\geq0}\langle -\pi' (E^{*}_{i'_{l+1}}),\ldots, -\pi' (E^{*}_{i'_{s}}), E_{i_{s+1}}, \ldots, E_{i_{|\mathcal{I}|-1}} \rangle + V_{\mathcal{J}\setminus \mathcal{I}}
$
is in the closed half-space $\{ \sum_{j\in \mathcal{J}'} y_{j}E_{j} :  y_{k}\leq 0\}$, while $\mathbb{R}_{>0}\langle \pi' (E^{*}_{j}) \rangle_{j\in \mathcal{J}'} $ lies in the open half-space $\{ \sum_{j\in \mathcal{J}'} y_{j}E_{j} :   y_{k}>0\}$, thus (\ref{Eq-27}) cannot hold.
\end{proof}

\begin{proof}[Proof of Proposition \ref{prop:uni}]
We use notations of Lemma \ref{Lm-6}. By Lemma \ref{Lm-5}(\ref{Lm-5ii})--(\ref{Lm-5iii}) the fraction $f_{H}(\mathbf{t}_{\mathcal{I}})$ simplifies to $P_{\mathcal{J}}(\mathbf{t}_{\mathcal{I}}) \cdot \prod_{j\in \mathcal{J}} \big( 1 - h_{j} \mathbf{t}_{\mathcal{I}}^{E^{*}_{j}} \big)^{\delta_{j,\mathcal{J}}-2}$ with $P_{\mathcal{J}}$ Laurent polynomial defined in (\ref{Eq-58}).
Note that $\Psi_{\mathcal{I}}$ is the configuration of vectors appearing in the denominator of $\mathbf{t}_{\mathcal{I}}^{-q} P_{\mathcal{J}}(\mathbf{t}_{\mathcal{I}}) \prod_{j\in \mathcal{J}} \big( 1 - h_{j} \mathbf{t}_{\mathcal{I}}^{E^{*}_{j}} \big)^{\delta_{j,\mathcal{J}}-2} \prod_{i\in \mathcal{I}} \frac{\mathbf{t}_{\mathcal{I}}^{E_{i}}}{1 - \mathbf{t}_{\mathcal{I}}^{E_{i}}}$, thus Lemma \ref{Lm-6} and Theorem \ref{Thm-SzV} imply the quasipolinomiality of the coefficient function (\ref{Eq-57}) when $\pi_{\mathcal{I}}(x) \in y_{\mathcal{I}} + \textnormal{int}\, (\pi_{\mathcal{I}}(\mathcal{S}'))$ for a suitable $y_{\mathcal{I}} \in \pi_{\mathcal{I}}(L')$.

We show that $\pi_{\mathcal{I}} \big( \sum_{v\in \mathcal{V}}(\delta_{v}-2)E^{*}_{v} \big)$ is an appropriate choice for $y_{\mathcal{I}}$. We can use the presentation $\mathbf{t}_{\mathcal{I}}^{-q} f_{H}(\mathbf{t}_{\mathcal{I}}) \prod_{i\in \mathcal{I}} \frac{ \mathbf{t}_{\mathcal{I}}^{E_{i}}}{1 - \mathbf{t}_{\mathcal{I}}^{E_{i}}}$ to determine the domain of quasipolynomiality of (\ref{Eq-57}) by Theorem \ref{Thm-SzV}. The monomials in the numerator of this presentation are of form $\mathbf{t}_{\mathcal{I}}^{\sum_{w\in \mathcal{N}} k_{w}E^{*}_{w}-q + \sum_{i\in \mathcal{I}}E_{i}}$ with $0\leq k_{w}\leq \delta_{w}-2$, $w\in \mathcal{N}$. Thus, it is enough to show that 
$\sum_{v\in \mathcal{V}} (\delta_{v}-2) \pi_{\mathcal{I}}(E^{*}_{v}) + \textnormal{int}\, (\pi_{\mathcal{I}}(\mathcal{S}')) $ is contained in
$\sum_{w\in \mathcal{N}} k_{w} \pi_{\mathcal{I}}(E^{*}_{w}) - \pi_{\mathcal{I}}(q) + \sum_{i \in \mathcal{I}} E_{i} - \sum_{i\in \mathcal{I}} [0,1]E_{i} - \sum_{u\in \mathcal{E}}[0,1] E^{*}_{u} + \textnormal{int}\, (\pi_{\mathcal{I}}(\mathcal{S}'))$ for all $0\leq k_{w}\leq \delta_{w}-2$, $w\in \mathcal{N}$, which follows from the inclusions 
\begin{itemize}
\item $\sum_{w\in \mathcal{N}} (\delta_{w}-2) \pi_{\mathcal{I}} (E^{*}_{w}) + \textnormal{int}\, (\pi_{\mathcal{I}}(\mathcal{S}'))\subseteq \sum_{w\in \mathcal{N}} k_{w} \pi_{\mathcal{I}}(E^{*}_{w}) + \textnormal{int}\, (\pi_{\mathcal{I}}(\mathcal{S}'))$ for all $0\leq k_{w}\leq \delta_{w}-2$, 
\item $0\in -\pi_{\mathcal{I}}(q) + \sum_{i\in \mathcal{I}} E_{i} - \sum_{i\in \mathcal{I}}[0,1] E_{i}$ for all $q\in L' \cap \sum_{v\in \mathcal{V}} [0,1)E_{v}$ and 
\item $-\sum_{u\in \mathcal{E}} \pi_{\mathcal{I}} (E^{*}_{u}) \in - \sum_{u\in \mathcal{E}}[0,1] \pi_{\mathcal{I}} (E^{*}_{u})$.  
\end{itemize}
\end{proof}

\begin{remark}
 In particular, the uniqueness implies that there is a periodic constant $\textnormal{pc}^{\mathcal{S}'_{\mathbb{R}}}(Z_h)$ associated with the cone 
 $\mathcal{S}'_{\mathbb{R}}$. Since in most cases we consider periodic constants $\textnormal{pc}^{\pi_{\mathcal{J}}(\mathcal{S}'_{\mathbb{R}})}(Z_{h}(\mathbf{t}_{\mathcal{J}}))$, we will omit the cone from the notation for simplicity. 
\end{remark}

\section{Structure of the counting function}\label{s:strcf}
The main goal of this section is to prove a presentation of the counting function $Q_h$ using counting functions associated with 
only one- and two-variable projections of $f_h(\mathbf{t})$. This presentation shows the structure of the quadratic quasipolynomial,  
associated with $Q_h$ on a suitable affine `subcone' of the Lipman cone, in terms of the graph $\mathcal{T}$. 
Moreover, we develop this structure for $Q^{red}_h$ associated with $f_h(\mathbf{t}_{\mathcal{N}})$ in terms of the orbifold graph 
$\mathcal{T}^{orb}$ as well.

\subsection{Projection Lemma}
We prove a projection lemma for the coefficient functions, which will be used in the sequel.

Let $V$ be an $r$-dimensional real vector space with a rank $r$ lattice $\Gamma \subset V$ and let $\Psi=[\beta_{1},\ldots,\beta_{n}] \subset \Gamma$ 
be a collection of vectors lying in an open half space of $V$. 
Denote $\pi:V\to W$  the projection along $\mathbb{R}\langle \beta_{1},\ldots,\beta_{r}\rangle$ to a complementary subspace $W\subset V$. Finally, let $H$ be a finite abelian group.

\begin{lemma}\label{Lm-Pr} 
Assume that $\{\beta_{1},\ldots,\beta_{r}\}$ can be extended to a basis of $\Gamma$ and there is a big chamber $\mathfrak{c} \subset \mathbb{R}_{\geq0}\langle\Psi\rangle$ associated with the vector configuration $\Psi$ such that whenever $\mathfrak{c} 
\subset \mathbb{R}_{\geq0}\langle\sigma\rangle$ for a basis $\sigma \in \mathfrak{B}(\Psi)$ then $\{\beta_{1},\ldots,\beta_{r}\}\subset \sigma$. Then we have
$$
\textnormal{Coeff} \bigg( \textnormal{T} \bigg[ \frac{\sum_{\xi\in I} c_{\xi} \mathbf{t}^{\xi}}{\prod_{i=1}^{r}(1- \mathbf{t}^{\beta_{i}})\prod_{j=r+1}^{n}(1-h_{j}\mathbf{t}^{\beta_{j}})}\bigg], \mathbf{t}^{\lambda} \bigg)
=
\textnormal{Coeff} \bigg( \textnormal{T} \bigg[ \frac{\sum_{\xi \in I}c_{\xi} \mathbf{t}^{\pi(\xi)}}{\prod_{j=r+1}^{n}(1-h_{j}\mathbf{t}^{\pi(\beta_{j}}))}\bigg], \mathbf{t}^{\pi(\lambda)}\bigg) 
$$
for all $\lambda \in  \bigcap_{\xi\in I} (\xi  +\mathfrak{c} - \square(\beta)) \cap \Gamma$, 
where $I\subset \Gamma$ is a finite subset, $\square(\beta) = \sum_{i=1}^{n}[0,1]\beta_{i}$, $c_{\xi}\in \mathbb{Z}[H]$ and $h_{j}\in H$. 
\end{lemma}

\begin{remark*}
Sometimes the numerator $\sum_{\xi \in I} c_{\xi} \mathbf{t}^{\xi}$ is not explicit, thus the description of the precise set of lattice points  where the above lemma holds can be cumbersome. 
Nevertheless, this set always contains a maximal dimensional affine subcone of form $y+(\mathfrak{c} \cap \Gamma)$ with $y\in \mathfrak{c} \cap \Gamma$ where the lemma can be applied. Therefore, we will refer to it as a suitable affine subcone.
\end{remark*}

\begin{proof}
It is enough to prove the lemma when  $\sum_{\xi\in I} c_{\xi}\mathbf{t}^{\xi}=1$. First, we show that for any $\mu\in (\lambda + \mathbb{R}\langle \beta_{1},\ldots,\beta_{r}\rangle) \cap \mathbb{Z}_{\geq0}\langle \beta_{r+1},\ldots,\beta_{n}\rangle$ we have $\lambda \in \mu + \mathbb{Z}_{>0}\langle \beta_{1},\ldots,\beta_{r}\rangle$. Indeed, if we write $\mu = \lambda + \sum_{i=1}^{r} a_{i}\beta_{i}$ then $a_{i}\in \mathbb{Z}$, since $\beta_{1},\ldots,\beta_{r}$ is part of a basis of $\Gamma$. We use induction on $r$ to show that $a_{i}<0$ for all $i$. 

Remark that for any $\sigma \in \mathfrak{B}(\Psi)$ such that $\mathfrak{c}\subset \mathbb{R}_{\geq0}\langle\sigma\rangle$ we have $\lambda, \beta_{1}, \ldots, \beta_{r}\in \mathbb{R}_{\geq0}\langle\sigma\rangle$, hence $\mathbb{R}_{\geq0}\langle \lambda, \beta_{1},\ldots, \beta_{r}\rangle$ is in the closure of $\mathfrak{c}$. 
Therefore, if $a_{i}\geq0$ for all $i$ then $\lambda\in \mathfrak{c}$ implies that $\mu \in \mathfrak{c}$.
The condition whenever $\mathfrak{c} 
\subset \mathbb{R}_{\geq0}\langle\sigma\rangle$ for a basis $\sigma \in \mathfrak{B}(\Psi)$ then $\{\beta_{1},\ldots,\beta_{r}\}\subset \sigma$, implies $\mathfrak{c} \subset \mathbb{R}_{\geq0}\langle\Psi\rangle \setminus \mathbb{R}_{\geq0}\langle \beta_{r+1},\ldots,\beta_{n}\rangle$, thus $\mu \in \mathfrak{c}$ contradicts the assumption $\mu \in \mathbb{Z}_{\geq0}\langle\beta_{r+1},\ldots,\beta_{n}\rangle$. Therefore, one of the $a_{i}$'s must be negative, and we can assume that $a_{r}<0$. Then we can write $\mu - a_{r}\beta_{r} = \lambda+ \sum_{i=1}^{r-1}a_{i}\beta_{i} \in (\lambda + \mathbb{R}\langle\beta_{1},\ldots,\beta_{r-1}\rangle)\cap \mathbb{Z}_{\geq0} \langle \beta_{r},\ldots,\beta_{n}\rangle$ and we can proceed as before. Thus, we get that $a_{1},\ldots,a_{r}<0$. 
Finally, we compute
\begin{gather*}
\textnormal{Coeff} \bigg( \textnormal{T} \bigg[ \frac{1}{\prod_{j=r+1}^{n}(1-h_{j}\mathbf{t}^{\pi(\beta_{j}}))}\bigg], \mathbf{t}^{\pi(\lambda)}\bigg)
=
\sum_{\mu \in \lambda + \mathbb{Z}_{<0}\langle \beta_{1},\ldots,\beta_{r}\rangle}
\textnormal{Coeff} \bigg( \textnormal{T} \bigg[ \frac{1}{\prod_{j=r+1}^{n}(1-h_{j}\mathbf{t}^{\beta_{j}})}\bigg], \mathbf{t}^{\mu}\bigg)
\\
=
\textnormal{Coeff} \bigg( \textnormal{T} \bigg[ \frac{1}{\prod_{i=1}^{r}(1- \mathbf{t}^{\beta_{i}})\prod_{j=r+1}^{n}(1-h_{j}\mathbf{t}^{\beta_{j}})}\bigg], \mathbf{t}^{\lambda} \bigg)
\end{gather*}
 for any $\lambda \in \mathfrak{c} \cap \Gamma$. Moreover, both coefficient functions are quasipolynomial for $\lambda \in (\mathfrak{c} -\square(\beta)) \cap \Gamma$ by Theorem \ref{Thm-SzV}. Furthermore, the quasipolynomials concide, since they agree on the maximal dimensional cone $\mathfrak{c} \cap \Gamma$, thus the coefficient functions must agree on the larger set $(\mathfrak{c} -\square(\beta)) \cap \Gamma$, too.
\end{proof}

\subsection{Structure theorem}\label{ss:structure}
\begin{thm}\label{Thm-2}
With short notation $\mathbf{t}_{vw}  = \mathbf{t}_{\{v,w\}}$, for any $x\in \sum_{v\in \mathcal{V}}(\delta_{v}-2)E^{*}_{v} + 
\textnormal{int}\,  (\mathcal{S}')$  we have
\begin{align*}
Q_{h}(x) 
={}&
\sum_{v\in \mathcal{V}} \textnormal{Coeff}\left(  \textnormal{T} \left[ t_{v}^{-r_{h-[x]}} f_{h}(t_{v}) \frac{t_{v}}{1-t_{v}}\right], t_{v}^{x}
\right)
\\
&{}-\sum_{\overline{vw}\textnormal{ edge of }\mathcal{T}} \textnormal{Coeff}\left(  \textnormal{T} \left[ \mathbf{t}_{vw}^{-r_{h-[x]}} f_{h}(\mathbf{t}_{vw}) \frac{t_{v}t_{w}}{(1-t_{v})(1-t_{w})}\right], \mathbf{t}_{vw}^{x}
\right).
\end{align*}

\end{thm}

Regrouping the coefficient functions we can express the above result in terms of counting functions.

\begin{cor}\label{cor:structure}
For any $x\in \sum_{v\in \mathcal{V}}(\delta_{v}-2)E^{*}_{v}+\textnormal{int}\, (\mathcal{S}')$  we have
\begin{align*}
Q_{h}(x) 
={}&
\sum_{\overline{vw}\textnormal{ edge}}
Q^{vw}_{h}(x)
- \sum_{v\in \mathcal{V}} (\delta_{v}-1) Q^{v}_{h}(x),
\end{align*}
where $Q^{v}_{h}$ and $Q^{vw}_{h}$ are counting functions associated with one-variable series $Z_{h}(t_{v}) = \textnormal{T}[f_{h}(t_{v})]$ and two-variable series $Z_{h}(\mathbf{t}_{vw}) = \textnormal{T}[f_{h}(\mathbf{t}_{vw})]$, respectively. In particular, the multivariable periodic constant can be expressed in terms of one- and two-variable periodic constants
$$
\textnormal{pc}(Z_{h}(\mathbf{t})) = \sum_{\overline{vw}\textnormal{ edge of }\mathcal{T}}\textnormal{pc}(Z_{h}(\mathbf{t}_{vw})) - \sum_{v\in \mathcal{V}} (\delta_{v}-1)\textnormal{pc}(Z_{h}(t_{v})).
$$
\end{cor}

In the proof of the theorem we use the following lemma.

\begin{lemma}\label{Lm-8}
Let $\mathcal{J}\subseteq \mathcal{V}$ be a subset such that $\mathcal{T}_{\mathcal{J}}$ is a subtree of $\mathcal{T}$ and denote $\mathcal{E}_{\mathcal{J}}\subseteq \mathcal{J}$ the set of end vertices of $\mathcal{T}_{\mathcal{J}}$. Then for any $\mathcal{I}$ such that $\mathcal{E}_{\mathcal{J}}\subseteq \mathcal{I} \subseteq \mathcal{J}$ and  for any $x\in \sum_{v\in \mathcal{V}} (\delta_{v}-2)E^{*}_{v} + \textnormal{int}\, (\mathcal{S}')$, one has the following
\begin{equation*}
\textnormal{Coeff} \bigg( \textnormal{T} \bigg[ \mathbf{t}_{\mathcal{J}}^{-r_{h-[x]}}  f_{h}(\mathbf{t}_{\mathcal{J}}) \prod_{j \in \mathcal{J}} \frac{\mathbf{t}_{\mathcal{J}}^{E_{j}}}{1-\mathbf{t}_{\mathcal{J}}^{E_{j}}} \bigg], \mathbf{t}_{\mathcal{J}}^{x} \bigg)
=
\textnormal{Coeff} \bigg( \textnormal{T} \bigg[ \mathbf{t}_{\mathcal{I}}^{-r_{h-[x]}}  f_{h}(\mathbf{t}_{\mathcal{I}}) \prod_{i \in \mathcal{I}} \frac{\mathbf{t}_{\mathcal{I}}^{E_{i}}}{1-\mathbf{t}_{\mathcal{I}}^{E_{i}}} \bigg], \mathbf{t}_{\mathcal{I}}^{x} \bigg).
\end{equation*}
\end{lemma}

\begin{proof}
Set $x = \bar{x} + r_{[x]}$ and note that the lemma is equivalent with  
$$
\textnormal{Coeff} \bigg( \textnormal{T} \bigg[ \mathbf{t}_{\mathcal{J}}^{-r_{h-[x]}-r_{[x]}}  f_{h}(\mathbf{t}_{\mathcal{J}}) \prod_{j \in \mathcal{J}} \frac{\mathbf{t}_{\mathcal{J}}^{E_{j}}}{1-\mathbf{t}_{\mathcal{J}}^{E_{j}}} \bigg], \mathbf{t}_{\mathcal{J}}^{\bar{x}} \bigg)
=
\textnormal{Coeff} \bigg( \textnormal{T} \bigg[ \mathbf{t}_{\mathcal{I}}^{-r_{h-[x]}-r_{[x]}}  f_{h}(\mathbf{t}_{\mathcal{I}}) \prod_{i \in \mathcal{I}} \frac{\mathbf{t}_{\mathcal{I}}^{E_{i}}}{1-\mathbf{t}_{\mathcal{I}}^{E_{i}}} \bigg], \mathbf{t}_{\mathcal{I}}^{\bar{x}} \bigg)
$$
for any $\bar{x} \in -r_{[x]} + \sum_{v\in \mathcal{V}}(\delta_{v}-2)E^{*}_{v} + \textnormal{int}\, (\mathcal{S}')$, which will follow from  Lemma \ref{Lm-Pr}. 

We start verifying the conditions of Lemma \ref{Lm-Pr}. In Lemma \ref{Lm-5}(\ref{Lm-5ii}) we proved that $f_{H}(\mathbf{t}_{\mathcal{J}})$ simplifies to a fraction $P_{\mathcal{J}}(\mathbf{t}_{\mathcal{J}}) \prod_{j\in \mathcal{J}} \big( 1 - h_{j} \mathbf{t}_{\mathcal{J}}^{E^{*}_{j}} \big)^{\delta_{j,\mathcal{J}}-2}$, where $P_{\mathcal{J}}$ is the Laurent polynomial given  by (\ref{Eq-58}).
Using expansion of type (\ref{Eq-49}) for this latter fraction,   $\mathbf{t}_{\mathcal{J}}^{-r_{h-[x]}-r_{[x]}}  f_{h}(\mathbf{t}_{\mathcal{J}}) \prod_{j \in \mathcal{J}} \frac{\mathbf{t}_{\mathcal{J}}^{E_{j}}}{1-\mathbf{t}_{\mathcal{J}}^{E_{j}}}$ can be simplified to a fraction $g_{\mathcal{J}}$ such that only vectors $|H|\cdot \pi_{\mathcal{J}}(E^{*}_{i})$ and $E_{j}$ for $i\in \mathcal{E}_{\mathcal{J}}$, $j\in \mathcal{J}$ appear in its denominator. Moreover, Lemma \ref{Lm-6} shows that $\textnormal{int}\, (\pi_{\mathcal{J}}(\mathcal{S}'_{\mathbb{R}}))$ is contained entirely in a big chamber $\mathfrak{c}$ of the vector configuration $\Psi_{\mathcal{J}} = \{|H|\cdot \pi_{\mathcal{J}}(E^{*}_{i}), E_{j} : i \in \mathcal{E}_{\mathcal{J}}, j\in \mathcal{J}\}$. Note that $\{E_{j}\}_{j\in \mathcal{J}}$ is a basis of $\pi_{\mathcal{J}}(L)$ and the fraction $g_{\mathcal{J}}(\mathbf{t}_{\mathcal{J}})$ is a rational function in variables $t_{j} = \mathbf{t}_{\mathcal{J}}^{E_{j}}$, $j\in \mathcal{J}$. 
The condition  of Lemma \ref{Lm-Pr} for  $\mathfrak{c}$ is satisfied if for any  $\alpha_{1},\ldots,\alpha_{|\mathcal{J}|} \in \Psi_{\mathcal{J}}$ such that 
\begin{equation}\label{Eq-52}
\textnormal{int}\, (\pi_{\mathcal{J}}(\mathcal{S}'_{\mathbb{R}})) =  \mathbb{R}_{> 0} \langle \pi_{\mathcal{J}}(E_{j}^{*})\rangle_{j\in \mathcal{J}} \subset \mathbb{R}_{\geq 0} \langle \alpha_{1},\ldots,\alpha_{|\mathcal{J}|}\rangle
\end{equation} 
we have $E_{v}\in \{\alpha_{1},\ldots,\alpha_{|\mathcal{J}|}\}$ for all $v\notin \mathcal{E}_{\mathcal{J}}$, in particular for all $v\in \mathcal{J}\setminus \mathcal{I}$. If (\ref{Eq-52}) does not hold then there exists $w\in \mathcal{E}_{\mathcal{J}}$ such that $E_{w} = \alpha_{k}$ and $|H|\cdot \pi_{\mathcal{J}}(E_{w}^{*}) = \alpha_{l}$ for some $k$ and $l$.  Denote $w'\in \mathcal{J}$ the unique neighbor of the vertex $w$ in $\mathcal{T}_{\mathcal{J}}$. Hence $\widetilde{A}_{ww}\pi_{\mathcal{J}}(E_{w}^{*}) - \pi_{\mathcal{J}}(E_{w'}^{*}) = E_{w}$ by Lemma \ref{Lm-5}(\ref{Lm-5iii}).  The inclusion (\ref{Eq-52}) yields $\pi_{\mathcal{J}}(E_{w'}^{*}) = \sum_{i=1}^{|\mathcal{J}|} b_{i}\alpha_{i}$ with $b_{i}\geq0$ for all $i$. Thus, $\pi_{\mathcal{J}}(E^{*}_{w'}) = -\alpha_{k} + \frac{\widetilde{A}_{ww}}{|H|}\alpha_{l} = b_{1}\alpha_{1} + b_{2}\alpha_{2} + \ldots + b_{|\mathcal{J}|}\alpha_{|\mathcal{J}|}$, which implies that $b_{k}=-1$ by linear independence of $\alpha_{i}$'s, which is a contradiction.

Finally, we show that $-\pi_{\mathcal{J}}(r_{[x]})+\sum_{v\in \mathcal{V}}(\delta_{v}-2) \pi_{\mathcal{J}}(E^{*}_{v}) + \textnormal{int}\, (\pi_{\mathcal{J}}(\mathcal{S}'))$ is included in  $\Lambda^{\mathfrak{c}} (g_{\mathcal{J}})$ (cf. \ref{Eq-BigLambdaf}).  Since $ \mathbf{t}_{\mathcal{J}}^{-r_{h-[x]}-r_{[x]}}  f_{h}(\mathbf{t}_{\mathcal{J}}) \prod_{j \in \mathcal{J}} \frac{\mathbf{t}_{\mathcal{J}}^{E_{j}}}{1-\mathbf{t}_{\mathcal{J}}^{E_{j}}}$ is a summand of $ \mathbf{t}_{\mathcal{J}}^{-r_{h-[x]}-r_{[x]}}  f_{H}(\mathbf{t}_{\mathcal{J}}) \prod_{j \in \mathcal{J}} \frac{\mathbf{t}_{\mathcal{J}}^{E_{j}}}{1-\mathbf{t}_{\mathcal{J}}^{E_{j}}}$, we have 
$\Lambda^{\mathfrak{c}} \Big( \mathbf{t}_{\mathcal{J}}^{-r_{h-[x]}-r_{[x]}}  f_{H}(\mathbf{t}_{\mathcal{J}}) \prod_{j \in \mathcal{J}} \frac{\mathbf{t}_{\mathcal{J}}^{E_{j}}}{1-\mathbf{t}_{\mathcal{J}}^{E_{j}}} \Big) \subset \Lambda^{\mathfrak{c}}(g_{\mathcal{J}})$, and similar proof as in Proposition \ref{prop:uni} shows that
$-\pi_{\mathcal{J}}(r_{[x]})+\sum_{v\in \mathcal{V}}(\delta_{v}-2) \pi_{\mathcal{J}}(E^{*}_{v}) + \textnormal{int}\, (\pi_{\mathcal{J}}(\mathcal{S}'))$ is contained in
$
\Lambda^{\mathfrak{c}} \Big(  \mathbf{t}_{\mathcal{J}}^{-r_{h-[x]}-r_{[x]}}  f_{H}(\mathbf{t}_{\mathcal{J}}) \prod_{j \in \mathcal{J}} \frac{\mathbf{t}_{\mathcal{J}}^{E_{j}}}{1-\mathbf{t}_{\mathcal{J}}^{E_{j}}} \Big)$.
\end{proof}

\begin{proof}[Proof of Theorem \ref{Thm-2}]
We have decomposition
\begin{equation}\label{Eq-53}
\{\mathcal{I}  : \emptyset \neq \mathcal{I} \subseteq \mathcal{V}\} = 
\biguplus_{\mathcal{J} \in \mathcal{V}_{tree}} \{ \mathcal{I}  : \mathcal{E}_{\mathcal{J}} \subseteq \mathcal{I} \subseteq \mathcal{J} \},
\end{equation}
where $\mathcal{V}_{tree} = \{\mathcal{J} \subseteq \mathcal{V} : \mathcal{T}_{\mathcal{J}}\textnormal{ is a subtree}\}$ and $\mathcal{E}_{\mathcal{J}}$ is the end-vertex set of $\mathcal{T}_{\mathcal{J}}$. Moreover, if $\mathcal{J}\in \mathcal{V}_{tree}$ such that $\mathcal{E}_{\mathcal{J}} \neq \mathcal{J}$ then 
\begin{gather}
\label{Eq-54}
\sum_{\mathcal{E}_{\mathcal{J}} \subseteq \mathcal{I} \subseteq \mathcal{J}} (-1)^{|\mathcal{I}|-1} \textnormal{Coeff} \bigg(  \textnormal{T} \bigg[  \mathbf{t}_{\mathcal{I}}^{-r_{h-[x]}}  f_{h}(\mathbf{t}_{\mathcal{I}})\prod_{i\in \mathcal{I}} \frac{\mathbf{t}_{\mathcal{I}}^{E_{i}}}{ 1 - \mathbf{t}_{\mathcal{I}}^{E_{i}} } \bigg], \mathbf{t}_{\mathcal{I}}^{x} \bigg)=
\\
\textnormal{Coeff} \bigg(  \textnormal{T} \bigg[  \mathbf{t}_{\mathcal{J}}^{-r_{h-[x]}} f_{h}(\mathbf{t}_{\mathcal{J}})\prod_{j\in \mathcal{J}} \frac{\mathbf{t}_{\mathcal{J}}^{E_{j}}}{ 1 - \mathbf{t}_{\mathcal{J}}^{E_{j}} } \bigg], \mathbf{t}_{\mathcal{J}}^{x} \bigg)
\sum_{\mathcal{E}_{\mathcal{J}} \subseteq \mathcal{I} \subseteq \mathcal{J}} (-1)^{|\mathcal{I}|-1} =0
\notag
\end{gather}
by Lemma \ref{Lm-8}. Furthermore,  the only subtrees $\mathcal{T}_{\mathcal{J}}$ of $\mathcal{T}$ with $\mathcal{J} = \mathcal{E}_{\mathcal{J}}$ are of form $\mathcal{J}=\{v\}$ with $v\in \mathcal{V}$ and $\mathcal{J}=\{v,w\}$ with $\overline{vw}$ edge of $\mathcal{T}$. Hence, the  decomposition (\ref{Eq-53}) of index sets and cancelations from (\ref{Eq-54})  yield the simplified formula for $Q_{h}$.
\end{proof}

\subsection{Structure theorem: orbifold version}\label{s:strcforb}
As we already discussed in Section \ref{ss:rps}, it is rather important to look at the `reduced' function $f_{h}(\mathbf{t}_{\mathcal{N}})$ and its series $Z_{h}(\mathbf{t}_{\mathcal{N}})$, since it reduced the amount of computation significantly and contains the same information about the Seiberg--Witten invariant. 
These facts  emphasize the importance 
of proving the reduced version of Theorem \ref{Thm-2} and Corollary \ref{cor:structure}. These will lead to a generalization of the polyonomial part of the reduced Poincar\'e series from one and two nodal graphs (cf. Section \ref{s:polpartsec}). 

Similarly to Section \ref{sec:coeff}, the  counting function  of the reduced series $Z_{h}(\mathbf{t}_{\mathcal{N}}) = \textnormal{T}[f_{h}(\mathbf{t}_{\mathcal{N}})]$ can be written as
$$
Q^{red}_{h}(x) 
= 
\sum_{\emptyset \neq \mathcal{I} \subseteq \mathcal{\mathcal{N}}} (-1)^{|\mathcal{I}|-1} 
\Coeff \bigg(  \textnormal{T} \bigg[ \mathbf{t}_{\mathcal{I}}^{-r_{h-[x]}} f_{h}(\mathbf{t}_{\mathcal{I}}) \prod_{i\in \mathcal{I}} \frac{\mathbf{t}_{\mathcal{I}}^{E_{i}}}{1 - \mathbf{t}_{\mathcal{I}}^{E_{i}}} \bigg] ,\, \mathbf{t}_{\mathcal{I}}^{x} \bigg).
$$
 This can be simplified as follows.

\begin{thm}\label{Prop-2}
For any $x\in \sum_{v\in \mathcal{V}}(\delta_{v}-2)E^{*}_{v} + \textnormal{int}\, (\mathcal{S}')$  we have
\begin{align*}
Q_{h}^{red}(x) 
={}&
\sum_{v\in \mathcal{N}} \textnormal{Coeff}\left(  \textnormal{T} \left[ t_{v}^{-r_{h-[x]}} f_{h}(t_{v}) \frac{t_{v}}{1-t_{v}}\right], t_{v}^{x}
\right)
\\
&{}-\sum_{\overline{vw}\textnormal{ edge of } \mathcal{T}^{orb}} \textnormal{Coeff}\left(  \textnormal{T} \left[ \mathbf{t}_{vw}^{-r_{h-[x]}} f_{h}(\mathbf{t}_{vw}) \frac{t_{v}t_{w}}{(1-t_{v})(1-t_{w})}\right], \mathbf{t}_{vw}^{x}
\right),
\end{align*}
where $\mathcal{T}^{orb}$ is the orbifold graph of $\mathcal{T}$.
\end{thm}

\begin{proof}
The proof goes similarly as in the case of Theorem \ref{Thm-2}.
For any $\mathcal{I} \subseteq \mathcal{N} $  let  $\bar{\mathcal{I}}$ be the minimal subset of $\mathcal{V}$ with property $\mathcal{I}\subseteq \bar{\mathcal{I}}$ and $\mathcal{T}_{\bar{\mathcal{I}}}$ is a subtree of $\mathcal{T}$. Denote $\mathcal{N}_{tree} := \{\mathcal{J}\subseteq \mathcal{N} : \mathcal{T}^{orb}_{\mathcal{J}} \textnormal{ subtree of }\mathcal{T}^{orb}\}$. Then $\mathcal{J}=\bar{\mathcal{I}} \cap \mathcal{N}$ is the minimal set of $\mathcal{N}$ such that $\mathcal{I}\subseteq \mathcal{J}$ and $\mathcal{J} \in \mathcal{N}_{tree}$. Since $\mathcal{E}_{\bar{\mathcal{I}}}\subseteq \mathcal{I} , \mathcal{J} \subseteq \bar{\mathcal{I}}$, Lemma \ref{Lm-8} implies that 
$$
\textnormal{Coeff} \bigg( \textnormal{T} \bigg[ \mathbf{t}_{\mathcal{I}}^{-r_{h-[x]}}  f_{h}(\mathbf{t}_{\mathcal{I}}) \prod_{i \in \mathcal{I}} \frac{\mathbf{t}_{\mathcal{I}}^{E_{i}}}{1-\mathbf{t}_{\mathcal{I}}^{E_{i}}} \bigg], \mathbf{t}_{\mathcal{I}}^{x} \bigg)
=
\textnormal{Coeff} \bigg(  \textnormal{T} \bigg[  \mathbf{t}_{\mathcal{J}}^{-r_{h-[x]}} f_{h}(\mathbf{t}_{\mathcal{J}}) \prod_{j \in \mathcal{J}} \frac{\mathbf{t}_{\mathcal{J}}^{E_{j}}}{1-\mathbf{t}_{\mathcal{J}}^{E_{j}}} \bigg], \mathbf{t}_{\mathcal{J}}^{x} \bigg).
$$
Moreover, for any $\mathcal{J}\in \mathcal{N}_{tree}$ with $\mathcal{E}_{\mathcal{J}}\neq \mathcal{J}$ one has
\begin{gather*}
\sum_{\mathcal{J} \in \mathcal{N}_{tree}} \sum_{\mathcal{E}_{\mathcal{J}} \subseteq \mathcal{I} \subseteq \mathcal{J} }  
(-1)^{|\mathcal{I}|-1} 
\Coeff \bigg(  \textnormal{T} \bigg[ \mathbf{t}_{\mathcal{I}}^{-r_{h-[x]}} f_{h}(\mathbf{t}_{\mathcal{I}}) \prod_{i\in \mathcal{I}} \frac{\mathbf{t}_{\mathcal{I}}^{E_{i}}}{1 - \mathbf{t}_{\mathcal{I}}^{E_{i}}} \bigg] ,\, \mathbf{t}_{\mathcal{I}}^{x} \bigg)
\\
=
\sum_{\mathcal{J} \in \mathcal{N}_{tree}} 
\Coeff \bigg(  \textnormal{T} \bigg[ \mathbf{t}_{\mathcal{J}}^{-r_{h-[x]}} f_{h}(\mathbf{t}_{\mathcal{J}}) \prod_{i\in \mathcal{J}} \frac{\mathbf{t}_{\mathcal{J}}^{E_{i}}}{1 - \mathbf{t}_{\mathcal{J}}^{E_{i}}} \bigg] ,\, \mathbf{t}_{\mathcal{J}}^{x} \bigg)
\sum_{\mathcal{E}_{\mathcal{J}} \subseteq \mathcal{I} \subseteq \mathcal{J} }  
(-1)^{|\mathcal{I}|-1} 
=0.
\end{gather*}
Furthermore, subsets $\mathcal{J}\in \mathcal{N}_{tree}$ with $\mathcal{E}_{\mathcal{J}}=\mathcal{J}$ are one element subsets $\mathcal{J}=\{v\}$ with $v\in \mathcal{N}$ and two element subsets $\mathcal{J}=\{v,w\}$ with $\overline{vw}$ edge of $\mathcal{T}^{orb}$. Thus, the decomposition and  cancellations yield the simplified form of $Q^{red}_{h}$.
\end{proof}

Again, we can express the above simplified formula in terms of counting functions.

\begin{cor}\label{Cor-1}
Let $\delta_{v}^{orb}$ be the valency of $v\in \mathcal{N}$ in the orbifold graph $\mathcal{T}^{orb}$. Then
\begin{align*}
Q_{h}^{red}(x) 
={}&
\sum_{\overline{vw}\textnormal{ edge of }\mathcal{T}^{orb}}
Q^{vw}_{h}(x)
- \sum_{v\in \mathcal{N}} (\delta_{v}^{orb}-1) Q^{v}_{h}(x) 
\end{align*}
for any $x\in \sum_{v\in \mathcal{V}}(\delta_{v}-2)E^{*}_{v} +\textnormal{int}\, (\mathcal{S}')$. In particular, 
$$
\textnormal{pc}(Z_{h}(\mathbf{t}_{\mathcal{N}})) 
= 
\sum_{\overline{vw}\textnormal{ edge of }\mathcal{T}^{orb}} \textnormal{pc}(Z_{h}(\mathbf{t}_{vw}))
- 
\sum_{v\in \mathcal{N}} (\delta_{v}^{orb}-1)  \textnormal{pc}(Z_{h}(t_{v}))
$$
(here the periodic constants are associated with appropriate projections of the Lipman cone).
\end{cor}

\section{Polynomial generalization of the Seiberg--Witten invariant}\label{s:polpartsec}

\subsection{Motivation}
The main goal of this section is to prove that for the reduced rational function $f_h(\mathbf{t}_{\mathcal{N}})$ ($h\in H$) there is a special 
decomposition 
$$f_h(\mathbf{t}_{\mathcal{N}})=P_h(\mathbf{t}_{\mathcal{N}})+f^{-}_h(\mathbf{t}_{\mathcal{N}}),$$
in such a way that $P_h$ is a Laurent polynomial with $P_h(\mathbf{1})=\textnormal{pc}(Z_h(\mathbf{t}_{\mathcal{N}}))$, hence 
$\textnormal{pc}(T[f^{-}_h(\mathbf{t}_{\mathcal{N}})])=0$, and $f^{-}_h(\mathbf{t}_{\mathcal{N}})$ satisfies certain properties.  
Therefore, Theorem \ref{th:JEMS} (cf. N\'emethi \cite{NJEMS}) and \cite[Theorem 5.4.2]{LN} imply that 
$$P_h(\mathbf{1})=-\mathfrak{sw}_{-h*\sigma_{can}}(M)-\frac{(K+2r_h)^2+|\mathcal{V}|}{8},$$
the $r_h$-normalized Seiberg--Witten invariant of $M$ associated with the class $h\in H$ (see Section \ref{ss:sw}).

The construction of the polynomial part has two main motivations:
\begin{enumerate}
 \item It gives a major tool in the computation of the Seiberg--Witten invariant for graphs with any number of nodes. Particular cases 
 were proved in \cite{BN} for one-variable Poincar\'e series, i.e. for graphs with only one node, and in \cite{LN} for graphs having at most two nodes.
 
 \item This will lead to a polynomial generalization of the Seiberg--Witten invariant for negative definite plumbed 3-manifolds. 
 One can expect a finer invariant both from topological and analytical point of view: more precise `connection' with the possible 
 analytic type of normal surface singularities which can be associated with the given 3-manifold.
\end{enumerate}

The main difficulty in the general case is that the chamber decomposition of the (reduced) Lipman cone can be cumbersome. Hence, one can 
not expect a division algorithm for each term of the numerator of $f_h(\mathbf{t}_{\mathcal{N}})$ separately, as in the case of \cite[Theorem 4.5.1]{LN}. 
However, thanks to the structure of the counting function given in the previous section,  we can reduce the construction to one- and two-variable cases.

First of all, we review the one- and two-variable cases from \cite{BN,LN} in a slightly more general setting, then we give the construction of the polynomial part for more variables in Section \ref{ss:polypart}. Finally, we illustrate the method on a graph with 
three nodes.

\subsection{Polynomial part in one-variable case}\label{ss:pcrevdim1}
Fix a vertex $v \in \mathcal{V}$ and consider a rational function  
$$
f_{H}(\mathbf{t}) = \frac{R_{H}(\mathbf{t})}{\prod_{i=1}^{m}(1-h_{i}\mathbf{t}^{\alpha_{i}})  }
$$
with coefficients in $\mathbb{Z}[H]$. In this subsection we consider $f_{H}(\mathbf{t})$ as function in $t_{v}$ and other variables $t_{u}$, $u\neq v$ are regarded as coefficients. Moreover, we assume that $R_{H}(\mathbf{t})$ and $1 - h_{i}\mathbf{t}^{\alpha_{i}}$ are supported on $\pi_{v}(L') \cap \mathbb{Q}_{\geq0} \langle E_{v} \rangle$ (as series in $t_{v}$).

\begin{lemma}{(cf. \cite[Lemma 7.0.2]{BN})}\label{lem:pcdim1}
There is a unique decomposition
\begin{equation}\label{Eq-56}
f_{H}(\mathbf{t}) = P_{H}^{v}(\mathbf{t}) + \frac{P'_{H}(\mathbf{t})}{\prod_{i=1}^{m}(1 - h_{i} \mathbf{t}^{\alpha_{i}})}
\end{equation}
such that as one-variable series in $t_{v}$
\begin{enumerate}
\item[(i)]
$P^{v}_{H}(\mathbf{t})$ is finitely supported on $\pi_{v}(L') \cap \mathbb{Q}_{\geq0} \langle E_{v} \rangle$,
 
\item[(ii)]
$P'_{H}(\mathbf{t})$ is finitely supported on $\pi_{v}(L') \cap \sum_{i=1}^{m}[0,1)\pi_{v}(\alpha_{i})$.
\end{enumerate}
Moreover, $\textnormal{pc}_{h}(\textnormal{T}[f_{H}(t_{v})]) = P^{v}_{h}(\mathbf{1})$, where $P^{v}_{H}(\mathbf{t}) = \sum_{h\in H} P^{v}_{h}(\mathbf{t}) \cdot h$.
\end{lemma}
We refer to $P^{v}_{h}(\mathbf{t})$ as $h$-equivariant polynomial part of $f_{H}(\mathbf{t})$ as function in $t_{v}$.

\begin{proof}
To show the uniqueness, note that the leading term of $P_H^{v}(\mathbf{t}) \prod_{i=1}^{m}(1 - h_{i}\mathbf{t}^{\alpha_{i}})$ as series in $t_{v}$ is supported in the region $\pi_{v}(L') \cap \big( \sum_{i=1}^{m} \pi_{v}(\alpha_{i}) + \mathbb{Q}_{\geq0} \langle E_{v} \rangle \big)$, which is disjoint from the support of $P'_{H}(\mathbf{t})$. Thus, if $f_{H}(\mathbf{t})=0$  then $P^{v}_{H}(\mathbf{t})$ and $P'_{H}(\mathbf{t})$ both vanish.

Considering $f_{H}(\mathbf{t})$ as rational fraction in $t_{v}$, division by remainder yields decomposition (\ref{Eq-56}).

For the last part, by Remark \ref{Rk-2} we have $\textnormal{pc}_{h} \left(\textnormal{T} \left[ \frac{P'_{H}(t_{v})}{\prod_{i=1}^{m}(1 - h_{i} t_{v}^{\alpha_{i}})} \right] \right) =0$, thus $\textnormal{pc}_{h}(\textnormal{T}[f_{H}(t_{v})]) = \textnormal{pc}_{h}(P^{v}_{H}(t_{v})) = P^{v}_{h}(\mathbf{1})$.
\end{proof}

\subsection{Polynomial part in two-variable case}\label{ss:pcrevdim2}
Fix vertices $v,w\in \mathcal{V}$ and introduce shorter notation $\pi_{vw}$ for the projection $\pi_{\{v,w\}}$. Consider the rational function
$$
f_{H}(\mathbf{t}) = \frac{R_{H}(\mathbf{t})}{\prod_{i=1}^{m} (1 - h_{i}\mathbf{t}^{\alpha_{i}})\prod_{j=1}^{n} (1 - g_{j}\mathbf{t}^{\beta_{j}})}
$$
with coefficients in $\mathbb{Z}[H]$. In this subsection we consider $f_{H}(\mathbf{t})$ as a function in $t_{v}$ and $t_{w}$, and the other variables $t_{u}$, $u\neq v,w$ are taken as coefficients. We impose the following conditions on $f_{H}(\mathbf{t})$.  
There are $\alpha, \beta\in \mathbb{R}_{>0}\langle E_{v},E_{w}\rangle$ such that $\pi_{vw}(\alpha_{i})$ and $\pi_{vw}(\beta_{j})$ are positive integer multiples of fixed vectors $\alpha$ and $\beta$, respectively. Moreover, about their position we assume that $\beta \in \mathbb{R}_{>0}\langle E_{w},\alpha\rangle$ and $\alpha \in \mathbb{R}_{>0}\langle \beta, E_{v}\rangle$. Finally, $R_{H}(\mathbf{t})$, $1-h_{i}\mathbf{t}^{\alpha_{i}}$, $1-g_{j}\mathbf{t}^{\beta_{j}}$ are supported on $L'$, moreover the projection $R_{H}(\mathbf{t})$ has support in $\mathbb{Q}_{\geq0}\langle\alpha,\beta\rangle$ (as series in $t_{v}$ and $t_{w}$).

\begin{lemma}\label{Lm-9}(cf. \cite[Theorem 4.5.1]{LN})\label{lem:pcdim2}
There is a unique decomposition
\begin{align}
\label{Eq-55}
\frac{R_{H}(\mathbf{t})}{\prod_{i=1}^{m} (1 - h_{i}\mathbf{t}^{\alpha_{i}})\prod_{j=1}^{n} (1 - g_{j}\mathbf{t}^{\beta_{j}})}
={}&
P^{vw}_{H}(\mathbf{t})
+
\frac{P'_{H}(\mathbf{t})}{\prod_{i=1}^{m}(1-h_{i}\mathbf{t}^{\alpha_{i}})}
+
\frac{P''_{H}(\mathbf{t})}{\prod_{j=1}^{n}(1-g_{j}\mathbf{t}^{\beta_{j}})}
\\
&{}+
\frac{P'''_{H}(\mathbf{t})}{\prod_{i=1}^{m} (1 - h_{i}\mathbf{t}^{\alpha_{i}})\prod_{j=1}^{n} (1 - g_{j}\mathbf{t}^{\beta_{j}})},
\notag
\end{align}
such that as two-variable series in $t_{v}$ and $t_{w}$ 
\begin{enumerate}
\item[(i)] 
$P^{vw}_{H}(\mathbf{t})$ is finitely supported on $\pi_{vw}(L')\setminus \mathbb{Q}_{<0}\langle E_{v},E_{w}\rangle$,

\item[(ii)]
$P'_{H}(\mathbf{t})$ is finitely  supported on $\pi_{vw}(L') \cap \big( Q_{\leq 0 } \langle E_{v} \rangle + \sum_{i=1}^{m} [0,1)\pi_{vw}(\alpha_{i})\big)$,

\item[(iii)]
$P''_{H}(\mathbf{t})$ is finitely  supported on $\pi_{vw}(L') \cap \big( Q_{\leq 0 } \langle E_{w} \rangle + \sum_{j=1}^{n} [0,1)\pi_{vw}(\beta_{j})\big)$,

\item[(iv)]
$P'''_{H}(\mathbf{t})$ is finitely  supported on $\pi_{vw}(L') \cap \big(  \sum_{i=1}^{n} [0,1)\pi_{vw}(\alpha_{i}) + \sum_{j=1}^{n} [0,1)\pi_{vw}(\beta_{j}) \big)$. 
\end{enumerate} 
 Furthermore,
$
\textnormal{pc}^{\mathfrak{c}}_{h}(\textnormal{T}[f_{H}(\mathbf{t}_{vw})]) = P^{vw}_{h}(\mathbf{1}),
$ 
where $P^{vw}_{H}(\mathbf{t}) = \sum_{h\in H}P^{vw}_{h}(\mathbf{t})\cdot h$ and $\mathfrak{c} = \mathbb{R}_{>0}\langle \alpha, \beta \rangle$.
\end{lemma}

We refer to $P^{vw}_{h}(\mathbf{t})$ as $h$-equivariant polynomial part of $f_{H}(\mathbf{t})$  as function in $t_{v}$ and $t_{w}$.

\begin{proof}
For uniqueness we show that if $R_{H}(\mathbf{t}) = 0$ then each term on the right hand side of (\ref{Eq-55}) is zero individually. 
We take the common denominator and we introduce the ordering on $\pi_{vw}(L')$ induced by $t_{v}>t_{w}$.
Note that as series in $t_{v},t_{w}$ the leading term of $P^{vw}_{H}(\mathbf{t}) \prod_{i=1}^{m}(1 - h_{i} \mathbf{t}^{\alpha_{i}}) \prod_{j=1}^{n} (1 - g_{j} \mathbf{t}^{\beta_{j}})$ is supported on $\pi_{vw}(L') \setminus \big( \sum_{i=1}^{m} \pi_{vw}(\alpha_{i}) + \sum_{j=1}^{n} \pi_{vw}(\beta_{j}) + \mathbb{Q}_{<0}\langle E_{v},E_{w}\rangle \big)$, which region is disjoint from
  $\pi_{vw}(L') \cap \big(  \sum_{i=1}^{m} [0,1)\pi_{vw}(\alpha_{i}) + \sum_{j=1}^{n} [0,1]\pi_{vw}(\beta_{j})  + Q_{\leq0} \langle E_{v} \rangle \big)  $,  $\pi_{vw}(L') \cap \big(\sum_{i=1}^{m} [0,1]\pi_{vw}(\alpha_{i}) + \sum_{j=1}^{n} [0,1)\pi_{vw}(\beta_{j})  + Q_{\leq0} \langle E_{w} \rangle  \big)$ and $\pi_{vw}(L') \cap \big( \sum_{i=1}^{m} [0,1)\pi_{vw}(\alpha_{i}) + \sum_{j=1}^{n}[0,1)\pi_{vw}(\beta_{j}) \big)$, where  $P'_{H}(\mathbf{t}) \prod_{j=1}^{n}(1 - g_{j}\mathbf{t}^{\beta_{j}})$, $P''_{H}(\mathbf{t}) \prod_{i=1}^{m} (1 - h_{i} \mathbf{t}^{\alpha_{i}})$ and $P'''_{H}(\mathbf{t})$ are supported, respectively. Thus $P^{vw}_{H}(\mathbf{t})$ must vanish. Moreover, the leading term of 
$P'_{H}(\mathbf{t}) \prod_{j=1}^{n}(1 - g_{j}\mathbf{t}^{\beta_{j}})$ (as series in $t_{v}$ and $t_{w}$) is supported in $\pi_{vw}(L') \cap \big( \sum_{j=1}^{n} \pi_{vw}(\beta_{j}) + Q_{\leq0} \langle E_{v} \rangle  + \sum_{i=1}^{m} [0,1)\pi_{vw}(\alpha_{i}) \big) $, which is disjoint from the support of $P''_{H}(\mathbf{t}) \prod_{i=1}^{m} (1 - h_{i} \mathbf{t}^{\alpha_{i}})$ and $P'''_{H}(\mathbf{t})$, hence $P'_{H}(\mathbf{t}) = 0$. By symmetry $P''_{H}(\mathbf{t})=0$ too, which implies $P'''_{H}(\mathbf{t})=0$.

For the existence we proceed as follows. We can write 
$
R_{H}(\mathbf{t}) = \rho_{1}  \prod_{i=1}^{m}(1-h_{i}\mathbf{t}^{\alpha_{i}}) \prod_{j=1}^{n}(1-g_{j}\mathbf{t}^{\beta_{j}}) + \rho_{2}  \prod_{i=1}^{m}(1-h_{i}\mathbf{t}^{\alpha_{i}}) + \rho_{3} \prod_{j=1}^{n}(1-g_{j}\mathbf{t}^{\beta_{j}}) + \rho_{4}
$ 
such that the support of $\rho_{2}$ is in $\pi_{vw}(L') \cap \big( \mathbb{Q}_{\geq0} \langle \alpha \rangle  + \sum_{j=1}^{n} [0,1)\pi_{vw}(\beta_{j}) \big)$, of $\rho_{3}$ is in $\pi_{vw}(L') \cap \big( \mathbb{Q}_{\geq0} \langle \beta \rangle +  \sum_{i=1}^{m} [0,1)\pi_{vw}(\alpha_{i}) \big)$ and  of $\rho_{4}$ is in $\pi_{vw}(L') \cap \big( \sum_{i=1}^{m} [0,1)\pi_{vw}(\alpha_{i}) +  \sum_{j=1}^{n} [0,1)\pi_{vw}(\beta_{j}) \big)$.
Moreover, as series in $t_{v}$ we can write $\rho_{2}= \rho_{2}'\prod_{j=1}(1-g_{j}\mathbf{t}^{\beta_{j}}) + \rho_{2}''$ such that $\rho_{2}''$ has support in $\pi_{vw}(L') \cap \big( \sum_{j=1}^{n} [0,1)\pi_{vw}(\beta_{j}) + \mathbb{Q}_{\leq 0} \langle E_{w} \rangle \big)$. Similarly, considered as series in $t_{w}$ we can write $\rho_{3} = \rho_{3}'  \prod_{i=1}^{m}(1-h_{i}\mathbf{t}^{\alpha_{i}}) + \rho_{3}''$ such that $\rho''_{3}$ has support in $\pi_{vw}(L') \cap \big( \sum_{i=1}^{m} [0,1)\pi_{vw}(\alpha_{i}) + \mathbb{Q}_{\leq0} \langle E_{v} \rangle \big)$. Thus, we get $P^{vw}_{H} = \rho_{1}+\rho_{2} + \rho_{3}$, $P'_{H}=\rho''_{3}$, $P''_{H}=\rho''_{2}$ and $P'''_{H} = \rho_{4}$.

For the last part, note that the periodic constant of the last term on the right hand side of (\ref{Eq-55}) vanishes by Remark \ref{Rk-2}. 
Directly from the definition of the two-variable periodic constant we can see that
$\textnormal{pc}^{\mathfrak{c}}_{h}\left( \textnormal{T} \left[ \frac{P'_{H}(\mathbf{t}_{vw})}{\prod_{i=1}^{m}(1-h_{i}\mathbf{t}_{vw}^{\alpha_{i}})} \right] \right) 
= 
\textnormal{pc}^{\mathfrak{c}}_{h}\left( \textnormal{T} \left[ \frac{P'_{H}(t_{w})}{\prod_{i=1}^{m}(1-h_{i}t_{w}^{\alpha_{i}})} \right] \right)$  
and  
$\textnormal{pc}^{\mathfrak{c}}_{h}\left( \textnormal{T} \left[ \frac{P''_{H}(\mathbf{t}_{vw})}{\prod_{j=1}^{n}(1-g_{j}\mathbf{t}_{vw}^{\beta_{j}})} \right] \right) 
= 
\textnormal{pc}^{\mathfrak{c}}_{h}\left( \textnormal{T} \left[ \frac{P''_{H}(t_{v})}{\prod_{j=1}^{n}(1-g_{j}t_{v}^{\beta_{j}})} \right] \right)$, moreover they vanish by Remark \ref{Rk-2}. 
Thus, we get $\textnormal{pc}^{\mathfrak{c}}_{h}(\textnormal{T}[f_{H}(\mathbf{t}_{vw})]) = \textnormal{pc}^{\mathfrak{c}}_{h}(P^{vw}_{H}(\mathbf{t}_{vw})) = P^{vw}_{h}(\mathbf{1})$.
\end{proof}

\subsection{Polynomial part for more variables}\label{ss:polypart}
We start with the decomposition 
$$
\textnormal{pc}(Z_{h}(\mathbf{t}_{\mathcal{N}})) 
= 
\sum_{\overline{vw}\textnormal{ edge of }\mathcal{T}^{orb}} \textnormal{pc}(Z_{h}(\mathbf{t}_{vw}))
- 
\sum_{v\in \mathcal{N}} (\delta_{v}^{orb}-1)  \textnormal{pc}(Z_{h}(t_{v}))
$$ 
of the periodic constant from Corollary \ref{Cor-1} and we will lift this relation to the level of the rational function. 
The strategy is to construct the corresponding multivariable polynomials which are polynomial parts of $f_h(\mathbf{t}_\mathcal{N})$ considered as function in $t_v$ and $\mathbf{t}_{vw}$.

By Section \ref{ss:pcrevdim1} we can consider decomposition 
$f_{h}(\mathbf{t}_{\mathcal{N}}) =  P^{v}(\mathbf{t}_{\mathcal{N}}) + f_{h}^{v,-}(\mathbf{t}_{\mathcal{N}})$, where $P^{v}_{h}(\mathbf{t}_\mathcal{N})$ is the polynomial part of $f_{h}(\mathbf{t}_{\mathcal{N}})$ as function in $t_{v}$, hence $P^{v}_{h}(\mathbf{1})=\textnormal{pc}(Z_{h}(t_{v}))$.

Similarly, for any edge $\overline{vw}$ of the orbifold graph $\mathcal{T}^{orb}$ we want to consider the decomposition $f_{h}(\mathbf{t}_{\mathcal{N}}) = P^{vw}_{h}(\mathbf{t}_{\mathcal{N}}) + f_{h}^{vw,-}(\mathbf{t}_{\mathcal{N}})$ by Lemma \ref{Lm-9}, where $P^{vw}_{h}(\mathbf{t}_{\mathcal{N}})$ is the polynomial part of $f_{h}(\mathbf{t}_{\mathcal{N}})$  as function in $\mathbf{t}_{vw}$, thus  
$P^{vw}_{h}(\mathbf{1})=\textnormal{pc}(Z_{h}(\mathbf{t}_{vw}))$.
We check the requirements of the lemma as follows. Removing the path connecting $v$ and $w$ from $\mathcal{T}$ we get two subtrees $\mathcal{T}_{v}$ and $\mathcal{T}_{w}$ containing $v$ and $w$, respectively. Then $\pi_{vw}(E^{*}_{u})$ is a positive multiple of $\pi_{vw}(E^{*}_{v})$  (resp. $\pi_{vw}(E^{*}_{w})$) for any vertex $u$ of $\mathcal{T}_{v}$ (resp. $\mathcal{T}_{w}$). 
Indeed, the graph $\mathcal{T}_{v}\setminus v$ decomposes into union of trees and denote $\mathcal{T}'$ one of its component with vertex set $\mathcal{V}'$. There is a unique vertex $v'\in \mathcal{V}'$ such that $\overline{vv'}$ is an edge of $\mathcal{T}$, hence set $E'_{v'}=\pi_{vw}(E^{*}_{v})$ and $E'_{u}=0$ for $u\in \mathcal{V}'\setminus v'$. Then, from the projected relations associated with vertices of $\mathcal{T}'$ follows that $[\pi_{vw}(E^{*}_{u})]_{u\in \mathcal{V}'} = (A_{\mathcal{V}'})^{-1} \cdot [E'_{u}]_{u\in \mathcal{V}'}$, where $A_{\mathcal{V}'}$ is the positive definite matrix associated with the subtree $\mathcal{T}'$. In particular, the inverse of $A_{\mathcal{V}'}$ has positive entries, thus $\pi_{vw}(E^{*}_{u})$ is a positive multiple of $\pi_{vw}(E^{*}_{v})$.

We can write $
f_{h}(\mathbf{t}_{\mathcal{N}}) = \sum_{\overline{vw}\textnormal{ edge of }\mathcal{T}^{orb}} f_{h}(\mathbf{t}_{\mathcal{N}}) - \sum_{v\in \mathcal{N}} (\delta_{v}^{orb}-1)  f_{h}(\mathbf{t}_{\mathcal{N}})$ in order 
to get a decomposition $f_h(\mathbf{t}_{\mathcal{N}})=P_h(\mathbf{t}_{\mathcal{N}})+f^{-}_h(\mathbf{t}_{\mathcal{N}})$ such that
\begin{equation}\label{eq:polypart}
P_{h}(\mathbf{t}_{\mathcal{N}}) = \sum_{\overline{vw}\textnormal{ edge of }\mathcal{T}^{orb}} P_{h}^{vw}(\mathbf{t}_{\mathcal{N}}) - \sum_{v\in \mathcal{N}} (\delta_{v}^{orb}-1) P^{v}_{h}(\mathbf{t}_{\mathcal{N}})
\end{equation}
is Laurent polynomial with $P_h(\mathbf{1})=\textnormal{pc}^{\pi_{\mathcal{N}}(S_{\mathbb{R}}')}(Z_h(\mathbf{t}_{\mathcal{N}}))$, thus the periodic constant of the rational function $f_h^-(\mathbf{t}_{\mathcal{N}}))$ is zero. Moreover, it has the form 
\begin{equation}\label{eq:f-}
f_{h}^{-}(\mathbf{t}_{\mathcal{N}}) = \sum_{\overline{vw}\textnormal{ edge of }\mathcal{T}^{orb}} f_{h}^{vw,-}(\mathbf{t}_{\mathcal{N}}) - \sum_{v\in \mathcal{N}} (\delta_{v}^{orb}-1)  f_{h}^{v,-}(\mathbf{t}_{\mathcal{N}}), 
\end{equation}
which is unique by Lemma \ref{lem:pcdim1} and Lemma \ref{lem:pcdim2}.

We summarize the above results in the following theorem.

\begin{thm}
 For any $h\in H$, there exists a unique decomposition 
 $$f_h(\mathbf{t}_{\mathcal{N}})=P_h(\mathbf{t}_{\mathcal{N}})+f^{-}_h(\mathbf{t}_{\mathcal{N}}),$$
 such that $P_h(\mathbf{t}_{\mathcal{N}})$ is a Laurent polynomial, supported on $\pi_{\mathcal{N}}(L')\setminus \mathbb{Q}_{<0}\langle E_v\rangle_{v\in \mathcal{N}}$, with 
 $P_h(\mathbf{1})=\textnormal{pc}^{\pi_{\mathcal{N}}(S_{\mathbb{R}}')}(Z_h(\mathbf{t}_{\mathcal{N}}))$
 and $f^{-}_h(\mathbf{t}_{\mathcal{N}})$ is rational, satisfying (\ref{eq:f-}).
\end{thm}

\subsection{Example}\label{sec:example}
We consider the following graph $\mathcal{T}$:
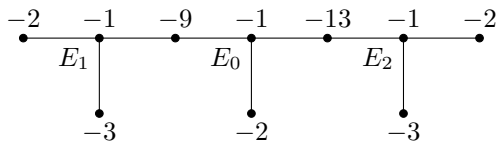
\begin{center}
\begin{tikzpicture}[scale=.5]
\coordinate (v11) at (0,0);
\draw  node[above] at (v11) {$-2$};
\draw[fill] (v11) circle (0.1);

\coordinate (v1) at (2,0);
\draw node[above] at (v1) {$-1$};
\draw[fill] (v1) circle (0.1);

\coordinate (v12) at (2,-2);
\draw   node[below] at (v12) {$-3$};
\draw[fill] (v12) circle (0.1);

\coordinate (u1) at (4,0);
\draw   node[above] at (u1) {$-9$};
\draw[fill] (u1) circle (0.1);

\coordinate (v0) at (6,0);
\draw   node[above] at (v0) {$-1$};
\draw[fill] (v0) circle (0.1);

\coordinate (v01) at (6,-2);
\draw   node[below] at (v01) {$-2$};
\draw[fill] (v01) circle (0.1);

\coordinate (u2) at (8,0);
\draw   node[above] at (u2) {$-13$};
\draw[fill] (u2) circle (0.1);

\coordinate (v2) at (10,0);
\draw   node[above] at (v2) {$-1$};
\draw[fill] (v2) circle (0.1);

\coordinate (v22) at (10,-2);
\draw   node[below] at (v22) {$-3$};
\draw[fill] (v22) circle (0.1);

\coordinate (v21) at (12,0);
\draw   node[above] at (v21) {$-2$};
\draw[fill] (v21) circle (0.1);

\draw[-] (v11) -- (v1);
\draw[-] (v12) -- (v1);
\draw[-] (u1) -- (v1);
\draw[-] (u1) -- (v0);
\draw[-] (v01) -- (v0);
\draw[-] (u2) -- (v0);
\draw[-] (u2) -- (v2);
\draw[-] (v21) -- (v2);
\draw[-] (v22) -- (v2);

\draw node[below left] at (v1) {$E_{1}$};
\draw node[below left] at (v0) {$E_{0}$};
\draw node[below left] at (v2) {$E_{2}$};
\end{tikzpicture}
\end{center}
In this case the group $H$ is trivial, hence 
\begin{gather*}
f_{H}(\mathbf{t}_{\mathcal{N}}) 
=
\frac{
( 1 - t_{0}^{84} t_{1}^{186} t_{2}^{72} ) 
( 1 - t_{0}^{42} t_{1}^{84} t_{2}^{36} )
( 1 - t_{0}^{36} t_{1}^{72} t_{2}^{42} )
}
{
( 1 - t_{0}^{42} t_{1}^{93} t_{2}^{36} )
( 1 - t_{0}^{28} t_{1}^{62} t_{2}^{24} )
( 1 - t_{0}^{21} t_{1}^{42} t_{2}^{18} )
( 1 - t_{0}^{18} t_{1}^{36} t_{2}^{21} )
( 1 - t_{0}^{12} t_{1}^{24} t_{2}^{14} )
}
\\
=
\frac{
( 1 + t_{0}^{42} t_{1}^{93} t_{2}^{36} )
( 1 + t_{0}^{21} t_{1}^{42} t_{2}^{18} )
( 1 + t_{0}^{18} t_{1}^{36} t_{2}^{21} )
}
{
( 1 - t_{0}^{28} t_{1}^{62} t_{2}^{24} )
( 1 - t_{0}^{12} t_{1}^{24} t_{2}^{14} )
}
\\
=
\frac{
1 +  t_{0}^{42} t_{1}^{93} t_{2}^{36} + t_{0}^{21} t_{1}^{42} t_{2}^{18}  +  t_{0}^{18} t_{1}^{36} t_{2}^{21} 
+ 
t_{0}^{63} t_{1}^{135} t_{2}^{54} + t_{0}^{60} t_{1}^{129} t_{2}^{57}  + t_{0}^{39} t_{1}^{78} t_{2}^{39}  
+
t_{0}^{81} t_{1}^{171} t_{2}^{75}
}
{
( 1 - t_{0}^{28} t_{1}^{62} t_{2}^{24} )
( 1 - t_{0}^{12} t_{1}^{24} t_{2}^{14} )
}.
\end{gather*}
The orbifold graph $\mathcal{T}^{orb}$ has three vertices $E^{*}_{1}$, $E^{*}_{0}$, $E^{*}_{2}$ and two edges connecting $E^{*}_{0}$ to $E^{*}_{1}$ and $E^{*}_{2}$, respectively.
According to formula (\ref{eq:polypart}) the polynomial part of $f_{H}(\mathbf{t}_{\mathcal{N}})$ has of form $P(\mathbf{t}_{\mathcal{N}}) = P^{01}(\mathbf{t}_{\mathcal{N}}) + P^{02}(\mathbf{t}_{\mathcal{N}}) - P^{0}(\mathbf{t}_{\mathcal{N}})$. 
For simplicity, we only discuss in more details the decomposition of the summand $\frac{t_{0}^{63} t_{1}^{135} t_{2}^{54}}{( 1 - t_{0}^{28} t_{1}^{62} t_{2}^{24} )( 1 - t_{0}^{12} t_{1}^{24} t_{2}^{14} )}$ and its contributions to $P^{01}$, $P^{02}$ and $P^{0}$.
\begin{enumerate}
\item 
Considering as function in variables $t_{0}$ and $t_{1}$ its decomposition by (\ref{Eq-55}) yields
$$
t_{0}^{23} t_{1}^{49} t_{2}^{16} + t_{0}^{11} t_{1}^{25} t_{2}^{2} + t_{0}^{-1} t_{1} t_{2}^{-12} 
- \frac{ t_{0}^{-1} t_{1} t_{2}^{-12}  } {   1 - t_{0}^{12} t_{1}^{24} t_{2}^{14}  } - \frac{ t_{0}^{23} t_{1}^{49} t_{2}^{16} }{  1 - t_{0}^{28} t_{1}^{62} t_{2}^{24} }
+
\frac{ t_{0}^{23} t_{1}^{49} t_{2}^{16} } {( 1 - t_{0}^{28} t_{1}^{62} t_{2}^{24} ) ( 1 - t_{0}^{12} t_{1}^{24} t_{2}^{14} )}, 
$$
hence it contributes $t_{0}^{23} t_{1}^{49} t_{2}^{16} + t_{0}^{11} t_{1}^{25} t_{2}^{2} + t_{0}^{-1} t_{1} t_{2}^{-12} $ to $P^{01}$.

\item  
Considering as function in variables $t_{0}$ and $t_{2}$ its decomposition by (\ref{Eq-55}) gives
$$
t_{0}^{23} t_{1}^{49} t_{2}^{16} + t_{0}^{11} t_{1}^{25} t_{2}^{2}
- \frac{ t_{0}^{11} t_{1}^{25} t_{2}^{2}  +  t_{0}^{7} t_{1}^{11} t_{2}^{6} } {   1 - t_{0}^{12} t_{1}^{24} t_{2}^{14}  } +
\frac{ t_{0}^{7} t_{1}^{11} t_{2}^{6} } {( 1 - t_{0}^{28} t_{1}^{62} t_{2}^{24} ) ( 1 - t_{0}^{12} t_{1}^{24} t_{2}^{14} )}, 
$$
hence it contributes $t_{0}^{23} t_{1}^{49} t_{2}^{16} + t_{0}^{11} t_{1}^{25} t_{2}^{2}$ to $P^{02}$.

\item 
Considering as one-variable function in $t_{0}$ its decomposition according to (\ref{Eq-56}) equals
$$
t_{0}^{23} t_{1}^{49} t_{2}^{16} + t_{0}^{11} t_{1}^{25} t_{2}^{2} +
\frac{ 
 t_{0}^{35} t_{1}^{73} t_{2}^{30} + t_{0}^{39} t_{1}^{87} t_{2}^{26} + t_{0}^{23} t_{1}^{49} t_{2}^{16} - t_{0}^{11} t_{1}^{25} t_{2}^{2} - t_{0}^{23} t_{1}^{49} t_{2}^{16}
}
{( 1 - t_{0}^{28} t_{1}^{62} t_{2}^{24} ) ( 1 - t_{0}^{12} t_{1}^{24} t_{2}^{14} ) 
},
$$
thus it contributes $t_{0}^{23} t_{1}^{49} t_{2}^{16} + t_{0}^{11} t_{1}^{25} t_{2}^{2}$ to $P^{0}$.
\end{enumerate}
By similar computations for each summand of $f_{H}$  we get that
\begin{align*}
P(t_{0},t_{1},t_{2}) 
={}& 
t_{0}^{41} t_{1}^{85} t_{2}^{37} +
t_{0}^{29} t_{1}^{61} t_{2}^{23} +
t_{0}^{23} t_{1}^{49} t_{2}^{16} +
t_{0}^{20} t_{1}^{43} t_{2}^{19} +
t_{0}^{17} t_{1}^{37} t_{2}^{9} +
t_{0}^{13} t_{1}^{23} t_{2}^{13} +
t_{0}^{11} t_{1}^{25} t_{2}^{2} 
\\
& + 
t_{0}^{8} t_{1}^{19} t_{2}^{5} +
t_{0}^{5} t_{1}^{13} t_{2}^{-5} + 
t_{0}^{2} t_{1}^{7} t_{2}^{-2} +
t_{0}^{} t_{1}^{-1} t_{2}^{-1} +
t_{0}^{-1} t_{1}^{} t_{2}^{-12} +
t_{0}^{-1} t_{1}^{-8} t_{2}^{},
\end{align*}
in particular the normalized Seiberg--Witten invariant of the associated manifold is equal with $\textnormal{pc}(
\textnormal{T}[f_{H}(\mathbf{t})]) = \textnormal{pc}(
\textnormal{T}[f_{H}(\mathbf{t}_{\mathcal{N}})]) = P(\mathbf{1}) = 13$.

\section{Surgery (recursion) formulas}\label{sec:surgform}

In this section we present a recursion formula for the counting function  (Theorem \ref{Thm-3}) using the interpretation of Section \ref{sec:coeff}. It is 
given in terms of a surgery on the graph $\mathcal{T}$. 
In particular, we discuss the recurrence on the level of periodic constants too, and compare 
with a surgery formula of Braun and N\'emethi \cite{BN} proved for the Seiberg--Witten invariants.
For sake of completeness, we recall first the formula from \cite{BN}.

\subsection{The Braun--N\'emethi surgery formula}

Although the formula is true for any vertex of $\mathcal{T}$, we restrict our attention to an end--vertex $u\in \mathcal{E}$. Denote 
the plumbed $4$-- and $3$--manifold associated with $\mathcal{T}\setminus u$ by $\widetilde X_u$ and $M_u$ respectively.

For any $h\in H$, the $spin^c$--structure $\sigma=h*\sigma_{can}\in \mathrm{Spin}^c(M)$ can be extended uniquely to $\widetilde\sigma\in \mathrm{Spin}^c(\widetilde X)$ 
such that 
$\widetilde\sigma=r_h*\widetilde\sigma_{can}$. 
We consider the projection 
$\pi^{(u)}:L'\to L'_{\mathcal{T}\setminus u}$, where $\pi^{(u)}(E^*_v)=E^*_v$ if  $v\in \mathcal{V}\setminus u$ and 
$\pi^{(u)}(E^*_u)=0$. 
Since the canonical $spin^c$--structure of $\widetilde X$ projects to the canonical $spin^c$--structure $\widetilde\sigma_{can,u}$ 
of $\widetilde X_u$, $\widetilde\sigma$ 
projects to $\widetilde\sigma_u:=\pi^{(u)}(r_h)*\widetilde\sigma_{can,u}$, whose restriction to the boundary $M_u$ is 
$\sigma_u=[\pi^{(u)}(r_h)]*\sigma_{can,u}$.
Then the main result of \cite{BN} is the following theorem.
\begin{thm}{(Braun--N\'emethi surgery formula)}\label{thm:BN}
\begin{align*}
 \mathfrak{sw}_{-h*\sigma_{can}}(M)+\frac{(K_{\mathcal{T}}+2r_h)^2+|\mathcal{V}|}{8}
 ={}&
 \mathfrak{sw}_{-[\pi^{(u)}(r_h)]*\sigma_{can,u}}(M_u)+\frac{(K_{\mathcal{T}\setminus u} + 2\pi^{(u)}(r_h))^2+
 |\mathcal{V}\setminus u|}{8}\\
 &- \mathrm{pc}(Z_{h}^u), 
\end{align*}
where $Z_{h}^u$ is the one-variable series $Z_h|_{t_v=1,v\neq u}$.
\end{thm}

\subsection{Recursion for the counting function and its quasipolynomial}
\label{ss:recurquasi}

To obtain a recursion for the quasipolynomial and the periodic constant we use partial fraction decomposition and projections.
Fix an end-vertex $u\in \mathcal{E}$ and denote $u'\in \mathcal{V}$ its unique neighboring vertex. 
Let $\mathcal{V}'=\mathcal{V}\setminus u$ and  $\mathcal{E}' = u' \cup \mathcal{E} \setminus u$ the vertex and end-vertex set of $\mathcal{T}\setminus u$.
We group non-empty subsets of $\mathcal{V}$ as follows: 
\begin{enumerate}
\item subsets $\mathcal{I}$ such that $u\in \mathcal{I}$ and $\mathcal{I}\setminus u\neq \emptyset$, 
\item subsets $\mathcal{I}'$ such that $u\notin \mathcal{I}'$, 
\item the subset $\{u\}$.
\end{enumerate}
We use notation $\mathcal{I}'=\mathcal{I}\setminus u$ for any subset $\mathcal{I}$ belonging to the first group, and write $C_{H}^{\mathcal{T}}$ and $f_{H}^{\mathcal{T}}$ emphasizing their dependence on $\mathcal{T}$.
We decompose the function $C^{\mathcal{T}}_{H}$ given in (\ref{Eq-40})  according to the above grouping of subsets
\begin{gather}
\label{Eq-20}
C_{H}^{\mathcal{T}}(x) 
=
\sum_{\substack{u \in \mathcal{I} \subseteq \mathcal{V} \\ |\mathcal{I}|>1}} (-1)^{|\mathcal{I}|-1} 
\Coeff \left(  \textnormal{T}\left[ f^{\mathcal{T}}_{H}(\mathbf{t}_{\mathcal{I}}) \prod_{i\in \mathcal{I}}\frac{\mathbf{t}_{\mathcal{I}}^{E_{i}}}{1-\mathbf{t}_{\mathcal{I}}^{E_{i}}}\right], \mathbf{t}_{\mathcal{I}}^{x} \right)
+
\\
\sum_{\mathcal{I}' \subseteq \mathcal{V}'} (-1)^{|\mathcal{I}'|-1} 
\Coeff \left(  \textnormal{T}\left[ f^{\mathcal{T}}_{H}(\mathbf{t}_{\mathcal{I}'}) \prod_{i\in \mathcal{I}'}\frac{\mathbf{t}_{\mathcal{I'}}^{E_{i}}}{1-\mathbf{t}_{\mathcal{I}'}^{E_{i}}}\right], \mathbf{t}_{\mathcal{I}'}^{x} \right)
+\Coeff \left(  \textnormal{T}\left[ f^{\mathcal{T}}_{H}(t_{u}) \frac{t_{u}}{1-t_{u}}\right], t_{u}^{x} \right).
\notag
\end{gather}
In general, for $z=x+y\in L'$ and $h_{z}=h_{x}h_{y} \in H$ we have the following partial fraction decomposition
$$
\frac{1}{(1 - h_{x}\mathbf{t}^{x})(1 - h_{y}\mathbf{t}^{y})} 
=
\frac{1}{(1 - h_{y}\mathbf{t}^{y})(1 - h_{z}\mathbf{t}^{z})} + \frac{1}{(1 - h_{x}\mathbf{t}^{x})(1 - h_{z}\mathbf{t}^{z})}  -\frac{1}{(1 - h_{z}\mathbf{t}^{z})}. 
$$
In particular, for $x=A_{uu}\,\pi_{\mathcal{I}} (E^{*}_{u})$, $y=-E_{u}$, $z = \pi_{\mathcal{I}} (E^{*}_{u'})$ 
it yields
\begin{gather*}
\frac{ \mathbf{t}_{\mathcal{I}}^{E_{u}} }{ \big(1 - \mathbf{t}_{\mathcal{I}}^{E_{u}} \big) \big(1 - h_{u}\mathbf{t}_{\mathcal{I}}^{E^{*}_{u}} \big) }
=
 -\frac{  
 \sum_{k=0}^{A_{uu}-1} h_{u}^{k} \mathbf{t}_{\mathcal{I}}^{kE_{u}^{*}}
}{ \big( 1 - \mathbf{t}_{\mathcal{I}}^{-E_{u}} \big) \big( 1 - h_{u}^{A_{uu}}\mathbf{t}_{\mathcal{I}}^{A_{uu} E^{*}_{u}} \big) }
\\
=
\frac{  
\mathbf{t}_{\mathcal{I}}^{E_{u}}
\sum_{k=0}^{A_{uu}-1} h_{u}^{k} \mathbf{t}_{\mathcal{I}}^{k E^{*}_{u}}
 }{ \big(1 - \mathbf{t}_{\mathcal{I}}^{E_{u}} \big) \big(1 - h_{u'}\mathbf{t}_{\mathcal{I}}^{E^{*}_{u'}} \big) }
-
\frac{  1 }{ \big( 1 - h_{u}\mathbf{t}_{\mathcal{I}}^{E^{*}_{u}} \big) \big( 1 - h_{u'}\mathbf{t}_{\mathcal{I}}^{E^{*}_{u'}} \big) }
+
\frac{  
\sum_{k=0}^{A_{uu}-1} h_{u}^{k}\mathbf{t}_{\mathcal{I}}^{k E^{*}_{u}}
}{ 1 - h_{u'}\mathbf{t}_{\mathcal{I}}^{E^{*}_{u'}} }.
\end{gather*} 
If we introduce notation $ \varphi_{\mathcal{I}} =  \prod_{v\in \mathcal{V}'} \big( 1 - h_{v} \mathbf{t}_{\mathcal{I}}^{E^{*}_{v}} \big)^{\delta_{v}-2} 
\prod_{i \in \mathcal{I}'} \frac{ \mathbf{t}_{\mathcal{I}}^{E_{i}}  }{ 1 - \mathbf{t}_{\mathcal{I}}^{E_{i}} }$ then the above 
relation gives
\begin{gather}
\label{Eq-34}
\Coeff \left(  \textnormal{T}\left[ f^{\mathcal{T}}_{H}(\mathbf{t}_{\mathcal{I}}) \prod_{i\in \mathcal{I}}\frac{\mathbf{t}_{\mathcal{I}}^{E_{i}}}{1-\mathbf{t}_{\mathcal{I}}^{E_{i}}}\right], \mathbf{t}_{\mathcal{I}}^{x} \right)
=
\Coeff \Bigg(  \textnormal{T} \Bigg[ 
\frac{  \mathbf{t}_{\mathcal{I}}^{E_{u}} \sum_{k=0}^{A_{uu}-1} h_{u}^{k}\mathbf{t}_{\mathcal{I}}^{kE^{*}_{u}}  }
{ \big( 1 - \mathbf{t}_{\mathcal{I}}^{E_{u}} \big) \big( 1 - h_{u'}\mathbf{t}_{\mathcal{I}}^{E^{*}_{u'}} \big) } 
\varphi_{\mathcal{I}} \Bigg],   
\mathbf{t}_{\mathcal{I}}^{x}
\Bigg)
\\
\notag
-
\Coeff\left( 
\textnormal{T} \left[
\frac{  \varphi_{\mathcal{I}} }{ \big( 1 - h_{u}\mathbf{t}_{\mathcal{I}}^{E^{*}_{u}} \big) \big( 1 - h_{u'}\mathbf{t}_{\mathcal{I}}^{E^{*}_{u'}} 
\big) } \right],\, 
\mathbf{t}_{\mathcal{I}}^{x} \right)
+
\Coeff\left( 
\textnormal{T} \left[ \frac{  \sum_{k=0}^{A_{uu}-1} h_{u}^{k}\mathbf{t}_{\mathcal{I}}^{kE^{*}_{u}} }{ 1 -
h_{u'}\mathbf{t}_{\mathcal{I}}^{E^{*}_{u'}} } \varphi_{\mathcal{I}} \right],\, 
\mathbf{t}_{\mathcal{I}}^{x} 
\right).
\notag
	\end{gather}

Let $\Psi'_{\mathcal{I}} =\{\pi_{\mathcal{I}} (E_{j}^{*}),\, E_{i} : j\in \mathcal{E}',\, i\in \mathcal{I}\setminus u\}$. We will apply Lemma \ref{Lm-Pr} to the right hand side of (\ref{Eq-34}) on a suitable affine subcone of the open polyhedral cone $\mathfrak{c}$ given by the following lemma.


\begin{lemma}\label{Lm-7}
There is a maximal dimensional open polyhedral cone $\mathfrak{c}$ in the Lipman cone $\mathcal{S}'_{\mathbb{R}}$ such that $\pi_{\mathcal{I}}(\mathfrak{c}) \cap \mathbb{R}_{\geq0}\langle\Psi'_{\mathcal{I}}\rangle = \emptyset$ for all $\mathcal{I}\subseteq \mathcal{V}$ with $u\in \mathcal{I}$.
\end{lemma}

\begin{proof}
Let $\mathfrak{c}_{\varepsilon} = \mathbb{R}_{\geq0}\langle \varepsilon E_{v} + (1-\varepsilon)E^{*}_{u}, E^{*}_{u}\rangle_{v\neq u} \cap \textnormal{int}\, (\mathcal{S}'_{\mathbb{R}})$.  
 We will show that there exists $\varepsilon>0$ such that $\mathfrak{c}=\mathfrak{c}_{\varepsilon}$ satisfies the desired relations. We distinguish two cases. If $u'\in \mathcal{I}$ then the hyperplane $\mathbb{R}\langle \pi_{\mathcal{I}}(E^{*}_{u'}), E_{v} : v\in \mathcal{I}\setminus \{u,u'\}\rangle $ separates $\mathbb{R}_{\geq0}\langle\Psi'_{\mathcal{I}}\rangle$ and $\pi_{\mathcal{I}}(\textnormal{int}\, (\mathcal{S}'_{\mathbb{R}}))$, hence we can choose any $\varepsilon>0$. If $u'\notin \mathcal{I}$ then we set $E'_{v}=E_{v}$ for $v\neq u,u'$ and $E'_{u'}=\pi_{\mathcal{U}}(E^{*}_{u})$, where $\mathcal{U}=\mathcal{V}\setminus u'$. Then we have relation $[\pi_{\mathcal{U}}(E^{*}_{v})]_{v\neq u} \cdot A_{\mathcal{T}\setminus u} = [E'_{v}]_{v\neq u}$, where $A_{\mathcal{T}\setminus u}$ is the positive definite matrix associated with the tree $\mathcal{T}\setminus u $. In particular, any $\pi_{\mathcal{I}}(E^{*}_{v})$ has positive coefficients in basis $\{\pi_{\mathcal{I}}(E^{*}_{u}),E_{v} \}_{v\in \mathcal{I}\setminus u}
$. Thus, we can find $\varepsilon=\varepsilon_{\mathcal{I}}>0$ such that  
$
\mathfrak{c}_{\varepsilon} \cap \mathbb{R}_{\geq0} \langle \pi_{\mathcal{I}}(E^{*}_{v}), E_{v} : v\neq u,\, i\in \mathcal{I}\setminus u\rangle = \emptyset.
$
Note that $\Psi'_{\mathcal{I}} \subset \{ \pi_{\mathcal{I}}(E^{*}_{v}), E_{i} : v\neq u, i\in \mathcal{I}\setminus u\} $, hence $\mathfrak{c}_{\varepsilon} \cap \mathbb{R}_{\geq 0}\langle \Psi'_{\mathcal{I}}\rangle = \emptyset$ holds too. Therefore, $\mathfrak{c} = \mathfrak{c}_{\varepsilon}$ for $\varepsilon = \min\{ \varepsilon_{\mathcal{I}} : \mathcal{I}\not\ni u'\}$ satisfies the requirements.
\end{proof}

First, we project along $E_{u}$ for $x \in L'$ in a suitable 
affine subcone of $\mathfrak{c}$. Note that $\{E^{*}_{v}, E_{u}\}_{v\neq u'}$ is a basis of $L'$, because $E^{*}_{u'} = A_{uu}E^{*}_{u}-E_{u}$.
Thus, Lemma \ref{Lm-Pr} gives
\begin{gather}
\Coeff \Bigg(  \textnormal{T} \Bigg[ 
\frac{  \mathbf{t}_{\mathcal{I}}^{E_{u}} \sum_{k=0}^{A_{uu}-1} h_{u}^{k} \mathbf{t}_{\mathcal{I}}^{kE^{*}_{u}}  }{ \big( 1 - 
\mathbf{t}_{\mathcal{I}}^{E_{u}} \big) \big( 1 - h_{u'} \mathbf{t}_{\mathcal{I}}^{E^{*}_{u'}} \big) } \varphi_{\mathcal{I}} \Bigg],   
\mathbf{t}_{\mathcal{I}}^{x}
\Bigg)
=
\Coeff \Bigg(  \textnormal{T} \Bigg[ 
\frac{  \sum_{k=0}^{A_{uu}-1} h_{u}^{k}\mathbf{t}_{\mathcal{I}'}^{kE^{*}_{u}}  }{  1 - h_{u'}\mathbf{t}_{\mathcal{I}'}^{E^{*}_{u'}}  } 
\pi_{\mathcal{I}'}(\varphi_{\mathcal{I}}) \Bigg],   
\mathbf{t}_{\mathcal{I}'}^{x}
\Bigg)
\notag\\
\label{Eq-35}
=\Coeff \Bigg(  \textnormal{T} \Bigg[ 
\frac{  \sum_{k=0}^{A_{uu}-1} h_{u}^{k}\mathbf{t}_{\mathcal{I}'}^{kE^{*}_{u}}  }{  1 - h_{u}^{A_{uu}}\mathbf{t}_{\mathcal{I}'}^{A_{uu}E^{*}_{u}}  } 
\pi_{\mathcal{I}'}(\varphi_{\mathcal{I}}) \Bigg],   
\mathbf{t}_{\mathcal{I}'}^{x}
\Bigg)
\\
=
\Coeff \Bigg(  \textnormal{T} \Bigg[ 
\frac{\pi_{\mathcal{I}'}(\varphi_{\mathcal{I}})    }{  1 - h_{u}\mathbf{t}_{\mathcal{I}'}^{E^{*}_{u}}  }  \Bigg],   
\mathbf{t}_{\mathcal{I}'}^{x}
\Bigg)
=
\Coeff \left(  \textnormal{T}\left[ f^{\mathcal{T}}_{H}(\mathbf{t}_{\mathcal{I'}}) \prod_{i\in \mathcal{I'}}\frac{\mathbf{t}_{\mathcal{I'}}^{E_{i}}}{1-\mathbf{t}_{\mathcal{I'}}^{E_{i}}}\right], \mathbf{t}_{\mathcal{I'}}^{x} \right),
\notag
\end{gather}
since  $\pi_{\mathcal{I}}(\mathfrak{c}) \subset \mathbb{R}_{\geq0}\langle E_{u}, 
\Psi'_{\mathcal{I}}\rangle \setminus \mathbb{R}_{\geq0}\langle \Psi'_{\mathcal{I}}\rangle$.

Moreover, for $x \in  L'$ in a suitable affine subcone of $\mathfrak{c}$  we have 
\begin{equation}\label{Eq-37}
\Coeff\left( 
\textnormal{T} \left[ \frac{  \sum_{k=0}^{A_{uu}-1} h_{u}^{k}\mathbf{t}_{\mathcal{I}}^{kE^{*}_{u}} }{ 1 -
h_{u'}\mathbf{t}_{\mathcal{I}}^{E^{*}_{u'}} } \varphi_{\mathcal{I}} \right],\, 
\mathbf{t}_{\mathcal{I}}^{x} 
\right) =0,
\end{equation}
since $\pi_{\mathcal{I}}(\mathfrak{c}) \cap \mathbb{R}_{\geq0}\langle \Psi'_{\mathcal{I}}\rangle = \emptyset$.
Therefore, (\ref{Eq-35}) and (\ref{Eq-37}) lead to cancellations and vanishings in (\ref{Eq-20}) for any $x\in  L'$ in a suitable affine subcone of $\mathfrak{c}$, thus
\begin{gather*}
C_{H}^{\mathcal{T}}(x) = 
-
\sum_{\substack{u\in \mathcal{I} \subseteq \mathcal{V}
\\ |\mathcal{I}|>1}} (-1)^{|\mathcal{I}|-1}
\Coeff\left( 
\textnormal{T} \left[
\frac{ f'_{H}(\mathbf{t}_{\mathcal{I}})
 }{  1 - h_{u}\mathbf{t}_{\mathcal{I}}^{E^{*}_{u}}   } 
\prod_{i \in \mathcal{I}'} \frac{ \mathbf{t}_{\mathcal{I}}^{E_{i}}  }{ 1 - \mathbf{t}_{\mathcal{I}}^{E_{i}} }
\right],\, 
\mathbf{t}_{\mathcal{I}}^{x} \right)
\\
+
\Coeff \left(  \textnormal{T}\left[ f^{\mathcal{T}}_{H}(t_{u}) \prod_{i\in \mathcal{I}}\frac{t_{u}}{1-t_{u}}\right], t_{u}^{x} \right),
\end{gather*}
where $f'_{H}(\mathbf{t}_{\mathcal{I}}) = \big( 1 - h_{u'} \mathbf{t}_{\mathcal{I}}^{E^{*}_{u'}} \big)^{-1} \prod_{v\in \mathcal{V}'} \big( 1 - h_{v} \mathbf{t}_{\mathcal{I}}^{E^{*}_{v}} \big)^{\delta_{v}-2}
=
 \prod_{v\in \mathcal{V}'} \big( 1 - h_{v} \mathbf{t}_{\mathcal{I}}^{E^{*}_{v}} \big)^{\delta_{v,\mathcal{V}'}-2}.
$ 

The projection along $E^{*}_{u}$ is a bit more involved, because the presence of the group element $h_{u}$. 
Denote $\pi^{(u)}_{\mathcal{I}'}: V_{\mathcal{I}} \to V_{\mathcal{I}'}$ the projection along 
$\pi_{\mathcal{I}}(E^{*}_{u})$ and use short notation $\pi^{(u)} = \pi^{(u)}_{\mathcal{V}'}$.
Note that we have commutation relation $\pi^{(u)}_{\mathcal{I}'} \pi_{\mathcal{I}} = \pi_{\mathcal{I}'} \pi^{(u)}$, moreover  
$\pi^{(u)} (E_{i}) = E_{i}$ for $i\neq u$.  
We do not have homomorphism from $H=H^{\mathcal{T}}=L'/L$ to $\widetilde{H}=H^{\mathcal{T}\setminus u}=\mathbb{Z}\langle\pi^{(u)}(E^{*}_{v})\rangle_{v\in \mathcal{V}'}/\mathbb{Z}\langle E_{v}\rangle_{v\in \mathcal{V}'}$, therefore we introduce an intermediate group $G:=L'/\mathbb{Z}\langle E_{v}\rangle_{v\in \mathcal{V}'}$, which admits natural homomorphisms $\phi:G \to H$  with kernel $\mathbb{Z}\langle E_{u} \rangle$ and $\widetilde{\phi}:G \to \widetilde{H}$ induced by projection along $E^{*}_{u}$. 
We denote by $\widehat{x}$  and  $\widehat{h}_{v}$  the classes of $x$ and $E^{*}_{v}$ in $G$, respectively. Moreover, there is a subgroup $H_{u}=\ker (\pi_{u}:G \to \pi_{u}(L'))$, which can be also considered as subgroup of $H$ and $\widetilde{H}$ via homomorphisms $\phi$ and $\widetilde{\phi}$. Then
\begin{equation}\label{Eq-46}
\textnormal{Coeff}\left( \textnormal{T} \left[
\frac{
\prod_{v\in \mathcal{V}'} (1 - \widehat{h}_{v} \mathbf{t}_{\mathcal{I}}^{E^{*}_{v}})^{\delta_{v,\mathcal{V}'}-2}
}{1 - \widehat{h}_{u} \mathbf{t}_{\mathcal{I}}^{E^{*}_{u}} }
\prod_{i\in \mathcal{I}'} \frac{ \mathbf{t}_{\mathcal{I}}^{E_{i}}
}{ 1 - \mathbf{t}_{\mathcal{I}}^{E_{i}}
}
\right], \mathbf{t}_{\mathcal{I}}^{x}
\right)
=
q_{\mathcal{I}} \cdot \widehat{x} + q'_{\mathcal{I}} \cdot \widehat{x}
\end{equation}
with $q_{\mathcal{I}}\in \mathbb{Z}$ and $q'_{\mathcal{I}} \in \mathbb{Z}[H_{u}] \setminus \mathbb{Z} \widehat{0}$, where $\widehat{0}$ is the identity element of $G$. Applying homomorphism $\phi$ to (\ref{Eq-46}) we get  
$$
\Coeff\left( 
\textnormal{T} \left[
\frac{ f'_{H}(\mathbf{t}_{\mathcal{I}})
 }{  1 - h_{u}\mathbf{t}_{\mathcal{I}}^{E^{*}_{u}}   } 
\prod_{i \in \mathcal{I}'} \frac{ \mathbf{t}_{\mathcal{I}}^{E_{i}}  }{ 1 - \mathbf{t}_{\mathcal{I}}^{E_{i}} }
\right],\, 
\mathbf{t}_{\mathcal{I}}^{x} \right)
=
q_{\mathcal{I}} \cdot [x] + q'_{\mathcal{I}} \cdot [x],
$$
thus by extracting coefficient of $[x]\in H$ from $C^{\mathcal{T}}_{H}$ we get
\begin{equation}\label{Eq-47}
C^{\mathcal{T}}(x) = \sum_{\substack{ u\in \mathcal{I} \subseteq \mathcal{V}\\ |\mathcal{I}|>1 }}(-1)^{|\mathcal{I}|} q_{\mathcal{I}} 
+
\frac{1}{|H|} \sum_{\rho \in \widehat{H}} \rho^{-1}\left( [x]^{-1} 
\Coeff \left(  \textnormal{T}\left[ f^{\mathcal{T}}_{H}(t_{u}) \prod_{i\in \mathcal{I}}\frac{t_{u}}{1-t_{u}}\right], t_{u}^{x} \right) \right). 
\end{equation}
Moreover, applying homomorphism $\widetilde{\phi}$ to (\ref{Eq-46}) one has
\begin{gather*}
q_{\mathcal{I}} \cdot [\pi^{(u)}(x)] + q'_{\mathcal{I}} \cdot [\pi^{(u)}(x)]
=
\Coeff\left( 
\textnormal{T} \left[
\frac{ f'_{\widetilde{H}}(\mathbf{t}_{\mathcal{I}})
 }{  1 - \mathbf{t}_{\mathcal{I}}^{E^{*}_{u}}   } 
\prod_{i \in \mathcal{I}'} \frac{ \mathbf{t}_{\mathcal{I}}^{E_{i}}  }{ 1 - \mathbf{t}_{\mathcal{I}}^{E_{i}} }
\right],\, 
\mathbf{t}_{\mathcal{I}}^{x} \right)
\\
=
\Coeff\left( 
\textnormal{T} \left[
 f^{\mathcal{T}\setminus u}_{\widetilde{H}}(\mathbf{t}_{\mathcal{I'}}) 
\prod_{i \in \mathcal{I}'} \frac{ \mathbf{t}_{\mathcal{I}'}^{E_{i}}  }{ 1 - \mathbf{t}_{\mathcal{I}'}^{E_{i}} }
\right],\, 
\mathbf{t}_{\mathcal{I}'}^{\pi^{(u)}(x)} \right),
\end{gather*}
where the second identity is deduced from Lemma \ref{Lm-Pr}, projecting along $E^{*}_{u}$. Hence, summing with respect to $\mathcal{I}'=\mathcal{I}\setminus u$ and extracting coefficient of $[\pi^{(u)}(x)]$ we get
\begin{equation}\label{Eq-48}
C^{\mathcal{T}\setminus u}(\pi^{(u)}(x))
=
\sum_{u\in \mathcal{I} \subseteq \mathcal{V}}(-1)^{|\mathcal{I}|} q_{\mathcal{I}}.
\end{equation}
Thus, we arrive at recursive relation
\begin{equation}\label{Eq-41b}
C^{\mathcal{T}}(x) = C^{\mathcal{T}\setminus u} (\pi^{(u)}(x)) 
+
\frac{1}{|H|} \sum_{\rho \in \widehat{H}} \rho^{-1}\left( [x]^{-1} 
\Coeff \left(  \textnormal{T}\left[ f^{\mathcal{T}}_{H}(t_{u}) \prod_{i\in \mathcal{I}}\frac{t_{u}}{1-t_{u}}\right], t_{u}^{x} \right) \right)  
\end{equation}
for any $x\in L'$ in a suitable affine subcone of $\mathfrak{c}$.

Denote $\mathcal{L}^{u}_{H}=\sum_{h\in H}\mathcal{L}^{u}_{h} \cdot h$ the quasipolynomial such that $\mathcal{L}^{u}_{H}(w) = 
\Coeff \big( \textnormal{T}[f^{\mathcal{T}}_{H}(t_{u})\frac{t_{u}}{1-t_{u}}], t^{w}_{u} \big)$ for $w \gg 0$. 
Then by relations (\ref{Eq-41b}) and (\ref{Eq-39}) we get the following theorem.

\begin{thm}\label{Thm-3}
For any $u\in \mathcal{E}$ end-vertex of $\mathcal{T}$ one has the recursion of quasipolynomials
\begin{equation*}
\mathcal{L}_{[x]}^{\mathcal{T}}(x) = \mathcal{L}^{\mathcal{T}\setminus u}_{[\pi^{(u)}(x)]} (\pi^{(u)}(x)) + \mathcal{L}^{u}_{[x]}(x_{u})
\end{equation*}
for all $x\in L'$.
\end{thm}

\begin{remark}
By Lemma \ref{Lm-5}(\ref{Lm-5ii}), the one-variable function $\prod_{v\in \mathcal{V}} \big( 1 - h_{u}t_{u}^{E^{*}_{v}} \big)^{\delta_{v,\mathcal{V}}-2}$ 
simplifies to a fraction of form $P(t_{u}) \big( 1 - h_{u}t_{u}^{E^{*}_{u}} \big)^{-2}$, where $P(t_{u})$ is a Laurent polynomial defined in (\ref{Eq-58}). 
Thus, $\mathcal{L}^{u}_{H}$ is the quasipolynomial  associated with the coefficient function of the Taylor expansion of 
$P(t_{u}) t_{u} \big( 1 - h_{u}t_{u}^{E^{*}_{u}} \big)^{-2} \big( 1-t_{u} \big)^{-1}$, hence $\mathcal{L}^{u}_{H}$ has degree two by Remark \ref{Rk-3}. 
Moreover, by Theorem \ref{Thm-3} we get that $\mathcal{L}^{\mathcal{T}}_{H}$ is also a quasipolynomial of degree two. 
\end{remark}

\subsection{Recursion for the periodic constants}\label{ss:pcrec}
Let $x\in L'$ such that $[x]=h$ for a fixed $h\in H$. Recall that we have defined two distinguished representatives $r_{h}$ and $s_{h}$ in Section \ref{ss:distreps}.
\subsubsection{$r_{h}$-normalization}
We represent $x=\overline{x}+r_h$ for some $\overline{x}\in L$. 
Then relations (\ref{Eq-39}) and Theorem \ref{Thm-3} imply that
$$
\mathcal{L}^{\mathcal{T}}_h(\overline x)=\mathcal{L}^{\mathcal{T}}_h(x)=
\mathcal{L}^{\mathcal{T}\setminus u}_{[\pi^{(u)}(\overline x+r_h)]}(\pi^{(u)}(\overline x+r_h))+\mathcal{L}^u_h((\overline x+r_h)_u). 
$$
One can happen that $[\pi^{(u)}(\overline x+r_h)]$ varies in $L'_{\mathcal{T}\setminus u} /L_{\mathcal{T}\setminus u}$, hence 
the formula chooses different quasipolynomials on $\mathcal{T}\setminus u$. However, from periodic constant point of view 
it is enough to look at a sparse sublattice $\overline L$ of $L$ in such a way that $[\pi^{(u)}(\overline x+r_h)]$ is constant 
and equals with $[\pi^{(u)} (r_h)]$.
Therefore, by (\ref{eq:PCDEF}) we get
$$
\mathrm{pc}_{h}(Z^{\mathcal{T}})  
= 
\mathcal{L}^{\mathcal{T}}_{h}(0)  =  \mathcal{L}^{\mathcal{T}\setminus u}_{[\pi^{(u)}(r_{h})]}(\pi^{(u)}(r_{h})) + 
\mathrm{pc}(Z_{h}^u).
$$
The problem is that in general $\widetilde{r}_{h}:= \pi^{(u)}(r_{h})$ is not in $\sum_{v\in \mathcal{V}'}[0,1)E_{v}$, hence 
$\mathcal{L}^{\mathcal{T}\setminus u}_{[\widetilde{r}_{h}]}(\widetilde{r}_{h})$ is not equal with 
$\mathrm{pc}_{[\widetilde{r}_{h}]}(Z^{\mathcal{T}\setminus u})$. 
This behaviour of the recurrence on the periodic constant level is in accordance with the  Braun--N\'emethi formula (Theorem \ref{thm:BN}), 
since for $\widetilde{\sigma} = r_{h} * \sigma_{can}$ we have $\widetilde{\sigma}_{u} = \widetilde{r}_{h} * \sigma_{can,u}$.
In special cases (e.g. $h=0$ or $r_{h} = s_{h}$) we get purely a recursion of periodic constants.

\subsubsection{$s_{h}$-normalization}
In this case, we write $x=\overline{x} + s_{h}$ for some $\overline{x}\in L$. One can also modify the definition of the periodic constant associated with $Z$ and introduce {\em $s_{h}$-normalized periodic constant} by 
$$
\overline{\mathrm{pc}}_{h}(Z):=\mathcal{L}_{h}(s_{h}),
$$
where $\mathcal{L}_{h}$ is the quasipolynomial on $h+L$ associated with $Z$. Then by the same argument as in the above section we get
\begin{equation}\label{Eq-43}
\mathcal{L}_{h}^{\mathcal{T}}(s_{h}) = \mathcal{L}_{[\pi^{(u)}(s_{h})]}^{\mathcal{T}\setminus u}(\pi^{(u)}(s_{h})) + \mathcal{L}_{h}^{u}((s_{h})_{u}).
\end{equation}
However, the next lemma shows that $s_{h}$  is projected under $\pi^{(u)}$ into a representative of the same type.

\begin{lemma}
$\pi^{(u)}(s_{h}) = s_{[\pi^{(u)}(s_{h})]}$ in $L'_{\mathcal{T}\setminus u}$.
\end{lemma}

\begin{proof}
 Denote $[\pi^{(u)}(s_{h})]$ by $\overline{h}$. 
 By the definition of $\pi^{(u)}$ one has $\pi^{(u)}(s_{h})\in \mathcal{S}'_{\mathcal{T}\setminus u}$ the Lipman cone in $L'_{\mathcal{T}\setminus u}$, therefore 
 the unique representative associated with $\overline{h}$ can be written as $s_{\overline{h}} = \pi^{(u)}(s_{h})-l$, with $l\in L_{\mathcal{T}\setminus u}\subset L$ and $l\geq0$. We set $s:=s_{h}-l$ and we show that $s\in \mathcal{S}'$, i.e. $(s,E_{v})\leq0$ for all $v$. This would imply that $l=0$ by the minimality of $s_{h}$, hence $\pi^{(u)}(s_{h}) = s_{\overline{h}}$. 

Notice that $s_{\overline{h}}\in \mathcal{S}'_{\mathcal{T}\setminus u}$ is equivalent with $(\pi^{(u)}(s_{h})-l,E_{v})\leq 0$ for all $v\neq u$. Moreover $(s_{h},E_{v})= (\pi^{(u)}(s_{h}),E_{v})$ for all $v\neq u$, hence $(s,E_{v})\leq0$ for all $v\neq u$. On the other hand, $(l,E_{u})\geq0$ since the $E_{u}$-coefficient of $l$ is $0$. Hence 
$
(s,E_{u}) = (s_{h},E_{u}) - (l,E_{u})\leq 0
$
by $s_{h}\in \mathcal{S}'$. 
\end{proof}

Therefore, (\ref{Eq-43}) can be interpreted as a recursion of $s_{h}$-normalized periodic constants
$$
\overline{\mathrm{pc}}_{h}(Z^{\mathcal{T}}) 
= 
\overline{\mathrm{pc}}_{[\pi^{(u)}(s_{h})]}(Z^{\mathcal{T}\setminus u}) 
+ 
\mathrm{pc}(t^{-(s_{h})_{u}}Z_{h}^u),
$$
where $Z_{h}^u$ is the one-variable series $Z_{h}^{\mathcal{T}}|_{t_v=1,v\neq u}$.


\end{document}